\documentclass[12pt, reqno]{amsart}
\usepackage{amsmath, amsthm, amscd, amsfonts, amssymb, graphicx, color, mathtools}
\usepackage[bookmarksnumbered, colorlinks, plainpages]{hyperref}
\usepackage{hyperref}
\hypersetup{
    colorlinks=true, 
    linkcolor=red,   
}

\usepackage{enumitem}
\usepackage{textcomp}
\usepackage{tikz}
\usepackage{geometry}
\newtheorem{theorem}{Theorem}[section]
\newtheorem{lemma}[theorem]{Lemma}
\newtheorem{proposition}[theorem]{Proposition}
\newtheorem{corollary}[theorem]{Corollary}
\theoremstyle{definition}
\newtheorem{definition}[theorem]{Definition}
\newtheorem{notation}[theorem]{Notation}
\newtheorem{observation}[theorem]{Observation}
\newtheorem{example}[theorem]{Example}

\DeclareMathOperator{\val}{val}
\DeclareMathOperator{\ev}{ev}
\DeclareMathOperator{\ft}{ft}
\DeclareMathOperator{\mult}{mult}
\DeclareMathOperator{\trop}{trop}
\theoremstyle{remark}
\newtheorem{remark}[theorem]{Remark}
\numberwithin{equation}{section}
\date{\today}

\usepackage{nomencl}
\makenomenclature
\begin{document}
\sloppy
\setcounter{page}{1}


\centerline{}

\centerline{}

\title[General Kontsevich-style formula for Hirzebruch Surfaces]{General Kontsevich-style formula for Hirzebruch Surfaces\\
}

\author[Parisa Ebrahimian]{Parisa Ebrahimian}
\maketitle
\begin{center}
\address{ \footnotesize {Department of Mathematics, University of T\"{u}bingen} \\
\email{\textcolor[rgb]{0.00,0.00,0.84}{
parisa.ebrahimian1@gmail.com
}}}

\end{center}
\date{Received: xxxxxx; Revised: yyyyyy; Accepted: zzzzzz.
\newline \indent $^{*}$ Corresponding author}

\begin{abstract}

Tyomkin's correspondence theorem states the equality of counts of rational curves of fixed homology class in a toric surface satisfying point and cross-ratio conditions with their tropical counterparts \cite{Tyo}. Such correspondence theorems allow us to derive non-tropical results from tropical ones; for example, Mikhalkin's correspondence theorem is used in the tropical proof of the famous Kontsevich formula for counts of plane rational curves of degree $d$ satisfying point conditions \cite{Mik}. The latter has been generalized to counts of curves in the Hirzebruch surface $\mathbb{F}_{2}$ satisfying point conditions in \cite{FM}. A generalization of these enumerative problems for curves in $\mathbb{P}^2$, which allows satisfying multiple cross-ratio conditions, has been provided in \cite{GM}. 
In this paper, we present a Kontsevich-style formula for the Hirzebruch surface $\mathbb{F}_{r}$, for $r \in \mathbb{N}$, which counts the number of rational tropical curves of fixed homology class satisfying points and multiple cross-ratio conditions via tropical methods. Moreover, the cross-ratio conditions we impose on the curves allow more freedom than the ones defined in \cite{GM}.

\end{abstract}

\tableofcontents
\section{Introduction}

\noindent Enumerative geometry is one of the oldest areas in mathematics, focusing on determining the number of geometric objects that meet specific conditions. A classical example is Apollonius’ problem, originating in Ancient Greece, which asks for a circle tangent to three given circles in the plane.
 Apollonius of Perga (c. 262 BC - c. 190 BC) posed and solved this famous problem in his work. Subsequently, a wide range of geometric and algebraic methods were developed to solve Apollonius' problem and its generalizations.

Although these counting problems are old and arise naturally, many are very hard to answer and require the development of new methodologies. For instance, one of the complicated problems in enumerative geometry is computing Gromov-Witten (GW) invariants. GW invariants are rational numbers that, in certain situations, count curves meeting prescribed conditions on a given manifold.
 These invariants have been used to distinguish (symplectic) manifolds that were previously indistinguishable. Moreover, they play a crucial role in string theory. From another viewpoint, some GW invariants are a natural generalization of the following basic example of enumerative geometry: how many lines are there through two distinct points?
Indeed, we can ask about the number of curves of fixed degree satisfying certain incidence conditions.

Maxim Lvovich Kontsevich found a recursive formula for counting the number of complex plane rational curves of degree $d$ that pass through $ {3d-1} $ fixed general points \cite{KM}. These numbers, denoted by $ N_{d} $, are GW invariants of $ \mathbb{P}^2 $. 
The discovery of this formula for generic $d$ was a breakthrough since the number $ N_3 = 12 $ was found in $ 1848 $ (Steiner), and $ N_4=620 $ was found in $ 1873 $ (Zeuthen), and the problem was open for an arbitrary $ d > 4 $. The famous Kontsevich's formula is:

\begin{equation}
    N_{d} = \sum_{\substack{d_{1} + d_{2} = d \\ d_{1} , d_{2} > 0}} \Big( d_{1}^{2} d_{2}^{2}\binom{3d - 4}{3d_{1} - 2} - d_{1}^{3} d_{2}\binom{3d - 4}{3d_{1} - 1} \Big) N_{d_{1}}N_{d_{2}}.
\end{equation}

An important step towards deducing such a formula is to study curves that satisfy certain point conditions and one cross-ratio condition. 
A cross-ratio condition is a number associated to 4 marked points on a rational curve. Two 4-tuples of points map to each other via an isomorphism of $ \mathbb{P}^1 $ if and only if their cross-ratios coincide. 
One of the principal challenges is to determine the number of rational curves in toric surfaces satisfying point and cross-ratio conditions.

Tropical geometry is a branch of mathematics that has gained significant attention in this field in the past two decades. It is a development in algebraic geometry that aims to simplify algebro-geometric problems into purely combinatorial ones, and it is used as a modern tool that can solve enumerative problems effectively. Tropical geometry can be viewed as a degenerate version of algebraic geometry, where the algebraic operations of addition and multiplication are replaced with the tropical operations of minimum and addition, respectively. A tropical plane curve corresponds to a balanced piece-wise linear graph in $ \mathbb{R}^2 $. The degree of a tropical curve is defined to be a set of direction vectors of outgoing ends of the tropical curve, and we denote it by $ \Delta $.
Therefore, tropicalizing enumerative problems converts a problem of counting the number of algebraic curves on a toric surface to a problem of counting the number of specific graphs in $ \mathbb{R}^2 $.

In order to translate an enumerative problem into the tropical world, we need a (suitable) correspondence theorem. Grigory Mikhalkin established a correspondence theorem in his paper \cite{Mik}, demonstrating the equality of  $ N_{d} $ and $ N^{\trop}_{d} $. Here, $ N^{\trop}_{d} $ denotes the analogous counting curve problem in the tropical setting: counting the number of rational tropical plane curves of degree $ d $ through $ 3d-1 $ general given points. The equality $ N_{d} = N^{\trop}_{d} $ makes it possible to count tropical curves instead of algebraic curves. A notable achievement in this journey is the tropical proof of Kontsevich’s formula, which was accomplished approximately a decade ago by Hannah Markwig and Andreas Gathmann \cite{GM}.

Ilya Tyomkin proved an algebraic-tropical correspondence for rational tropical curves in toric surfaces satisfying point and cross-ratio conditions \cite{Tyo}. To make use of this correspondence theorem for algebraic geometry, we need to count tropical curves that satisfy point and cross-ratio conditions. Christoph Goldner studied this problem for curves in toric surfaces satisfying certain point and cross-ratio constraints, by generalizing the lattice path algorithm. The lattice path algorithm is a well-known combinatorial method in tropical geometry for enumerating all rational tropical curves satisfying point conditions and degenerate cross-ratio constraints \cite{G18}.

In the tropical setting, a cross-ratio condition corresponds to the length of the path enclosed by four marked ends; see Figure \ref{cr}. A cross-ratio is called \emph{degenerate} if this path has length zero.
The degenerate cross-ratio conditions cause vertices of valence higher than 3 in the tropical curve.  For a more detailed definition of the tropical cross-ratio condition, see \ref{CR}, and to see an example of a degenerate cross-ratio, see Example \ref{Gold}.

 \begin{figure}[h!]
		
		\includegraphics[scale=0.9]{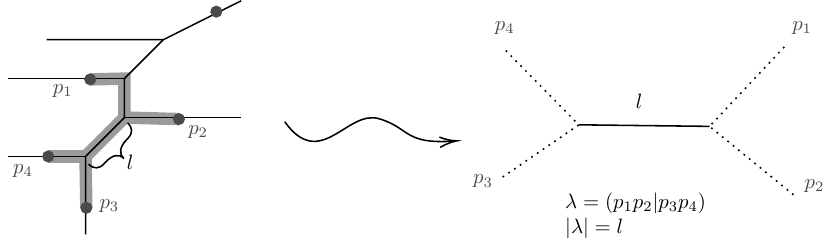}
		\centering
		\caption{A tropical cross-ratio is an unordered pair of pairs of unordered marked ends like $ \lambda = (p_{1} p_{2}|p_{3} p_{4}) $ and a length $|\lambda|$. We call a cross-ratio nondegenerate if its length \( |\lambda| \) is nonzero. If \( |\lambda| = 0 \), we call it a degenerate cross-ratio and denote it by \( \lambda = \{p_{1}, p_{2}, p_{3}, p_{4}\} \).
        This figure also illustrates a forgetful map $\ft_{\lambda}$, which maps a 5-marked tropical stable map of degree $\Delta_{\mathbb{F}_{2}}(1,1)$ to a point in $\mathcal{M}_{0,4}$ (for more details see Definition \ref{forgetfulmap}).}
        
		\label{cr}
	\end{figure}

Christoph Goldner also introduced so-called cross-ratio floor diagrams and showed that they allow us to determine the number of rational space tropical curves that satisfy general positioned point and cross-ratio conditions \cite{G20}. In another paper, Goldner provided a general Kontsevich’s formula for $\mathbb{P}^2$, which holds for more than one cross-ratio condition. His generalized version of Kontsevich’s formula recursively computes the number of rational plane curves of degree \(d\) that satisfy point, curve, and cross-ratio conditions in general position \cite{G}.

Other Kontsevich-style formulas exist for counting curves in other surfaces or under specific conditions, some of which have been discovered using tropical geometric methods. For instance, in \cite{FM}, an elegant formula is established for counting the number of irreducible rational (genus zero) plane tropical curves of a fixed degree in \( \mathbb{F}_{2} \) that satisfy point conditions in general position.
 A Hirzebruch surface $ \mathbb{F}_{r} $ is a ruled surface over the projective line.  A tropical curve in $ \mathbb{F}_{r} $ of degree $ \Delta_{\mathbb{F}_{r}}(a,b) $ is a graph dual to a subdivision of the polygon in Figure \ref{F}.


\begin{figure}[h!]
	
	\includegraphics[scale=0.8]{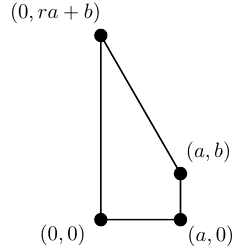}
	\centering
	\caption{A tropical curve in $\mathbb{F}_{r}$ of degree $\Delta_{\mathbb{F}_{r}}(a,b)$ is dual to a regular subdivision of this polygon.}
	\label{F}
\end{figure} 

We can obtain many interesting results by imposing multiple cross-ratio conditions on a curve. For example, determining the number of configurations of $n$ points on $\mathbb{P}^1$ with $n-3$ prescribed cross-ratios is known as a cross-ratio degree problem. 
Rob Silversmith established a closed formula
for a class of cross-ratio degrees indexed by triangulations of an $n$-gon in \cite{Sil}. These degrees are connected to the
geometry of the real locus of $\mathcal{M}_{0,n}$ and to cluster algebras. In the same theme, Alessio Cela and Aitor Iribar L\'{o}pez in \cite{CI} introduced a formula to count the number of tropical stable maps in Hirzebruch surfaces satisfying $n$ point conditions and exactly $n-3$ cross-ratio conditions. 

In this paper, we explore a more flexible framework where it is not necessary to impose all possible cross-ratio conditions. Instead, any subset of cross-ratio conditions can be applied, as considered for curves in $\mathbb{P}^2$ in \cite{G}. There is currently no formula for enumerating rational curves in $ \mathbb{F}_r $ that satisfy general position constraints for points, curves, and a chosen subset of cross-ratio conditions. \\

\begin{center}
\fbox{
\parbox{0.95\textwidth}{
\emph{Our main goal is to provide a recursive formula for counting the number of irreducible rational (genus zero) tropical curves of degree $ \Delta_{\mathbb{F}_{r}}(a,b) $ in Hirzebruch surface $ \mathbb{F}_{r} $ that satisfy $n$ point, $k$ curve, and 
$l$ cross-ratio conditions. This number is represented by $ \mathcal{N}^0_{\Delta_{\mathbb{F}_r}(a,b)}(p_{\underline{n}}, L_{\underline{k}}, \lambda_{\underline{l}}) $.}


}
}\\\hfill (\textcolor{red}{$\star$}) 
\label{box:recursive_formula}
\end{center}

\def\refstar{*} 

However, we discovered that constructing a recursion for these counts introduces non-contracted marked ends with weights greater than or equal to 1 in the primitive direction \( (1,0) \). To handle this, we mark all ends in the primitive direction \( (1,0) \) with \( \Tilde{e}_d \), and denote their corresponding weights by \( w_{\underline{d}} \) (this is explained with more details in Example \ref{Gold} and then in Section \ref{Prel}). 
So now we need to address this problem:\\

\begin{center}

\fbox{
\parbox{0.95\textwidth}{\emph{Counting the number of irreducible rational plane tropical curves of degree $ \Delta_{\mathbb{F}_{r}}(a,b, w_{\underline{d}}) $ in Hirzebruch surface $ \mathbb{F}_{r} $ that satisfy $n$ point, $k$ curve, and 
$l$ cross-ratio conditions where we have $d$ marked non-contacted ends of primitive direction $(1,0)$}}
}\\\hfill (\textcolor{red}{$\star \star$}) 
\label{box:recursive_formula1}
\end{center}

\def\refstar{**} 

To solve this new problem (\textcolor{red}{\hyperref[box:recursive_formula]{$\star \star$}}) we degenerate the tropical curves to deduce a recursive formula. In order to do that, it is necessary to (slightly) modify the definition of a tropical cross-ratio condition and address the cross-ratio conditions that offer more flexibility in terms of the chosen marked ends that contribute to a cross-ratio (see the definition of the cross-ratio in \ref{CR}). 
Marking all \( d \) non-contracted ends in the primitive direction \( (1,0) \) implies that the new cross-ratio conditions may now also involve these marked ends.
The reason why we need to solve a more general problem can be seen in the following example:
\begin{example}\label{Gold}
When searching for a tropical curve \(C\), a curve condition is given by the intersection of \(C\) with another tropical curve. The tropical curves used as conditions are of specific degrees, as defined in \ref{Line}, and are referred to as multi-line conditions.
In Figure \ref{gold}, we have a tropical curve of degree $ \Delta_{\mathbb{F}_{2}}(3,1) $ satisfying $11$ point conditions, $1$ multi-line condition $L$, a non-degenarated cross-ratio condition $\lambda'_1 = (t_{\alpha}t_{\beta}|t_{\eta}t_{\gamma})$, and a degenerate cross-ratio $\lambda_2 = \{t_{\eta}, t_{\gamma},  t_{1}, t_{\alpha} \}$. The non-degenerate cross-ratio condition comes with a fixed length $|\lambda'_1|$. Based on Remark \ref{RationalyEquivalent}, the number of such curves doesn't depend on the length of $\lambda'_1$. Therefore, we can assume that $|\lambda'_1|$ grows arbitrarily, and this results in having a 1-dimensional family of curves that only differ in the length $|\lambda'_1|$. 
In order to produce a recursive formula, we need to split the curves. As we can see in this example, it is possible to split such a curve over its movable part, which is shown with black straight lines in Figure \ref{gold}. But  $\lambda_2$ will no longer be satisfied if we do not replace $t_{\alpha}$ with another marked end which lives in the same connected component as $t_{\eta}, t_{\gamma},  t_{1} $. To resolve this problem, we mark the new end in the direction $(1, 0)$  with $\Tilde{e}$, as it is also shown in Figure \ref{gold}, and we replace $\lambda_2$ with ${\lambda}_2 = \{t_{\eta}, t_{\gamma},  t_{1}, \Tilde{e} \}$. Since, after cutting the movable component, we have to deal with this new type of cross-ratio conditions, we need to solve the problem for a more general setting, where we have a marked right end and the cross-ratio conditions can include such a marked end. 

    \begin{figure}[h!]
	
	\includegraphics[scale=0.8]{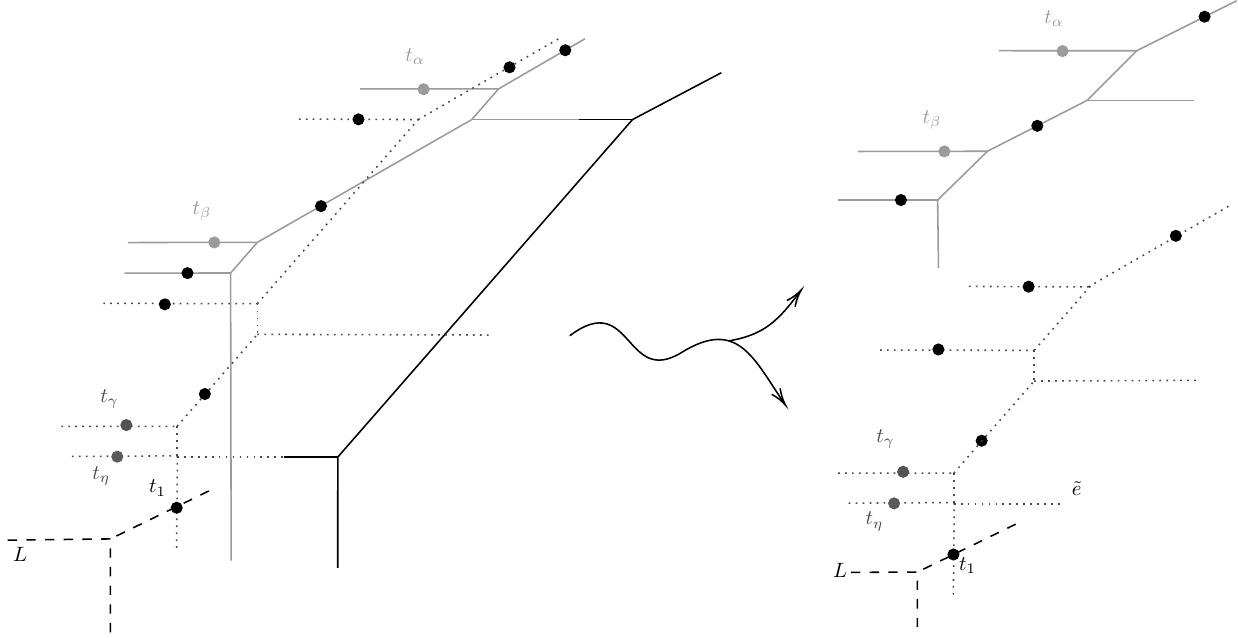}
	\centering
	\caption{In the left picture, we have a tropical curve of degree \( \Delta_{\mathbb{F}_{2}}(3,1) \) satisfying 11 point conditions, 1 multi-line condition, and 2 cross-ratio conditions. As described in Example \ref{Gold}, we split this tropical curve over its movable string, which is drawn with black straight lines. After splitting the curve, we obtain two tropical curves of degrees \( \Delta_{\mathbb{F}_{2}}(1,1) \) and \( \Delta_{\mathbb{F}_{2}}(1,2) \) in the right-hand pictures.
 The first curve is drawn with straight grey lines, and the second one is drawn with dashed black lines. }
	\label{gold}
\end{figure} 
\end{example}
Therefore, we will need to develop more general methods for counting curves.
In Proposition \ref{*} and Section \ref{Split curves}, we present and thoroughly demonstrate all possible ways to split a tropical curve in $ \mathbb{F}_{r} $ satisfying one very long, non-degenerate cross-ratio condition.

To degenerate the tropical curves in $ \mathbb{F}_{r} $, we face new cases compared to the problem of degenerating the tropical curves in $ \mathbb{P}^2 $. Indeed, in \cite[Proposition  29]{G}, tropical curves in $ \mathbb{P}^2 $ degenerate to reducible tropical curves via a contracted bounded edge. However, we prove in Proposition \ref{*} that the tropical curves in $ \mathbb{F}_{r} $ degenerate either via contracted bounded edges or via cutting movable branches. Since we can impose more generalized cross-ratio conditions on the curves, the potential movable components in our problem are more general than the one in \cite[Lemma 2.10]{FM}. 

More precisely, we wanted to find a recursive formula for counting the number of $m$-marked rational tropical stable maps (alternatively: tropical curves) in the moduli space of $m$-marked tropical stable maps of degree $ \Delta_{\mathbb{F}_{r}}(a,b) $
that satisfy $ n $ point conditions, $ k $ multi-line conditions, and $ l $ cross-ratio conditions, where only marked points can contribute to the cross-ratios.
But as we saw in Example \ref{Gold}, we need to modify the definition of the cross-ratio condition to be more flexible. Indeed, we mark all $d$ non-contracted ends in the direction of $(1,0)$, and the new cross-ratio conditions can include these marked ends as well (for the precise definitions, see \ref{tropical} and \ref{CR}). Cutting the edges leading to a movable component results in ends with weight greater than 1 in the \((1,0)\) direction.
 This means that a tropical stable map of degree $ \Delta_{\mathbb{F}_{r}}(a,b) $ can have $d$ ends of weights $w_{\underline{d}}$ in the primitive direction $(1,0)$ (of course $b= \sum_{i=1}^d w_i$), so we modify the notation of the degree of such a tropical stable maps to $\Delta_{\mathbb{F}_{r}}(a,b, w_{\underline{d}})$ (Definition \ref{Degree} and in Section \ref{Split curves}). Therefore, for a given degree $\Delta_{\mathbb{F}_{r}}(a,b, w_{\underline{d}})$, we need to count the number of $(m+d)$-marked rational tropical stable maps satisfying $ n $ point conditions, $ k $ multi-line conditions, and $ l $ cross-ratio conditions where we have $d$ marked non-contracted ends of weights $ w_{\underline{d}} $ and this is a more general problem than the one we aimed to solve.
Since we labeled all of the ends in the $(1,0)$ direction (see Definition \ref{tropical}), when $b=d$ and $w_i =1$ for each $1 \leq i \leq d$, the answer to the main enumerative problem (\textcolor{red}{\hyperref[box:recursive_formula]{$\star$}}) is as follows:

\begin{center}
\fbox{
\parbox{0.49\textwidth}{\emph{$ \mathcal{N}^{0}_{\Delta_{\mathbb{F}_{r}}(a,b)} (p_{\underline{n}}, L_{\underline{k}}, {\lambda}_{\underline{l}})= \frac{1}{b!}\mathcal{N}^{0}_{\Delta_{\mathbb{F}_{r}}(a,b, w_{\underline{b}})} (p_{\underline{n}}, L_{\underline{k}}, {\lambda}_{\underline{l}})$.}}
}
\end{center}

Theorem \ref{Main0} presents a recursive formula for enumerating tropical stable maps contributing to $\mathcal{N}^{0}_{\Delta_{\mathbb{F}_{r}}(a,b, w_{\underline{d}})} (p_{\underline{n}}, L_{\underline{k}}, {\lambda}_{\underline{l}})$, with $l\geq1$. If we set $l = 0$, an interesting enumerative problem is to compute $\mathcal{N}^{0}_{\Delta_{\mathbb{F}_{r}}(a,b)} (p_{\underline{n}})$, see Remark \ref{Frn}. Such numbers arise naturally when we expand the recursive formula in Theorem~\ref{Main0}.

In order to solve this problem, the next thing to adjust is the length of the cross-ratios. Let $ {\lambda}_{\underline{u}} = \{ {\lambda}_r \}_{r=1}^u $ and $ {\lambda}'_{\underline{u}'} = \{ {\lambda}_r' \}_{r'=1}^{u'} $ be the sets of all degenerate and non-degenerate cross-ratio conditions respectively, where $ l = u + u' $ and the underlined symbols indicate a set of symbols. Then $ \mathcal{N}^{0}_{\Delta_{\mathbb{F}_{r}}(a,b, w_{\underline{d}})}(p_{\underline{n}}, L_{\underline{k}}, {\lambda}_{\underline{u}}, {\lambda}'_{\underline{u}'}) $ is the number of rational tropical stable maps of degree $ \Delta_{\mathbb{F}_{r}}(a,b, w_{\underline{d}})$ with $d$ marked right ends that satisfy $ n $ point conditions $ p_{\underline{n}} $, $ k $ multi-line conditions $ L_{\underline{k}} $, and $ l $ cross-ratio conditions, counted with multiplicity (for a precise definition see \ref{Nd}). This number is independent of the exact length of the cross-ratio conditions (see Remark \ref{RationalyEquivalent}). Therefore by assuming that all cross-ratio conditions are degenerate, the main problem can be reduced to the problem of counting the number of the tropical stable maps contributing to $ \mathcal{N}^{0}_{\Delta_{\mathbb{F}_{r}}(a,b, w_{\underline{d}})} (p_{\underline{n}}, L_{\underline{k}}, {\lambda}_{\underline{l}}) $.

Consider the last cross-ratio $ \lambda_l $, we can replace it with a non-degenerate cross-ratio condition $\lambda'_l $ that degenerates to $ \lambda_l $. Our approach for finding a recursive formula for counting the number of tropical stable maps contributing to 
$ \mathcal{N}^{0}_{\Delta_{\mathbb{F}_{r}}(a,b, w_{\underline{d}})}(p_{\underline{n}}, L_{\underline{k}}, {\lambda}_{\underline{l}}) $
is to study $\mathcal{N}^{0}_{\Delta_{\mathbb{F}_{r}}(a,b, w_{\underline{d}})}(p_{\underline{n}}, L_{\underline{k}}, {\lambda}_{\underline{l-1}}, {\lambda}'_{l})$ instead, where the length of $ |{\lambda}'_{l}|$ is very long. Forgetting the non-degenerate cross-ratio condition results in having a 1-dimensional family of curves, and we will study this 1-dimensional family in Section \ref{STSM}. In Proposition \ref{*} we prove that there are only two general ways to split tropical stable maps contributing to $\mathcal{N}^{0}_{\Delta_{\mathbb{F}_{r}}(a,b, w_{\underline{d}})}(p_{\underline{n}}, L_{\underline{k}}, {\lambda}_{\underline{l-1}}, {\lambda}'_{l})$. 

The idea of considering a cross-ratio condition with a very long length comes from the proof of Kontsevich's formula in \cite{GM}.
Our approach to study this one-dimensional family offers a new viewpoint for studying the movable components that appear in the tropical stable maps in Hirzebruch
surfaces satisfying one less cross-ratio condition.



\subsection{Structure of the paper}

 In Section \ref{Prel} below, we describe the basic objects of our study and give the background on necessary concepts about tropical stable maps and their moduli spaces, as well as the conditions that we want to impose on the tropical stable maps, such as the definition of the cross-ratio condition \ref{CR}.

Section \ref{STSM} forms the core of our results. In this section, we generalize tropical methods to enable the splitting of tropical stable maps in $\mathbb{F}_{r}$ that satisfy prescribed incidence conditions.
The main goal of this section is to prove Proposition \ref{*} and Lemma \ref{n0}, which explain all possible ways to split the tropical stable maps contributing to $\mathcal{N}^{0}_{\Delta_{\mathbb{F}_{r}}(a,b, w_{\underline{d}})}(p_{\underline{n}}, L_{\underline{k}}, {\lambda}_{\underline{l-1}}, {\lambda}'_{l})$ where $n \geq 1$ or $n=0$. 

In Section \ref{Split curves}, we study how the existence of multiple cross-ratio conditions can affect the way we split the curves. We also study the cut parts and the ways they affect each other.

In Section \ref{MOSC}, we compute the multiplicity of the cut curves in every splitting scenario, and then we find the relation between the multiplicities of the main curve and the split version. 

Finally, in Section \ref{Formula}, we derive a recursive formula for counting the number of irreducible rational tropical stable maps of degree  $ \Delta_{\mathbb{F}_{r}}(a,b, w_{\underline{d}}) $ in the toric surface $ \mathbb{F}_{r} $ that satisfy point, curve (multi-line), and cross-ratio conditions, where we have at least one cross-ratio condition, see Theorem \ref{Main0}. Then in Section \ref{GTKFrn}, we present a recursive formula for computing 
\(\mathcal{N}^{0}_{\Delta_{\mathbb{F}_{r}}(a,b, w_{\underline{d}})}(p_{\underline{n}})\).  Moreover, we expand the recursive formula from Theorem~\ref{Main0} for a fixed degree to illustrate its application, see Example \ref{ExF3}.

In Section~\ref{Test}, we prove that for specific degrees and a fixed number of cross-ratio conditions, our generalized formula recovers known results established in~\cite{FM} and~\cite{G}. \\

\textbf{Acknowledgements.} 
I would like to thank Hannah Markwig for her valuable feedback and insightful discussions throughout every stage of this research. I am grateful to Christoph Goldner, whose work \cite{G} inspired the methods developed here, and who kindly proofread the manuscript and provided detailed feedback. I also thank Firoozeh Dastur for helpful suggestions regarding the structure of the recursive formula, Daniel Funck for proofreading the first two sections of the paper, and Victoria Schleis for valuable conversations on constructing some of the initial examples of tropical curves. This work was supported by the German Academic Exchange Service (DAAD), whose support I gratefully acknowledge.


\section{Preliminaries}\label{Prel}

This section includes a review of some standard definitions from tropical geometry. Moreover, some of the more technical definitions that will be used in the next sections will be introduced or recalled here. One can see \cite{GM, G, FM} for more details in this regard as well.

The purpose of this paper is to provide a class of tropical enumeration invariants for $ \mathbb{F}_{r} $, a special toric variety classified as a Hirzebruch surface. The Hirzebruch surface $\mathbb{F}_{r}$ is defined to be the surface $\mathbb{P}(\mathcal{O}_{\mathbb{P}^1} \oplus \mathcal{O}_{\mathbb{P}^1}(r))$; it is a smooth projective toric surface.
A toric variety $ X $ is an algebraic variety containing a torus $ T = (C^{*})^r $ as a dense open subset, in such a way that the action of the torus on itself extends to the whole $ X $. Toric varieties can be derived from lattice polytopes, and the tropicalization of a curve on a toric variety will be a graph dual to a subdivision of the mentioned lattice polytope. (For more details see \cite[Example 2.3.16]{CLS} and \cite[Proposition 3.1.6]{MS}). 

From another point of view, we can consider the tropical curves as abstract metric graphs mapped to the real plane that have some specific features. 

\begin{notation}\label{Notation}
In this paper, underlined symbols denote sets of indices. For example, \( \underline{n} = \{1, \dots, n\} \) where \( \# \underline{n} = n \).

\end{notation}

\begin{definition}
   Let \( \Gamma \) be a metric tree. We call the unbounded edges of \( \Gamma \) the \emph{ends}.
 An $m$-marked abstract rational tropical curve is a tuple $ (\Gamma, x_{\underline{m}})$, where $ \Gamma $ has exactly $ m $ ends that are labeled with $ x_{1}, \cdots, x_{m} $.
The moduli space of all m-marked abstract rational tropical curves with exactly m unbounded edges will be denoted by $ \mathcal{M}_m $. 
\end{definition}

\begin{example}
   The first nontrivial example of the moduli space of abstract curves is $ \mathcal{M}_{0,4} $. An abstract 4-marked tropical curve is a tree with 4 unbounded ends. Therefore, $ \mathcal{M}_{0,4} $ is simply a rational tropical curve with 3 ends. See Figure \ref{M4}.

 \begin{figure}[h!]
		
		\includegraphics[scale=0.9]{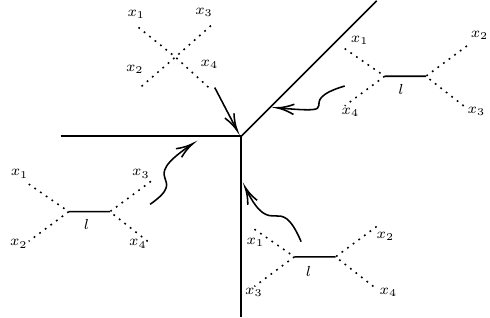}
		\centering
		\caption{The length of the only bounded edge of an abstract tropical curve in $ \mathcal{M}_{0,4} $ determines its exact location on the polyhedral complex $ \mathcal{M}_{0,4} $.}
		\label{M4}
	\end{figure}
   
\end{example}

\begin{definition} \label{tropical}
    An $(m+d)$-marked rational tropical stable map is represented as a tuple $ C = (\Gamma, x_{\underline{m}}, \Tilde{e}_{\underline{d}}, h) $, where $ \Gamma $ denotes an abstract tropical curve, and $ x_{\underline{m}} $ and $\Tilde{e}_{\underline{d}}$ are distinct labeled ends of $ \Gamma $. Additionally, $ h: \Gamma \rightarrow \mathbb{R}^2 $ is a continuous map and $C$ satisfies the following conditions:
\begin{enumerate}
    \item On each edge of $ \Gamma $, $ h $ is an integer affine linear map with rational slope. That is, for every bounded edge $ e $ of $ \Gamma $ of length $ l(e) \in [0, +\infty] $, the map $ h|_e $ is of the form $ h(t) = a + t \cdot v(e) $ and therefore it stretches the edge $ e $ with length $ l(e) $ to a line segment of length $ l(e) \cdot ||v(e)|| $ (where $ ||.|| $ denotes the usual Euclidean length). The vector $ v(e) $ is called the direction vector of the edge $ e $ at $ 0 \in [0, l(e)]$. The direction vector can be considered as $ v(e):= w(e) \cdot u(e)$, where $ w(e) \in \mathbb{N} $ is called the weight of the edge, and $ u(e) $ is the primitive integer vector in the direction of the edge.  

    \item The balancing condition holds on each vertex of $ \Gamma $. That is, the sum of all direction vectors of edges adjacent to a vertex has to be zero.

    \item Each marked end $ x_i $ maps to a point in $ \mathbb{R}^2 $.

    \item All non-contracted ends of the primitive direction vector $(1,0)$ are labeled with $	\Tilde{e}_{\underline{d}}$.
\end{enumerate}
\end{definition}

Notice that in Definition \ref{tropical}, we have two different types of marked ends: the contracted ends, which we will refer to as marked points, and the non-contracted marked ends, which we refer to as marked right ends.

\begin{definition} \label{Degree}
   The {degree} of an \((m+d)\)-marked tropical stable map is defined as the multiset \( \Delta \) of directions of its non-contracted unbounded edges. Let \( \Delta \) consist of:
\begin{itemize}
    \item \( a \) edges in the direction \( (0,-1) \),
    \item \( a \) edges in the direction \( (r,1) \),
    \item \( d \) edges in the directions \( (w_i,0) \) for \( 1 \leq i \leq d \),
    \item \( ra+\sum_{i=1}^{d} w_i \) edges in the direction \( (-1,0) \).
\end{itemize}

Defining \( b = \sum_{i=1}^{d} w_i \), we introduce the shorthand notation 
$\Delta_{\mathbb{F}_{r}}(a, b, w_{\underline{d}})$ for the degree of \((m+d)\)-marked tropical stable maps in \( \mathbb{F}_{r} \).
 
\end{definition}

 In Definition \ref{tropical}, we have slightly changed the definition of the m-marked tropical stable map introduced in \cite{GM} by marking all of the ends with the primitive direction vector $(1,0)$. The reason we might need some marked non-contracted ends in the $(1,0)$ direction is that, as explained in Example \ref{Gold} and Proposition \ref{*}, we will cut a movable component of the tropical stable map to degenerate the curves and this might lead to losing one of the marked ends contributing to a cross-ratio condition.  Cutting the edges leading to a movable component results in having weight higher than 1 for ends with primitive direction $(1,0)$, so these ends can have weights $w_{\underline{d}}$ greater than one (moreover, the ends of the cut movable component with primitive direction $(-1,0)$ can have weights $w_{\underline{d}}$ greater than one as well, see Definition \ref{ftildas}). We denote these non-contracted ends $\Tilde{e}_i$ with weight $w_i$ for $1 \leq i \leq d$. If $b=d$ and $w_i =1$ for each $1 \leq i \leq d$, we denote the degree of an $m$-marked tropical stable map with $ \Delta_{\mathbb{F}_{r}}(a,b) $ (this is the same notation used in \cite{FM, G}).

Based on Propositions 3.1.6 and 3.3.10 in \cite{MS} there is a correspondence between Laurent polynomials and tropical curves. Therefore, the image of the tropical stable map that we defined in \ref{tropical} is dual to a regular subdivision of its corresponding Newton polygon, which is defined in Definition \ref{Newton}.
\begin{definition}
    \label{Newton}
    Corresponding to a Laurent polynomial, we can define a Newton polygon, which is constructed by plotting the exponents of its monomials and taking the convex hull of these points in the plane.

\end{definition}

 For a tropical stable map $C = (\Gamma, x_{\underline{m}}, \Tilde{e}_{\underline{d}}, h) $, the image $h(\Gamma)$ as a tropical curve on $ \mathbb{F}_{r} $ of degree $ \Delta_{\mathbb{F}_{r}}(a,b, w_{\underline{d}}) $ is a graph dual to a subdivision of a right trapezoid with vertices $ (0,0) $, $ (a,0) $, $ (a, b) $ and $ (0, ra + b) $, such that it allows $d$ ends in the primitive direction $(1,0)$ of weights $w_{\underline{d}}$. 

\begin{definition}[{\cite[Combinatorial type]{GM}}]
    \label{CombinatoriyalType}
    The \emph{combinatorial type} of an abstract marked tropical curve 
$(\Gamma, x_{\underline{m}}, \Tilde{e}_{\underline{d}})$ is defined as its homeomorphism class relative to the markings $x_1, \dots, x_m$ and ends $\Tilde{e}_1, \dots, \Tilde{e}_d$. In other words, it is the data of $(\Gamma, x_{\underline{m}}, \Tilde{e}_{\underline{d}})$ modulo homeomorphisms of $\Gamma$ that fix each marked point $x_i$ and each end $\Tilde{e}_j$, where $i \in \underline{m}$ and $j \in \underline{d}$.
 The combinatorial type of an $(m+d)$-marked tropical stable map $(\Gamma, x_{\underline{m}}, \Tilde{e}_{\underline{d}}, h)$ is the data of the combinatorial type of the abstract marked tropical curve $(\Gamma, x_{\underline{m}}, \Tilde{e}_{\underline{d}})$ together with the direction vectors of all edges.
    

\end{definition}

Moduli spaces and tropical intersection theory are tools that have been used in the existing results in tropical enumerative geometry. 
The moduli space of $(m+d)$-marked rational tropical stable maps in $\mathbb{F}_{r} $ is denoted by $ \mathcal{M}^{\trop}_{0,m+d}( \mathbb{R}^{2}, \Delta_{\mathbb{F}_{r}}(a,b, w_{\underline{d}}) ) $ (see \cite{GKM}). A point in this space is a tuple $ C = (\Gamma, x_{\underline{m}}, \Tilde{e}_{\underline{d}}, h) $.  


\begin{definition}
We call the following map the \(i\)-th evaluation map:

\begin{center}
    $ \ev_i: \mathcal{M}^{\trop}_{0,m+d}( \mathbb{R}^{2}, \Delta_{\mathbb{F}_{r}}(a,b, w_{\underline{d}}) ) \to \mathbb{R}^2 $\\
    $ (\Gamma, x_{\underline{m}}, \Tilde{e}_{\underline{d}}, h) \mapsto h(x_i). $
\end{center}

The i-th evaluation map is a morphism of polyhedral complexes that returns the exact locus of the i-th marked point on a tropical stable map (see \cite{GKM}). 
\end{definition}

\begin{definition}
    \label{fixedends}
    Let $e$ be an end of a tropical stable map  $C = (\Gamma, x_{\underline{m}}, \Tilde{e}_{\underline{d}}, h)$
      in \( \mathcal{M}_{0,m+d}(\mathbb{R}^{2}, \Delta_{\mathbb{F}_{r}}(a,b, w_{\underline{d}})) \) with primitive direction \( (-1,0) \). The following evaluation map returns the $y$-coordinate of the image of \( e \) in the plane:
    
    \[
    (\ev_{e})_y: \mathcal{M}^{\trop}_{0,m+d}( \mathbb{R}^{2}, \Delta_{\mathbb{F}_{r}}(a,b, w_{\underline{d}}) ) \to \mathbb{R},
    \]
    \[
    (\Gamma, x_{\underline{m}}, \Tilde{e}_{\underline{d}}, h) \mapsto \text{$y$-coordinate of } h(e).
    \]

    For a real number $g$, we define the \emph{fixed end condition} as the constraint that an end \( e \) has a prescribed \( y \)-coordinate \( g \). This is imposed by requiring the pullback condition $(\ev_{e})_y^*(g)$.
    
\end{definition}

This fixed-end condition will be used later in Section \ref{MOSC}.

\begin{definition}[{\cite[Forgetful map]{GM}}]\label{forgetfulmap}
    Let $C = (\Gamma, x_{\underline{m}}, \Tilde{e}_{\underline{d}}, h) \in \mathcal{M}_{0,m+d}(\mathbb{R}^{2}, \Delta_{\mathbb{F}_{r}}(a,b, w_{\underline{d}}))$. We can forget the map $h$ and obtain an abstract tropical curve, but sometimes we are only interested in the minimal connected subgraph of $\Gamma$ that only contains 4 given marked ends. Let $\lambda$ be a set of four marked ends of a tropical stable map. The forgetful map $\ft_\lambda$ is a map defined on $\mathcal{M}^{\trop}_{0,m+d}(\mathbb{R}^2, \Delta_{\mathbb{F}_{r}}(a,b, w_{\underline{d}}))$ and its target is $\mathcal{M}_{0,4}$: given four distinct marked ends on a tropical stable map $C = (\Gamma, x_{\underline{m}}, \Tilde{e}_{\underline{d}}, h) $, we define $\ft_{\lambda}(C)$ as the minimal connected sub-graph of $\Gamma$ that only contains the four marked ends, which can be either contracted or non-contracted marked ends of $\Gamma$. We straighten 2-valent vertices in this subgraph by replacing the two adjacent edges and the vertex with a single edge whose length is the sum of the lengths of the original edges; thus, we obtain an element of $\mathcal{M}_{0,4}$. (For an example, see Figure \ref{cr}).    
    
\end{definition}


In algebraic geometry, a cross-ratio is a number associated with a list of four points on a (projective) line.
We define the tropical analog of the classical cross-ratio in the following way:

\begin{definition}\label{CR}
 Let $ t_{\alpha}, t_{\beta}, t_{\gamma}, t_{\eta}  $ be four marked ends on a tropical stable map, then:
\begin{itemize}
    \item \textbf{Non-degenerate cross ratio:} An unordered pair of pairs of unordered marked ends of a tropical stable map like $ {\lambda}' = (t_{\alpha} t_{\beta}| t_{\gamma} t_{\eta}) $, together with a length $ |{\lambda}'| \in \mathbb{R} \geq 0 $ is a tropical non-degenerate cross-ratio. The length $ |{\lambda}'| $ can be obtained from $ \ft_{{\lambda}'}(C) \in \mathcal{M}_{0, 4} $, where the map $\ft_{{\lambda}'}$ represents the forgetful map and can be defined using the four marked ends that define ${\lambda}'$.

    \item \textbf{Degenerate cross ratio:} A set of four marked ends of a tropical stable map like $ {\lambda} = \{ t_{\alpha}, t_{\beta}, t_{\gamma}, t_{\eta} \} $ is a tropical degenerate cross-ratio, if the length $ |{\lambda}| = \ft_{{\lambda}}(C) $ is equal to zero. 
\end{itemize}
If all of the marked ends contributing to a cross-ratio are marked points (i.e. \( \alpha, \beta, \gamma, \eta \in \underline{m} \)), then we have the cross-ratio condition as defined in \cite{GM, G, FM}. However, in this paper, non-contracted marked ends can also contribute to cross-ratios.

\end{definition}

\begin{definition}\label{Line}
A multi-line $L$ in $ \mathbb{F}_{r} $ is a rational tropical stable map with one vertex and three ends in directions $ (-1, 0) $, $ (0, -1) $ and $ (r,1) $, where the ends in the directions of $ (0, -1) $ and $ (r,1) $ have weight $ s $ and the end in the direction of $ (-1, 0) $ has weight $ sr $.\\
A tropical stable map $ C =(\Gamma, x_{\underline{m}}, \Tilde{e}_{\underline{d}}, h) $ satisfies a tropical multi-line condition $ L $, if one of the marked points $ x_i $ with label in $ \underline{k} $ belongs to the set of intersection points $ h(x_i) \in L\cap h(\Gamma) $. 
 
\end{definition}

 \begin{figure}[h!]
		
		\includegraphics[scale=0.9]{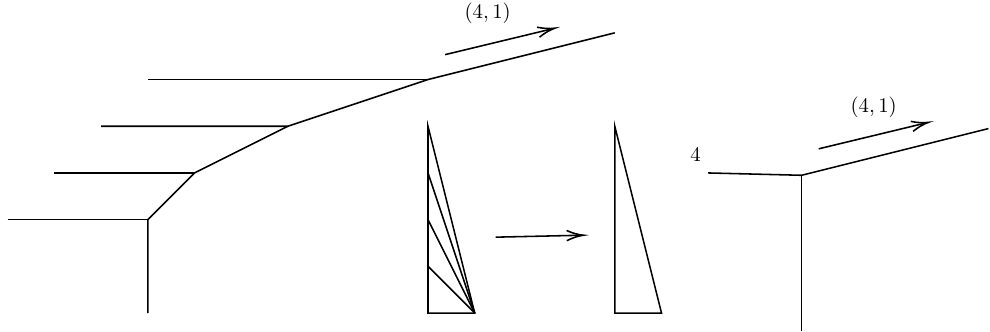}
		\centering
		\caption{The right-hand picture is a tropical multi-line in $\mathbb{F}_4$, as defined in Definition \ref{Line}. The left picture is a tropical stable map of degree $ \Delta_{\mathbb{F}_{4}}(1, 0) $. The dual polygon to each tropical curve is also drawn next to it.}
		\label{line}
	\end{figure}


 \begin{remark}  \label{RationalyEquivalent}
 Let $C = (\Gamma, x_{\underline{m}}, \Tilde{e}_{\underline{d}}, h) \in \mathcal{M}_{0,m+d}(\mathbb{R}^{2}, \Delta_{\mathbb{F}_{r}}(a,b, w_{\underline{d}}))$. The image of the tropical stable map $C$ in the plane $ h(\Gamma) $ is rationally equivalent to the recession fan derived from $ h(\Gamma) $ (\cite[Recession fan]{AHR}). Therefore, both tropical stable maps in Figure \ref{line} are rationally equivalent. Hence, by applying \cite[Rational equivalence]{AR}, we can conclude that the pull-backs of $ C $ and its recession fan along the evaluation maps are also rationally equivalent. Indeed, the intersection theory provided in \cite{AHR, AR} allows us to avoid repeating the proof of the independence of the number of tropical stable maps, in our enumerative problem, from the exact length of the bounded edges in the multi-line conditions, the exact position of the point conditions, and the exact length of the cross-ratio conditions.

 \end{remark}

We use the following notations for the three types of marked points on the tropical stable map:
\begin{itemize}
    \item Contracted ends with label in $ \underline{n} $ which satisfy point conditions ($ p_{\underline{n}} = \{p_{i}\}_{i=1}^{n} $).
\item Contracted ends with label in $ \underline{k} $ which satisfy multi-line conditions ($ L_{\underline{k}} = \{L_{j}\}_{j=1}^{k} $).
\item Contracted ends with label in $ \underline{f} $ which satisfy no condition, we call them $f-$points.
    
\end{itemize}

Therefore, we have $m= n+ k +f $ marked points for an $(m+d)$-marked rational tropical stable maps $ C = (\Gamma, x_{\underline{m}}, \Tilde{e}_{\underline{d}}, h) \in \mathcal{M}_{0,m+d}(\mathbb{R}^{2}, \Delta_{\mathbb{F}_{r}}(a,b, w_{\underline{d}})) $.




We want to find a recursive formula to count the number of $(m+d)$-marked rational tropical stable maps $ C = (\Gamma, x_{\underline{m}}, \Tilde{e}_{\underline{d}}, h) \in \mathcal{M}_{0,m+d}(\mathbb{R}^{2}, \Delta_{\mathbb{F}_{r}}(a,b, w_{\underline{d}})) $ which satisfy $ n $ point conditions, $ k $ multi-line conditions, and $ l $ cross-ratio conditions.  Since $ \mathcal{M}_{0,m+d}(\mathbb{R}^{2}, \Delta_{\mathbb{F}_{r}}(a,b, w_{\underline{d}})) $ is a tropical object (a polyhedral complex, for more details see \cite{GKM}) of dimension $ \#\Delta_{\mathbb{F}_{r}}(a,b, w_{\underline{d}}) + m -1  $,
we need to impose enough conditions in order to reduce the dimension of the moduli space to zero and get only a finite number of points. Each point condition reduces the dimension of the moduli space by 2, and each multi-line or cross-ratio condition reduces it by 1. In addition, notice that there are $ f $ marked points on the tropical stable maps that satisfy no condition. Hence, the following equality must hold: 
 \begin{center}
	$ \dim \mathcal{M}_{0,m+d}(\mathbb{R}^{2}, \Delta_{\mathbb{F}_{r}}(a,b, w_{\underline{d}})) = \#\Delta_{\mathbb{F}_{r}}(a,b, w_{\underline{d}}) + m -1  = 2n + k + l $.
\end{center}

\begin{definition}
    \label{Nd}
    
    We define $ \mathcal{N}^{0}_{\Delta_{\mathbb{F}_{r}}(a,b, w_{\underline{d}})} (p_{\underline{n}}, L_{\underline{k}}, {\lambda}_{\underline{l}})  := \text{deg}(Z_{\Delta_{\mathbb{F}_{r}}(a,b, w_{\underline{d}})}(p_{\underline{n}}, L_{\underline{k}}, {\lambda}_{\underline{l}}))$, where $Z_{\Delta_{\mathbb{F}_{r}}(a,b, w_{\underline{d}})}(p_{\underline{n}}, L_{\underline{k}}, {\lambda}_{\underline{l}})$ is the following cycle:

    \begin{center}
    $ Z_{\Delta_{\mathbb{F}_{r}}(a,b, w_{\underline{d}})}(p_{\underline{n}}, L_{\underline{k}}, {\lambda}_{\underline{l}}) =  \prod_{i\in \underline{n}}\ev^{*}_i (p_i ) \cdot \prod_{j\in \underline{k}}\ev^{*}_j (L_j ) \cdot \prod_{\iota\in \underline{l}}ft^{*}_{{\lambda}_{\iota}} (0) \cdot \mathcal{M}_{0,m+d}(\mathbb{R}^{2}, \Delta_{\mathbb{F}_{r}}(a,b, w_{\underline{d}})). $
\end{center}
    
    Here $\text{deg}$ is the degree function that sums up all multiplicities of the points in the intersection product. Additionally, we always assume that all conditions are in general position, meaning that $Z_{\Delta_{\mathbb{F}_{r}}(a,b, w_{\underline{d}})}(p_{\underline{n}}, L_{\underline{k}}, {\lambda}_{\underline{l}})$ is a zero-dimensional nonzero cycle that lies inside top-dimensional cells of $\prod_{j\in \underline{l}}ft^{*}_{{\lambda}_{j}} (0) \cdot \mathcal{M}_{0,m+d}(\mathbb{R}^{2}, \Delta_{\mathbb{F}_{r}}(a,b, w_{\underline{d}}))$. Hence, $ \mathcal{N}^{0}_{\Delta_{\mathbb{F}_{r}}(a,b, w_{\underline{d}})} (p_{\underline{n}}, L_{\underline{k}}, {\lambda}_{\underline{l}}) $ represents the number of rational tropical stable maps of degree $ \Delta_{\mathbb{F}_{r}}(a,b, w_{\underline{d}}) $ with $ d $ marked right ends, satisfying $ n $ point conditions $ p_{\underline{n}} $, $ k $ multi-line conditions $ L_{\underline{k}} $, and $ l $ cross-ratio conditions $ {\lambda}_{\underline{l}} $, counted with multiplicity.

\end{definition}

\begin{remark}
   We will assume that the weight \( s \) in Definition \ref{Line} is equal to one, since \( s > 1 \) only contributes multiplicative factors to the intersection product \( \mathcal{N}^{0}_{\Delta_{\mathbb{F}_{r}}(a,b, w_{\underline{d}})} (p_{\underline{n}}, L_{\underline{k}}, {\lambda}_{\underline{l}}) \). Without this assumption, each multi-line condition would carry a weight \( s_j \), leading to a scalar coefficient \( \prod_{j=1}^{k} s_j \) in the final formula of Theorem \ref{Main0}, which can be factored out from all summands.

\end{remark}

\begin{definition}[{\cite{G}}]\label{multcr}
    Let $C = (\Gamma, x_{\underline{m}}, \Tilde{e}_{\underline{d}}, h)$ be a tropical stable map that contributes to  
$\mathcal{N}^{0}_{\Delta_{\mathbb{F}_{r}}(a,b, w_{\underline{d}})}(p_{\underline{n}}, L_{\underline{k}}, {\lambda}_{\underline{l}})$.
For a vertex $v$ of $\Gamma$ with valence $(3 + \#\lambda_v)$ where $\lambda_v = \{\lambda_{j_1}, \dots, \lambda_{j_r} \}$ a \emph{total resolution} of $v$ is a 3-valent labeled abstract tropical curve with $r$ vertices obtained by resolving $v$ through the following recursive process:

\begin{enumerate}
    \item First, resolve $v$ according to $\lambda'_{j_1}$. The resulting two new vertices are denoted by $v_1$ and $v_2$.
    \item Next, choose a vertex $v_k$ such that $\lambda_{j_2} \in \lambda_{v_k}$ and resolve it according to $\lambda'_{j_2}$. This resolution may not be unique; choose one possible resolution.
    \item Continuing this process, after resolving $\lambda'_{j_2}$, we obtain three vertices $v_1, v_2, v_3$. Choose the one containing $\lambda_{j_3} \in \lambda_{v_k}$ and resolve it.
    \item Repeat this procedure until all $\lambda'_{j_r}$ are resolved.
\end{enumerate}

The \emph{cross-ratio multiplicity} $\operatorname{mult}_{\operatorname{cr}}(v)$ of $v$ is defined as the number of total resolutions of $v$. This number is independent of the choice of non-degenerate cross-ratios $\lambda'_{j_1}, \dots, \lambda'_{j_r}$ and, in particular, does not depend on their ordering $\vert \lambda'_{j_1} \vert > \cdots > \vert \lambda'_{j_r} \vert$, see \cite{G18}. 

In the special case where $\#\lambda_v = 0$, we set $\operatorname{mult}_{\operatorname{cr}}(v) = 1$.
\end{definition}

Then the multiplicity of $C$ is given by
\[
\mult(C) = \mult_{\ev}(C) \prod_{v \mid v \text{ vertex of } C} \mult_\text{cr}(v),
\]

where the ev-multiplicity \(\mult_{\ev}(C)\) of $C$ is the absolute value of the determinant of the ev-matrix associated to \(C\) and $ \mult_\text{cr}(v) $ is the number of total resolution of $v$ as defined in \ref{multcr} (see \cite{G18}). 

The determinant of the ev-matrix $M(C)$ associated to \(C\), is the determinant of the evaluation map on the maximal cell of $\prod_{j\in \underline{l}}ft^{*}_{{\lambda}_{j}} (0) \cdot \mathcal{M}_{0,m+d}(\mathbb{R}^{2}, \Delta_{\mathbb{F}_{r}}(a,b, w_{\underline{d}}))$ containing $C$. See Example \ref{Ex0} and \cite[Lemma 5.1]{MR} for more details.

\begin{example}\label{Ex0}
In Figure \ref{ex0} we can see a 6-marked tropical stable map of degree $\Delta_{\mathbb{F}_{3}}(1,1)$ that satisfies four point conditions $ p_{1}, p_2 $, $  p_3 $ and $  p_4 $, two multi-line conditions $ L_{1}, L_2 $. Moreover, this tropical stable map satisfies a non-degenerate cross-ratio condition $ {\lambda}'_1 = (x_{3} x_{5}| x_{6} x_{4})  $ and a degenerate one $ {\lambda}_2 =\{ x_{1}, x_{2}, x_{3}, x_{5} \}$. As we can see in Figure \ref{ex0}, the marked points on the abstract tropical curve are denoted by $ x_i$ and they are mapped to the marked points on the $ h(\Gamma) $ in $ \mathbb{R}^2 $ which are denoted by $ p_{i} $.

    \begin{figure}[h!]
		
		\includegraphics[scale=0.8]{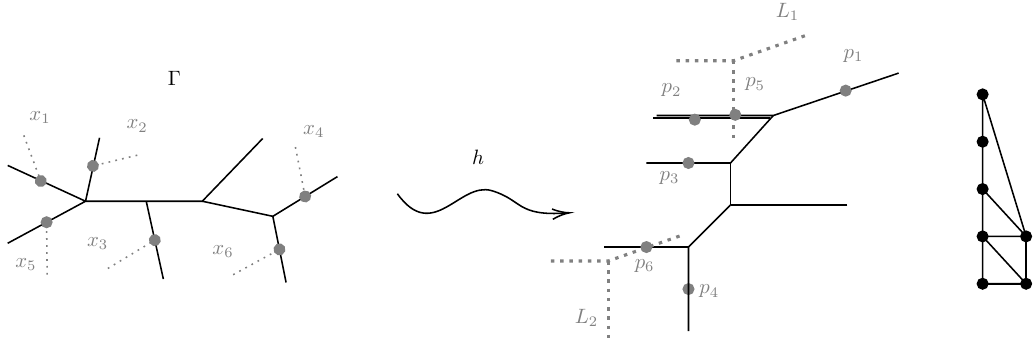}
		\centering
		\caption{In this figure from left to right we can see an abstract tropical curve $ \Gamma $ which is mapped to the plane, by $ h $, so the middle picture is the image of the tropical stable map $h(\Gamma)$ and the rightmost picture is the subdivided Newton polygon dual to the tropical curve. This tropical stable map is fixed in the plane since it satisfies the conditions that are explained in Example \ref{Ex0}. Notice that the marked points on the abstract tropical curve are mapped to the marked points of the image of $ \Gamma $ in the following order: $ h(x_i)=p_{i} $ for $ i= 1, \cdots, 4 $ and $ h(x_j)=p_j \in L_j\cap h(\Gamma) $ for $ j= 5, 6$.} 
		\label{ex0}
	\end{figure}

If we know the combinatorial type of a tropical stable map $C$ in $\mathcal{M}_{0,m+d}(\mathbb{R}^{2}, \Delta_{\mathbb{F}_{r}}(a,b, w_{\underline{d}}))$, having the coordinate of a so-called root vertex and the exact length of the bounded edges will give us the coordinate of $C \in \mathcal{M}_{0,m+d}(\mathbb{R}^{2}, \Delta_{\mathbb{F}_{r}}(a,b, w_{\underline{d}}))$. Therefore, if we wanted to count the number of tropical stable maps in $ \mathcal{N}^{0}_{\Delta_{\mathbb{F}_{3}}(1,1)}(p_{\underline{4}}, L_{\underline{2}}, {\lambda}'_{1}, {\lambda}_{2}) $, then the ev-multiplicity of the curve in Figure \ref{ex0} is the determinant of a matrix whose columns correspond to the coordinates of the root vertex (without loss of generality we can consider $p_1$ as base point) and the lengths of all bounded edges. Notice that, in this example $\mult_{cr}(v)=1$ for all vertices of $\Gamma$.
Denote the lengths of the bounded edges adjacent to marked points $p_i$ (in Figure \ref{ex0}) with $l_i$, and denote the length of the three bounded edges adjacent to no marked point with $l_7,l_8, l_9$. The rows of this matrix correspond to each condition we impose on the tropical stable map.  
In this paper, we refer to such a matrix as the evaluation matrix (ev-matrix) of $ C $ (for more details about ev-matrix, see \cite{GM, FM}). 

 In order to compute the values of each row of the ev-matrix of the tropical stable map in Figure \ref{ex0}, note first that,

\begin{align*}
    h(x_2) &= h(x_1) + l_1 \cdot \begin{pmatrix} -3 \\ -1 \end{pmatrix} 
    + l_2 \cdot \begin{pmatrix} -1 \\ 0 \end{pmatrix}, \\
    h(x_3) &= h(x_1) + l_1 \cdot \begin{pmatrix} -3 \\ -1 \end{pmatrix} 
    + l_7 \cdot \begin{pmatrix} -1 \\ -1 \end{pmatrix} + l_3 \cdot \begin{pmatrix} -1 \\ 0 \end{pmatrix}, \\
    h(x_4) &= h(x_1) + l_1 \cdot \begin{pmatrix} -3 \\ -1 \end{pmatrix} 
    + l_7 \cdot \begin{pmatrix} -1 \\ -1 \end{pmatrix} 
    + l_8 \cdot \begin{pmatrix} 0 \\ -1 \end{pmatrix} 
    + l_9 \cdot \begin{pmatrix} -1 \\ -1 \end{pmatrix}
    + l_4 \cdot \begin{pmatrix} -1 \\ 0 \end{pmatrix}, \\
    h(x_5) &= h(x_1) + l_1 \cdot \begin{pmatrix} -3 \\ -1 \end{pmatrix} 
    + l_5 \cdot \begin{pmatrix} -1 \\ 0 \end{pmatrix}, \\
     h(x_6) &= h(x_1) + l_1 \cdot \begin{pmatrix} -3 \\ -1 \end{pmatrix} 
    + l_7 \cdot \begin{pmatrix} -1 \\ -1 \end{pmatrix} 
    + l_8 \cdot \begin{pmatrix} 0 \\ -1 \end{pmatrix} 
    + l_9 \cdot \begin{pmatrix} -1 \\ -1 \end{pmatrix} + l_6 \cdot \begin{pmatrix} -1 \\ 0 \end{pmatrix}.
\end{align*}

Since $L_1$ intersects with $h(\Gamma)$ with its end in $(0,-1)$ direction, to determine the locus of $L_1 \cap h(\Gamma)$ we need to know $(p_5)_x$.
For the other multi-line condition, $L_2$ intersects with $h(\Gamma)$ with its end in $(3,1)$ direction, so we need to know $(p_6)_x$. Moreover,

\begin{center}
    $
     h(x_6) = x\begin{pmatrix} 1 \\ 0 \end{pmatrix} + y\begin{pmatrix} 3 \\ 1 \end{pmatrix}. 
    $
\end{center}

So the ev-matrix of $C$ is Matrix \eqref{multi}.

\begin{equation}\label{multi}
    \begin{array}{c|ccccccccccc}
        & (p_1)_x & (p_1)_y & l_1 & l_2 & l_3 & l_4 & l
        _5& l_6 &l_7 & l_8 &l_9 \\
        \hline
        (p_1)_x & 1 & 0 & 0 & 0 & 0 & 0 & 0 & 0 & 0 & 0&0\\
        (p_1)_y & 0 & 1 & 0 & 0 & 0 & 0 & 0 & 0 & 0 & 0&0\\
        (p_2)_x & 1 & 0 & -3 & -1 & 0 & 0 & 0 & 0 & 0 & 0&0\\
        (p_2)_y & 0 & 1 & -1 & 0 & 0 & 0 & 0 & 0 & 0 & 0&0\\
        (p_3)_x & 1 & 0 & -3 & 0 & -1 & 0 & 0 & 0 & -1 & 0 & 0\\
        (p_3)_y & 0 & 1 & -1 & 0 & 0 & 0 & 0 & 0 & -1 & 0&0\\
        (p_4)_x & 1 & 0 & -3 & 0 & 0 & 0 & 0 & 0 & -1 & 0&-1\\
        (p_4)_y & 0 & 1 & -1 & 0 & 0 & -1 & 0 & 0 & -1 & -1&-1\\
        L_1 & 1 & 0 & -3 & -1 & 0 & 0 & 0 & 0 & 0 & 0&0\\
        L_2 & 1 & -r & -1 & 0 & 0 & 0 & 0 & -1 & 1 & -3 & 1\\
        \lambda'_1 & 0 & 0 & 0 & 0 & 0 & 0 & 0 & 0 & 0 & 1&1\\
    \end{array}
\end{equation}

\end{example}

\begin{remark}\label{Gol18}
In our new setting, Lemma 2.14 of \cite{G18}, which says the following fact about cross-ratio conditions, still holds:
Let $C$ be a tropical stable map that contributes to $ \mathcal{N}^{0}_{\Delta_{\mathbb{F}_{r}}(a,b, w_{\underline{d}})}(p_{\underline{n}}, L_{\underline{k}}, {\lambda}_{\underline{l}}) $. Let $v \in C$ be a vertex of $C$ such that $\val(v)>3$. Then for every edge $e$ adjacent to $v$ either there is an entry $x_i$ in some ${\lambda}_j \in {\lambda}_v$ such that $e$ is in the shortest path from $v$ to the end labeled with $x_i$ or $e$ leads to a non-contracted marked end in some ${\lambda}_j \in {\lambda}_v$.
\end{remark}



\section{Splitting tropical stable maps on Hirzebruch surfaces 
}\label{STSM}

 The main goal of this section is to prove Proposition \ref{*} and Lemma \ref{n0}, which explain how the tropical stable maps we are dealing with can be cut.
To provide a recursion to count the number of tropical stable maps contributing to ${\mathcal{N}^{0}_{\Delta_{\mathbb{F}_{r}}(a,b, w_{\underline{d}})}(p_{\underline{n}}, L_{\underline{k}}, {\lambda}_{\underline{l}})}$, based on Remark \ref{RationalyEquivalent}, we can consider all of the cross-ratios are of length zero, except for the last one. We prove in Proposition \ref{*}, that if the length of the last cross-ratio condition is very long then there are two generally different ways to split a tropical stable map contributing to ${\mathcal{N}^{0}_{\Delta_{\mathbb{F}_{r}}(a,b, w_{\underline{d}})}(p_{\underline{n}}, L_{\underline{k}}, {\lambda}_{\underline{l-1}}, \lambda'_l)}$ into pieces. Indeed, we will study two cases separately: one where there is at least one point condition on the tropical stable map (i.e. $n\geq1$) and another one where there is no point condition (i.e. $n=0$). Proposition \ref{*} addresses the first case, and Lemma \ref{n0} addresses the second case.
The main challenge, as expected, is having more than one cross-ratio condition. This results in having more complicated movable components compared to the string appearing in \cite[Lemma 2.10]{FM}, which only has two 3-valent vertices. Here, the movable parts that can occur are called movable branches, defined in Definition \ref{DefString}. 


{The proof of Proposition \ref{*} contains the following 3 main steps:

\begin{itemize} 

\item \textbf{I}: Let $C$ be a tropical stable map that satisfies the given incidence conditions in Proposition \ref{*}. 
If we drop the cross-ratio condition ${\lambda}'_{l}$, as it is described in Remark \ref{B}, we obtain a 1-dimensional family of tropical stable maps that satisfy all of the incident conditions in Proposition \ref{*}, except for the ${\lambda}'_{l}$. A tropical stable map in this 1-dimensional family has a so-called movable component $B$ (see Definition \ref{Movable}). We will first study some of the important features of this movable component in Definitions \ref{IIII}, \ref{Partial order}, and Observations \ref{f5}, \ref{f6}, which we are going to use in the next steps. 

\item \textbf{II}: Using a partial order on the set of vertices in $B$, we obtain a maximal chain in $B$ (See Definition \ref{Partial order}). Then in  {three steps, see Lemma \ref{1Ray} and Remark  \ref{bigstring}}, we prove that the length of such a maximal chain is either equal to one or the maximal chain consists of the vertices in one of the movable branches classified in the second part of Proposition \ref{*}. 

\item \textbf{III}: Then we prove in Lemma \ref{1ver}, that if $ B $ contains only one vertex and $C$ does not have a contracted bounded edge or a movable branch as classified in \ref{*}, then the movement of $B$ is bounded. This is a contradiction. Therefore, a tropical stable map satisfying the incident conditions of the Proposition \ref{*}, either has a contracted bounded edge or a movable branch $ \mathcal{S} $. \end{itemize}}

\begin{remark}\label{-1}
    The moduli space $\mathcal{M}_{0,m+d}(\mathbb{R}^{2}, \Delta_{\mathbb{F}_{r}}(a,b, w_{\underline{d}}))$ of all $(m+d)$-marked tropical stable maps is of dimension $\#\Delta_{\mathbb{F}_{r}}(a,b, w_{\underline{d}}) + m - 1$.
 Imposing a point condition on the tropical stable maps reduces the dimension of the moduli space by 2, while each multi-line or cross-ratio condition reduces it by 1. If we consider all $(m+d)$-marked tropical stable maps contributing to ${\mathcal{N}^{0}_{\Delta_{\mathbb{F}_{r}}(a,b, w_{\underline{d}})}(p_{\underline{n}}, L_{\underline{k}}, {\lambda}_{\underline{l}})}$ where $m = n+k+f$, since $ \#\Delta_{\mathbb{F}_{r}}(a,b, w_{\underline{d}}) + m - 1 = 2n +k+l$, dropping one of the cross-ratio conditions results in obtaining a 1-dimensional family in ${\mathcal{M}_{0,m+d}(\mathbb{R}^{2}, \Delta_{\mathbb{F}_{r}}(a,b, w_{\underline{d}}))}$. We denote this 1-dimensional family by $Y$:
\begin{center}
    $ Y =   \prod_{i\in \underline{n}}\ev^{*}_i (p_i ) \cdot \prod_{j\in \underline{k}}\ev^{*}_j (L_j ) \cdot \prod_{\iota\in \underline{l-1}}ft^{*}_{{\lambda}_{\iota}} (0) \cdot \mathcal{M}_{0,m+d}(\mathbb{R}^{2}, \Delta_{\mathbb{F}_{r}}(a,b, w_{\underline{d}})). $
\end{center}
    
\end{remark}

\begin{definition}\label{Movable}
Let $  C = (\Gamma, x_{\underline{m}}, \Tilde{e}_{\underline{d}}, h) $ be an $(m+d)$-marked tropical stable map in ${ \mathcal{M}_{0,m+d}(\mathbb{R}^{2}, \Delta_{\mathbb{F}_{r}}(a,b, w_{\underline{d}}))} $  that  contributs to ${ \mathcal{N}^{0}_{\Delta_{\mathbb{F}_{r}}(a,b, w_{\underline{d}})}( p_{\underline{n}}, L_{\underline{k}}, {\lambda}_{\underline{l-1}} ) }$  where $ m = n+k+f $. Based on Remark \ref{-1}, $C$ is in the 1-dimensional family $Y$.
By moving some of the vertices of any tropical stable map, such as $C$, in this family, we can find all of the other tropical stable maps in this family. Therefore, we can talk about the movable and fixed parts of $C$.
A subgraph of $ \Gamma $ is called \textit{a movable component} of $ C $ if it can be moved unboundedly without changing the combinatorial type of $ C $. We can also define \textit{fixed components} of $ C $, which are the connected components of $ \Gamma $ obtained by removing the movable component.  

\end{definition}

\begin{remark}\label{B}

Let $ C $ be a tropical stable map contributing to ${\mathcal{N}^{0}_{\Delta_{\mathbb{F}_{r}}(a,b, w_{\underline{d}})}(p_{\underline{n}}, L_{\underline{k}}, {\lambda}_{\underline{l-1}}, {\lambda}'_{l})}$, where $|{\lambda}'_{l}|$ can be very large. Based on Remark \ref{-1}, this results in obtaining a 1-dimensional family in ${ \mathcal{M}_{0,m+d}(\mathbb{R}^{2}, \Delta_{\mathbb{F}_{r}}(a,b, w_{\underline{d}}))} $. A tropical stable map in this family has a movable component that we call $ B $. We denote the set of all vertices in $B$ with $V$. Obviously, $V$ is a non-empty set and from now we assume that $\#V\geq 2$. The case when $\#V=1$ is studied in Lemma \ref{1ver} and Observation \ref{f5}.
\end{remark}
Proposition \ref{*} claims that a tropical stable map $C$ contributing to ${\mathcal{N}^{0}_{\Delta_{\mathbb{F}_{r}}(a,b, w_{\underline{d}})}(p_{\underline{n}}, L_{\underline{k}}, {\lambda}_{\underline{l-1}}, {\lambda}'_{l})}$ where $|{\lambda}'_{l}|$ can be very large, will either have a contracted bounded edge or a movable branch $ \mathcal{S} $. Studying the features of the movable component $ B $, defined in Remark \ref{B} will lead to realizing that $B$ must be the movable branch $\mathcal{S}$.

\begin{definition}\label{IIII}
	
	A vertex $v$ of the tropical stable map is called an \textit{ordinary vertex} if its valence is at least 3 when we exclude all {contracted marked ends} at this vertex. Every vertex in the movable component $B$, defined in \ref{B}, is either adjacent to a fixed vertex or adjacent to a contracted end caused by a multi-line condition. Otherwise, it can move independently from the movement of $B$ (see Figure \ref{0type}). So every vertex that is not adjacent to any contracted ends in $B$ has to be adjacent to a vertex in the fixed component. 

 Therefore, the following are the three different types of vertices in $B$:

 \begin{itemize}
     \item  \textbf{Type I:} An ordinary vertex in $B$ that is adjacent to a vertex in a fixed component via a bounded edge is called Type I.

     \item  \textbf{Type II:} An ordinary vertex in $B$ adjacent to a contracted end caused by some multi-line condition is called Type II.

     \item \textbf{Type III:} Except for ordinary vertices, there can be vertices of valence at least 3, which are 2-valent where all of the contracted {marked end}s are excluded; we call them type III. At least one of the contracted {marked end}s at a vertex of type III is a {marked point} caused by a multi-line condition. Notice that the existence of two contracted {marked end}s caused by multi-line conditions fixes the vertex. This contradicts the fact that we are considering vertices in the movable component. Therefore, exactly one of the contracted {marked end}s at a vertex of type III is caused by a multi-line condition. We will also refer to these vertices as marked points caused by a multi-line condition.

 \end{itemize}
 
 	\begin{figure}[h!]
		
		\includegraphics[scale=0.8]{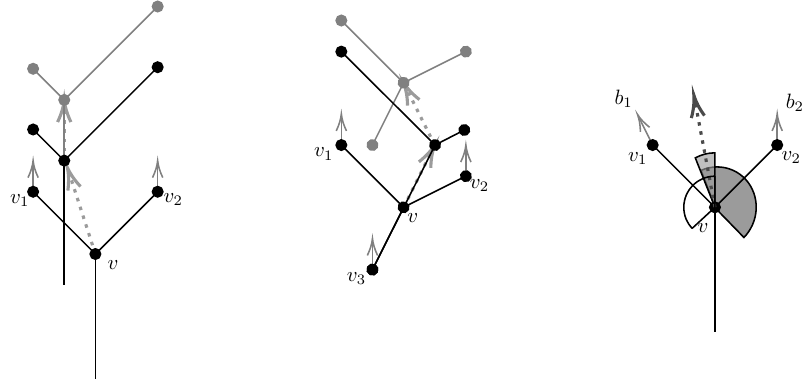}
		\centering
		\caption{Let $v$ be a vertex in the movable component $B$ that is not adjacent to a fixed vertex or to a contracted end caused by a multi-line condition. In all three pictures, the dashed grey vector shows a possible direction of movement for $v$. In the left picture, $v$ is adjacent to two vertices that have the same directions of movement.  In the middle picture, $v$ is adjacent to three vertices that have the exact same directions of movement. In the right-hand picture, $v$ is adjacent to two vertices that have different directions of movement, and the light grey arc shows all of the possible unbounded directions of movement for $v$. In all of these pictures, we can see that $v$ does not have a unique direction of movement, and this results in having more than one dimension of movement. This observation proves that every ordinary vertex in $B$ is either adjacent to a fixed vertex or adjacent to a contracted end caused by a multi-line condition.}
		\label{0type}
	\end{figure}
\end{definition}

 A similar classification for the vertices of the movable component on $\mathbb{P}^2 $ can be seen in \cite{G}. 
Since we define a multi-line condition as in Definition \ref{Line}, we call the three directions $ (-1, 0) $, $ (0, -1) $, and $ (r, 1) $ the standard directions.
The direction of movement of type II and type III vertices in $ B $ (if we already moved the movable component enough) can only be in one of the standard directions.

\begin{notation}\label{Gati}
Let $\Gamma$ be the abstract graph of a tropical stable map $C$  contributing to ${\mathcal{N}^{0}_{\Delta_{\mathbb{F}_{r}}(a,b, w_{\underline{d}})}(p_{\underline{n}}, L_{\underline{k}}, {\lambda}_{\underline{l-1}}, {\lambda}'_{l})}$, where $|{\lambda}'_{l}|$ can be very large. We can forget all the type III vertices on $ \Gamma $ (by using a forgetful map, as in \ref{forgetfulmap}) and call the resulting graph $ \Tilde{\Gamma} $. We mainly focus on ordinary vertices rather than type III vertices. When it is necessary to consider the type III vertices, we will indicate that we are switching to $ \Gamma $. Notice that, if $ \Tilde { \Gamma } $ allows no movement, then there is only one vertex in the movable component of $ \Gamma $ which is of type III (because we are in a 1-dimensional family). This specific case, where $B$ contains one vertex, is separately studied in Lemma \ref{1ver}.
     
\end{notation}

Here, we recall some of the notations and definitions from \cite{G}.

\begin{remark}\cite[Angle Lemma]{G}\label{Ang}
 If $ \Tilde{\Gamma} $ allows an unbounded 1-dimensional movement (based on Notation \ref{Gati}) and $ v_1, v_2 $ are adjacent vertices in the movable component of $ \Tilde{\Gamma} $, then $ b_2 $ the direction of movement of $ v_2 $ lies in the half-open cone:
 \begin{center}
     $  \sigma_{v_2}(b_1, e_1) := \{ x \in \mathbb{R}^2 : \; x = v_2 + {\lambda}_1 v(e_1, v_1) + {\lambda}_2 b_1, {\lambda}_1 \in \mathbb{R}_{\geq 0}, {\lambda}_2 \in \mathbb{R}_{> 0} \},  $
 \end{center}
where $b_1$ is the direction of movement of $v_1$ and $ v(e_1, v_1) $ is the direction vector of the edge $ e_1 $ at $ v_1 $ that connects $ v_1 $ and $ v_2 $.  Half-open means that the boundary of $ \sigma_{v_2}(b_1, e_1) $ that is generated by $b_1$ is part of the cone, and the boundary that is generated by $v(e_1, v_1)$ is not part of the cone, while $v_2$ itself is also not part of the cone.\\
     This half-open cone can be seen in Figure \ref{angle}. Indeed, we can find all possible directions of movement of one vertex by knowing the direction of movement of another.

		\begin{figure}[h!]
		
		\includegraphics[scale=1]{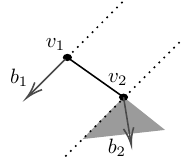}
		\centering
		\caption{In this figure, $ b_1 $ is the given direction of movement of $ v_1 $ and all possible directions of movements for the vertex $ v_2 $ live in the grey area. The grey area in this figure is $ \sigma_{v_2}(b_1, e_1) $.}
		\label{angle}
	\end{figure}
\end{remark}

\begin{definition} Here we recall the definition of the partial order from \cite[Definition 37]{G}. \label{Partial order}
Let $ \Tilde{\Gamma} $ be the metric tree of a tropical stable map $ C = (\Gamma, x_{1}, \cdots, x_{m}, h) $ in $ \mathbb{R}^2 $ that allows an unbounded 1-dimensional movement and let $ H $ be an open half-plane. If we translate \( H \) to a vertex \( v \in \widetilde{\Gamma} \), meaning that \( v \) lies on the boundary of \( H \), we denote the translated half-plane by \( H_v \).
 Let \( M \) be the set of all vertices of the movable component of \( \widetilde{\Gamma} \); that is, \( M \) consists of all type (I) and type (II) vertices of the movable component of \( \Gamma \).
The half-plane $ H $ induces a partial order $ \Omega(H) $ on $ M $ as follows: 

For $ v_{1},  v_2 \in M $, define:

\begin{equation*}
    v_1 \geq v_2 : \Longleftrightarrow \left\{
    \begin{array}{l}
        v_1 = v_2,  \text{ or } \\
         v_2 \text{ is adjacent to } v_1 \text{ and } v_2 \in H_{v_1}.  
    \end{array}\right.
\end{equation*}

Here, we only use open half-planes $ H $ such that $ b_1 + v_1 \in H_{v_1} $ ($ b_1 $ is the direction of movement of $ v_1 $). Therefore if $ v_1 \geq \cdots \geq v_n $ is a maximal chain and $ b_i $ is the direction of movement of $ v_i $ for $ i = 1, \cdots , n $, then $ b_i + v_i \in H_{v_i} $
for $ i = 1, \cdots , n $. This follows from Remark \ref{Ang} (Angle Lemma).

\end{definition} 

\begin{notation}\label{38} Given a chain $v_1 \geq \cdots \geq v_n$ in the movable component of $\Tilde{\Gamma}$, throughout this section, we denote the direction of movement of $v_i$ by $b_i$ for each $i = 1, \cdots, n$. If such a chain is maximal, then we usually denote an edge connecting $v_i$ and $v_{i+1}$ by $e_i$ for each $i = 1, \cdots, n-1$.

\end{notation}

The following observation about the direction of movement of the vertices in the movable component is handy when we want to observe the behavior of the vertices in $ B $ in each combinatorial case that will arise and also to get a contradiction to eliminate all of the impossible cases.
\begin{observation}\label{f4}
  If $v_1$ is a vertex in $B$ adjacent to a contracted end caused by a multi-line condition, then, after sufficient movement, its unbounded direction is one of the standard directions $(r,1)$, $(0,-1)$, or $(-1,0)$. Otherwise, moving $B$ will cause the intersection of $h(\Gamma)$ with the multi-line condition to no longer occur at $v_1$, thereby changing the combinatorial type of $C$.
Notice that if $ v_1 $ is adjacent to a non-contracted end $ e $, then $ b_1 $ cannot be aligned with the direction of the end $ e $.

\end{observation}
Here, we address the case of parallel edges to avoid discussing it multiple times:

\begin{observation}\label{f5}
    We can assume that not all bounded and unbounded edges adjacent to an ordinary vertex in $ B $ are parallel unless $\#V=1$, which will be studied in \ref{1ver}. If $ v $ is a vertex in $ B $ such that all edges adjacent to it are parallel, one of the two following cases happens:
    
  Either there is another ordinary vertex in $ B $ adjacent to $ v $ like $ v_2 $ as shown in Figure \ref{fBp}.

	\begin{figure}[h!]
		
		\includegraphics[scale=0.9]{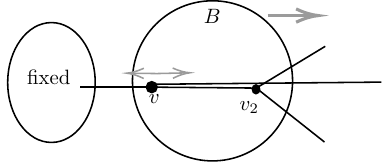}
		\centering
		\caption{In this figure, we can see a vertex $ v$ in the movable component that can freely move along the two-direction arrow, which is independent of the direction of the movement of $ B $.}
		\label{fBp}
	\end{figure}

       Obviously, the movement of $ v $ in this setting is independent of the movement of $ B $, so $ B $ has at least a 2-dimensional movement, which is a contradiction.

The second case is when there is one bounded edge adjacent to $ v $ and all of the other edges are parallel ends. In this case, $ B $ has only one vertex $ v $ and this case will be studied later in \ref{1ver}.

\end{observation}

\begin{observation}\label{f6}
  We call $P$ the Newton polygon dual to $h(\Gamma)$ (see Definition \ref{Newton}). The movable component $ B $ is dual to some part of the subdivision of $ P $. Indeed, from one side (assume the left side) $ B $ is connected to the fixed components and if we name all bounded edges that connect $ B $ to the fixed components by $ u_{1}, \cdots, u_s $ from top to bottom (as in Figure \ref{wi}), then each $ u_i $ can be obtained from $ u_{i+1} $ by a left turn (otherwise the movement is bounded). Hence, $B$ is dual to a concave subsection of polygon $P$ on one side and to at least two edges of $P$ on the other side. If the Newton polygon dual to $B$ contains only one of the boundary edges of $P$, then it cannot be dual to a concave part from the other side.

\begin{figure}[h!]
		
		\includegraphics[scale=0.9]{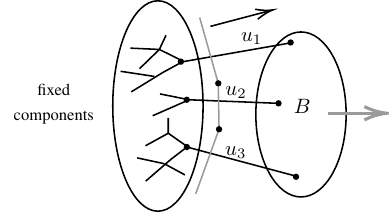}
		\centering
		\caption{In this figure, the image of the tropical stable map $C$ in the plane is depicted as a fixed component and a movable component $B$ and some bounded edges that connect them. }
		\label{wi}
	\end{figure}

  Notice that there might be some vertices of the fixed components on the right side of $ B $. Since these vertices cannot be adjacent to any vertex in $ B $ (like the following figure), we can move $B$ over the fixed part, without changing the combinatorial type (see Definition \ref{CombinatoriyalType}.) of the tropical stable map.
\begin{figure}[h!]
		
		\includegraphics[scale=0.9]{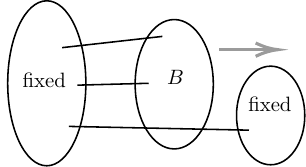}
		\centering
		\caption{This diagram illustrates a scenario in which the movable component $B$ will pass over one of the fixed components. }
		\label{fBfix}
	\end{figure}

 We can also have some ends, adjacent to the vertices in the fixed components, that intersect with the movable components, like Figure \ref{fBend}.  

\begin{figure}[h!]
		
		\includegraphics[scale=0.9]{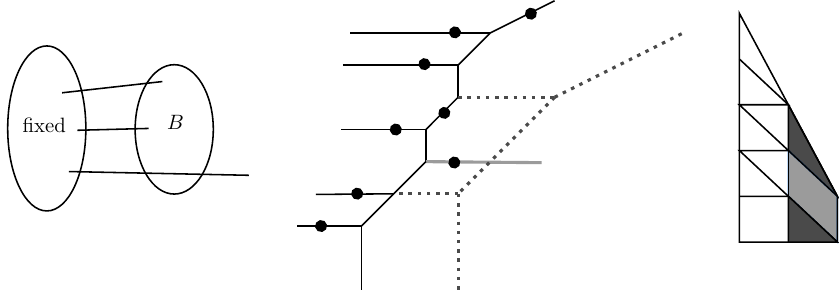}
		\centering
		\caption{In the middle graph of this figure, there is an example of a tropical stable map with a movable component. The fixed parts of the tropical stable map, known as the fixed component, are represented by simple lines (black and grey). This fixed part of the tropical stable map corresponds to the uncolored parts of the Newton polygon depicted on the right-hand side of the picture.\\
  The dashed black lines represent the movable component, corresponding to the dark grey parts of the Newton polygon shown on the right side of the picture. Additionally, there is a light grey rectangle corresponding to the fixed grey end of the tropical stable map. We can see that the fixed grey end is intersecting with the movable component.}
		\label{fBend}
	\end{figure}
		
\end{observation}

\begin{definition}\label{H}

For any two ordinary adjacent vertices $ v_{1}, v_2 $, and their directions of movement $ b_{1}, b_2 $ we can have an open half-plane $ H $ such that $ v_1 \geq v_2 $ (i.e., \( e_1, b_1 + v_1 \in H_{v_1} \)). So we can make a maximal chain starting from $ v_1 $ like $ v_1 \geq v_2 \geq \cdots \geq v_n $. We consider the number of vertices in a maximal chain as its length.\\
We call an open half-plane special if its boundary is parallel to one of the lines: $ \{ x=0 \} $, $ \{ y=0 \} $, or $ \{ ry=x \} $ (see Figure \ref{half-planes}). 

\begin{figure}[h!]
		
		\includegraphics[scale=0.8]{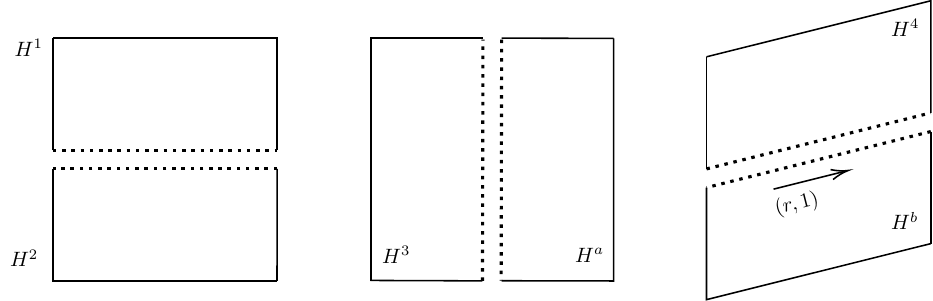}
		\centering
		\caption{This figure depicts all types of special open half-planes. We will refer to these special open half-planes with the names written on them in this figure. }
		\label{half-planes}
	\end{figure}

Some of the special open half-planes only contain one direction of ends, we call them 1-ray special open half-plane ($H^{i}$ for $i=1,\cdots, 4$ in Figure \ref{half-planes}), and some of them contain two directions of ends, we call them 2-ray special open half-planes. The only 2-ray special open half-planes are $ H^{a} $ and $ H^{b} $ in Figure \ref{half-planes}.

\end{definition}

\begin{remark} 

Let $ v_1 \geq v_2 \geq \cdots \geq v_n $ be a maximal chain of vertices in the movable component of $ \Tilde{\Gamma}$ with respect to a $ \Omega(H) $. There are 3 scenarios for $ H $ and the maximal chain of vertices mentioned in Definition \ref{H}. These are: (1) when $ H $ is a 1-ray special open half-plane (studied in Lemma \ref{1Ray}), (2) when $ H $ is a 2-ray special open half-plane (studied in Remark \ref{bigstring}), and (3) when $ H $ is not a special open half-plane (studied in Lemma \ref{NoRay}). 

\end{remark}

   We need to provide a more general version of the definition of {string} given in \cite{GM} in the following way:

\begin{definition}

\label{DefString} 
Let $ C = (\Gamma, x_{\underline{m}}, \Tilde{e}_{\underline{d}}, h)  $ be a tropical stable map contributing to ${\mathcal{N}^{0}_{\Delta_{\mathbb{F}_{r}}(a,b, w_{\underline{d}})}(p_{\underline{n}}, L_{\underline{k}}, {\lambda}_{\underline{l-1}}, {\lambda}'_{l})}$, where $|{\lambda}'_{l}|$ can be very large.
A subgraph $\mathcal{S} $ of $ \Gamma $  is called a \emph{movable branch} of $ C $ if it satisfies the following conditions. 

\begin{enumerate}[label=(\arabic*)] 
    \item There is no contracted end satisfying a point condition in $\mathcal{S}$.

    \item \label{cond2} Any bounded edge of $\mathcal{S} $ that is adjacent to a 1-valent vertex is of primitive direction $(1,0)$ (in direction away from the 1-valent vertex). By abuse of notation, we prolong these bounded edges and treat them as left ends that are fixed (see Definition \ref{fixedends}). One can view $\mathcal{S} $ as a connected component after cutting $\Gamma$ at the bounded edges, becoming ends. We refer to a movable branch with $\sigma$ fixed ends of primitive direction $(-1,0)$ as $\mathcal{S}_\sigma$. These $\sigma$ fixed left ends have weights $c_{\underline{\sigma}}$. 

    \item The direction of movement of all vertices of $\mathcal{S}_\sigma$ is in the $(1,0)$ direction.

\end{enumerate}

Based on the number of fixed left ends of the movable branches, we classify them as follows:

\begin{itemize}
    \item Case $\sigma > 1$:
\end{itemize}

\begin{enumerate}[label=(\arabic*)] 
    \setcounter{enumi}{3} 

    \item \label{fifth}$\mathcal{S}_\sigma$ corresponds to a connected component of $\Gamma$ with $0<a_0 \leq a$ ends in each of $(r,1)$ and $(0,-1)$ directions and $d_0 \leq d $ marked right ends. All of the marked right ends are labeled with $\Tilde{e}_{0\underline{d_0}}$, and have weights $w_{0\underline{d_0}}$ (for more details see Figure \ref{ExSa}). The Newton polygon dual to the image of $\mathcal{S}_\sigma$ in the plane is the convex hull of these vertices: 
    \[
    \{(0,0), (a_0,0), (a_0,\sum_{i=1}^{d_0}w_{0i}), (0, ra_0 +\sum_{i=1}^{d_0}w_{0i}) \}.
    \]
    Of course, $\sum_{i=1}^\sigma c_i = ra_0 +\sum_{i=1}^{d_0}w_{0i} $.

    \item Each vertex adjacent to an end in direction $(r,1)$ or $(0,-1)$ is 3-valent.

    \item Each vertex adjacent to an end with primitive direction $(1,0)$ or $(-1,0)$ has valence higher than 3. Notice that we can have ends in the $(-1,0)$ direction that are not the fixed ends raised from cutting bounded edges, as discussed in Condition \ref{cond2}. Since such an end is adjacent to a high-valence vertex, there must be a marked point caused by a multi-line condition on it.

\end{enumerate}

\begin{itemize}
    \item Case $\sigma = 1$:
\end{itemize}

\begin{enumerate}[label=($\arabic*'$)]
    \setcounter{enumi}{3} 
    
    \item $\mathcal{S}_1$ has a single vertex $v$ adjacent to ends in the $(1,0)$ or $(-1,0)$ direction. The vertex $v$ is connected to multiple ends, all of which are parallel. 

    \item There must be a marked point caused by a multi-line condition on every end in $(-1,0)$ direction, which is not a fixed end discussed in Condition \ref{cond2}.

\end{enumerate}

\end{definition}

In Figure \ref{S} you can see different examples of a movable branch.
	\begin{figure}[h!]
		
		\includegraphics[scale=0.6]{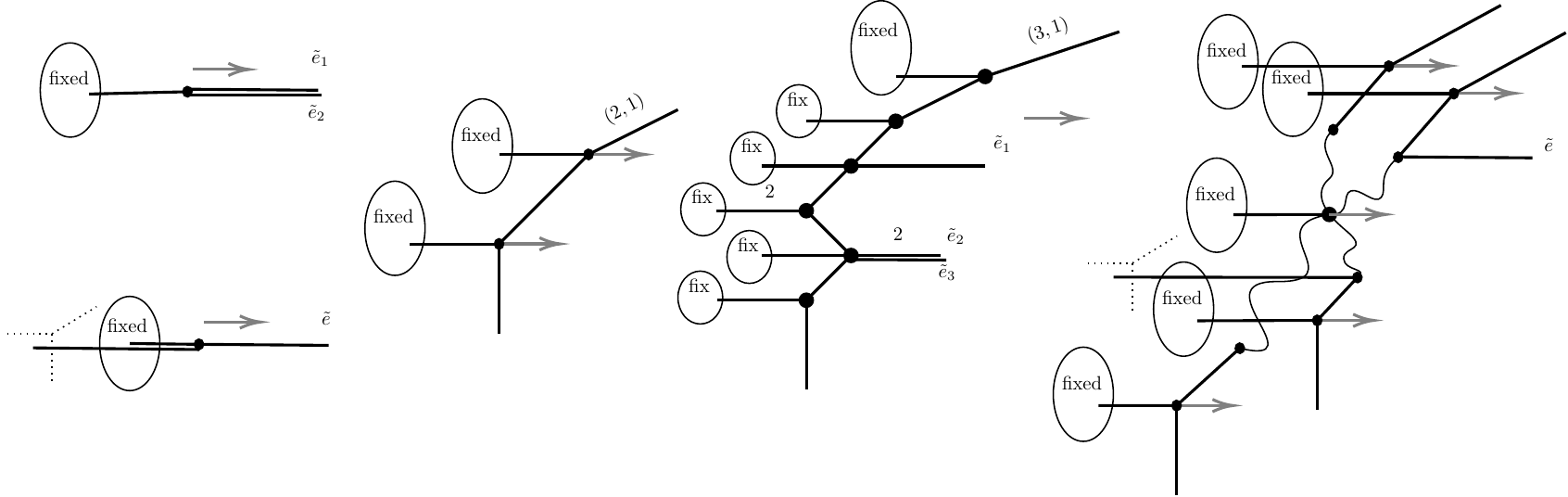}
		\centering
		\caption{Some movable branch that a tropical stable map of degree $ \Delta_{\mathbb{F}_{r}}(a,b, w_{\underline{d}}) $ can have where it satisfies a cross-ratio condition $ {\lambda}'_{l} $ with very long length. The leftmost movable branches are two examples of $\mathcal{S}_1$.}
		\label{S}
	\end{figure}

\begin{proposition}\label{*}
	Let $ n \geq 1 $ and $ C $ be a tropical stable map contributing to $ \mathcal{N}^{0}_{\Delta_{\mathbb{F}_{r}}(a,b, w_{\underline{d}})}(p_{\underline{n}}, L_{\underline{k}}, {\lambda}_{\underline{l-1}}, {\lambda}'_{l}) $. If  $|{\lambda}'_{l}|$  has a very large length, then $ C $ satisfies one of the following conditions:
	
	\begin{enumerate}
		\item $ C $ has a contracted bounded edge.
		\item $ C $ has a movable branch $ \mathcal{S} $ (defined in \ref{DefString}).

	\end{enumerate}
\end{proposition}

\begin{observation}\label{oneEnd}
 Let $ B $ be a movable component of a tropical stable map contributing to ${\mathcal{N}^{0}_{\Delta_{\mathbb{F}_{r}}(a,b, w_{\underline{d}})}(p_{\underline{n}}, L_{\underline{k}}, {\lambda}_{\underline{l-1}}, {\lambda}'_{l})}$, where $|{\lambda}'_{l}|$ can be very large.
It is easy to see that the movable component of $ C $ has the following features:

	\begin{enumerate}

		\item Let \( v_1 \) be an ordinary vertex in \( B \) adjacent to an end. If there is a marked point caused by a multi-line condition on this end, then we can see in Figure \ref{Ends} that all possible unbounded directions of movement for \( v_1 \), which do not change the combinatorial type of \( C \), lie inside the grey areas. This means that if \( v_1 \) moves sufficiently in a direction outside the grey area, \( C \) no longer satisfies the multi-line condition on the end adjacent to \( v_1 \).

	\begin{figure}[h!]
		
		\includegraphics[scale=0.8]{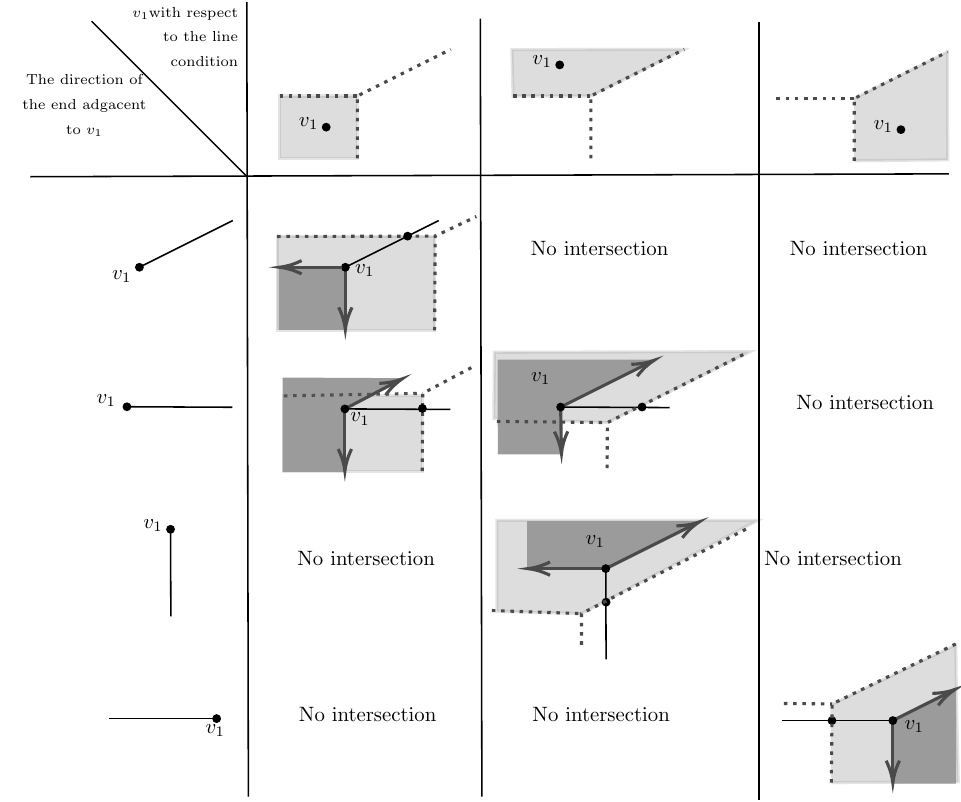}
		\centering
		\caption{In the left column of this table, we can see the 4 different directions for an end that a tropical stable map in $\mathbb{F}_{r} $ can have. Assume $v_1 $ is a vertex in $V$ adjacent to an end (we consider vertices in $\Gamma$). Also, we assume that there is a {marked point} on this end that is caused by a multi-line condition. We can determine all of the possible directions of unbounded movement that $v_1$ can have. So in the top row of the table, we depicted all three possible initial positions of the vertex $v_1$ with respect to the multi-line condition. Notice that these locations are light grey areas. \\ Inside the table, the dark grey areas trapped between two grey arrows identify all possible directions of movement for $v_1$. In some of the cases, the given end adjacent to $v_1$ does not intersect with the multi-line condition, where $v_1$ is in a specific area given in the top row. }
		\label{Ends}
	\end{figure}

\item Since we want to move each vertex unboundedly, if there is more than one end adjacent to a vertex and there is at least one {marked point} caused by a multi-line condition on each end, the only possible cases are the 9 cases in Figure \ref{9}.\\
In the third row of Figure \ref{9}, we have the cases where a vertex is adjacent to 3 ends. Notice that if the end in the direction $(1,0)$ was a marked right end, we would have obtained the same directions of unbounded movement for $v$.  Also there is one case missing in Figure \ref{9} and that is where $v$ is adjacent to 3 ends in directions $(r,1)$, $(-1,0)$ and $(0, -1)$. Since the intersection between 3 dark areas corresponding to the unbounded directions of movement of a vertex adjacent to these ends is empty, this case cannot happen.

\begin{figure}[h!]
\centering
\includegraphics[scale = 0.8]{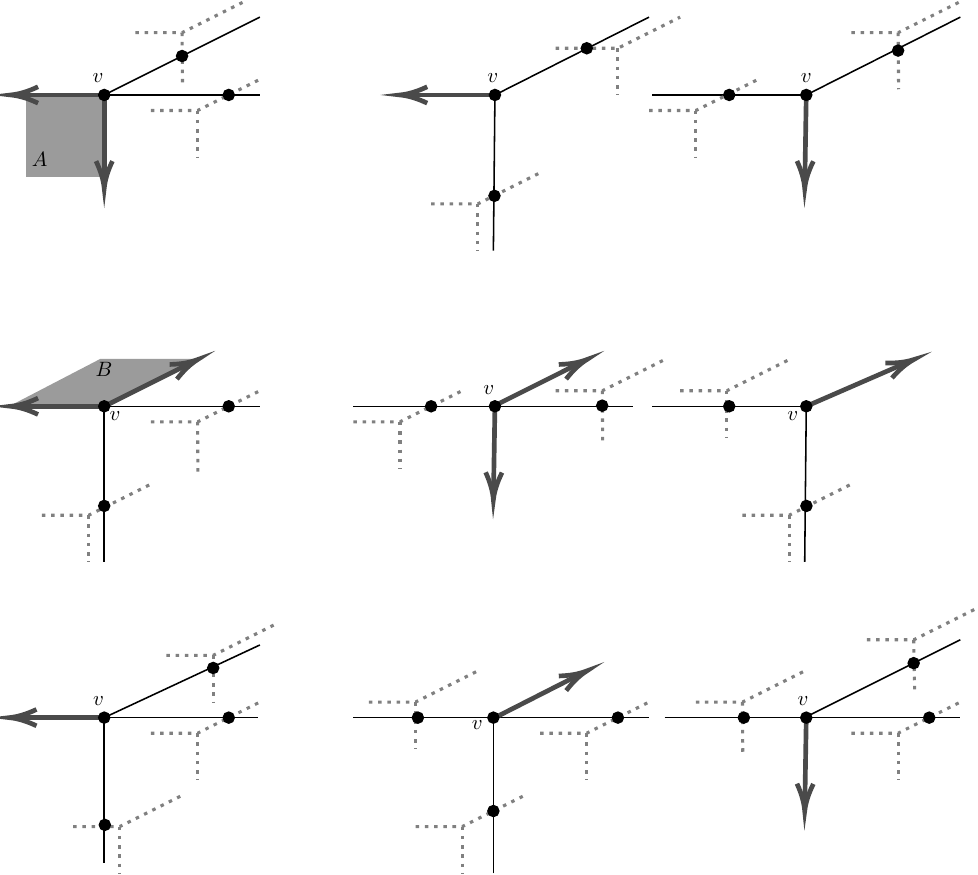}
\caption{ If there are two or three nonparallel ends adjacent to a vertex $v$ in $B$ and on each end there is a {marked point} caused by a multi-line condition, then we should find the intersection area between two or three areas that are shown in Figure \ref{Ends}. For instance, for the top left graph, we intersect the two dark grey areas in the left column of the table in Figure \ref{Ends}. The possible directions of unbounded movement for $v$ that maintain the intersections are shown with dark grey areas trapped between grey arrows. 
}
\label{9}
\end{figure}

\end{enumerate}
    
\end{observation}

\begin{lemma}
    
\label{1Ray}
  Let $ B $ be a movable component of a tropical stable map $ C = (\Gamma, x_{\underline{m}}, \Tilde{e}_{\underline{d}}, h)  $ contributing to ${\mathcal{N}^{0}_{\Delta_{\mathbb{F}_{r}}(a,b, w_{\underline{d}})}(p_{\underline{n}}, L_{\underline{k}}, {\lambda}_{\underline{l-1}}, {\lambda}'_{l})}$, where $|{\lambda}'_{l}|$ can be very large.
Let $ v_1 \geq v_2 \geq \cdots \geq v_n $ be a maximal chain of vertices in the movable component of $ \Tilde{\Gamma}$ with respect to some $ \Omega(H) $. If $H$ is a 1-ray special open half-plane, then the maximal chain is of length one.

\end{lemma}
    
\begin{proof}
    Let $ v_1 \geq v_2 \geq \cdots \geq v_n $ be a maximal chain induced by a 1-ray special open half-plane $ H $. Then because of the balancing, $ v_n $ has to be connected to the only existing end in $ H_{v_n} $, we call it $ e $: Otherwise, $ v_n $ needs to be adjacent to another vertex $\Tilde{v}$ in $ H_{v_n} $, either $\Tilde{v}$ is in the movable component, which is in contradiction with the fact that we chose a maximal chain $ v_1 \geq v_2 \geq \cdots \geq v_n $ at the beginning, or it is in the fixed component. Lets call $e'$ the bounded edge connecting $ v_{n}$ and $\Tilde{v}$, then $  \langle e' \rangle + v_n =   \langle b_n \rangle + v_n $. That implies $b_n +v_n \notin \Bar{H}_{v_n} $, which is again a contradiction.  Therefore, $ v_n $ is connected to the only existing end in $ H_{v_n} $.

We claim that $ v_n $ is a 3-valent vertex of type I. Since we are considering 1-ray special half-planes, the end $e$ is not a marked right end. Therefore if $ \val(v_n) > 3 $, then (we need to switch from $\Tilde{\Gamma}$ to $\Gamma$) since $ v_n $ satisfies at least one cross-ratio condition, based on Remark \ref{Gol18} there must be one {marked point} $ v $ caused by a multi-line condition on $ e $.
The possible directions of movement of $ v $ are in a standard direction. In part (1) of Observation \ref{oneEnd}, we see that any unbounded direction of movement for a vertex adjacent to an end in the standard direction, with a marked point caused by a multi-line condition on it, will always fall outside of the 1-ray special half-plane containing the end $e$. This is a contradiction.  Notice that based on Observation \ref{f4}, the direction of movement of $ v $ cannot be aligned with the direction of $ e $.

		

    Based on Definition \ref{IIII}, if $ v_n $ is not connected to the fixed component, then there is a multi-line condition at $ v_n $ that defines the direction of movement of $ v_n $. This direction cannot be aligned with the end $ e $, and $ H $ is a 1-ray open half-plane, so again it is a contradiction that the direction of movement of $ v_n $ does not live in $ H_{v_n} $.

    So $ v_n $ is a 3-valent type I vertex, and for any choice of the end $ e $ adjacent to $ v_n $ and $ b_n $, we can find all possible directions of two other bounded edges adjacent to $ v_n $. In all cases, one of the bounded edges adjacent to $ v_n $ lives on the boundary of $ H_{v_n} $. But if $n \geq 2$, this is in contradiction with the fact 
that $v_{n-1}\geq v_n$ with respect to this $H$. Therefore, the length of the maximal chain in this case is equal to one.

\end{proof}

\begin{lemma}\label{4cases}
 Let $ B $ be a movable component of a tropical stable map contributing to ${\mathcal{N}^{0}_{\Delta_{\mathbb{F}_{r}}(a,b, w_{\underline{d}})}(p_{\underline{n}}, L_{\underline{k}}, {\lambda}_{\underline{l-1}}, {\lambda}'_{l})}$, where $|{\lambda}'_{l}|$ can be very large. If $v$ is an ordinary vertex in the movable component $B$, then:
\begin{enumerate}
    \item Either $v$ is adjacent to an end in $(r,1)$ or $(0,-1)$ direction, then $v$ is a 3-valent vertex of type I,
    \item or $v$ is not adjacent to any ends, then the direction of movement of $v$ is $(1,0)$ and it is of type I,
    
     \item or $v$ is a vertex of valence higher than $3$ and it is adjacent to marked right ends or ends in $(-1,0)$ direction or both, then the direction of movement of $v$ is $(1,0)$ and it is of type I.
\end{enumerate}

\end{lemma}

\begin{proof}
  We prove each case separately:
   \begin{itemize}
       \item If $v$ is adjacent to an end in $(r,1)$ or $(0,-1)$ direction, we claim that $v$ must be 3-valent. First, consider the case when $v$ is adjacent to the end in $(r,1)$ direction. If $\val(v)>3$, then there is a marked point caused by a multi-line condition on this end. As shown in Figure \ref{Ends}, all possible directions of unbounded movement of $v$ live inside of 1-ray special half-planes. If $b$, the direction of movement of $v$, lives in $H^{3}_v$ then due to Lemma \ref{1Ray} we cannot have any vertex adjacent to $v$ in $H^{3}_v$. Balancing condition requires $v$ to be adjacent to the end in $(-1,0)$ direction in $H^{3}_v$ and this is a contradiction with the fact that now as it is shown in the first row and third column of Figure \ref{9} the unbounded direction of movement of $v$ has to be $b = (0,-1)$ but this direction is not inside of $H^{3}_v$. Similarly, if we assume $b \in H^{2}_v$, then we will be in the case in the first row and second column in Figure \ref{9}, which is again a contradiction with the assumption $b \in H^{2}_v$.
 In the same way, if we consider the case when $v$ is adjacent to the end in $(0,-1)$ direction, then as it is shown in Figure \ref{Ends}, all possible directions of unbounded movement of $v$ either lives in $H^{1}_v$ or $H^{3}_v$. Again, with similar reasoning, we get a contradiction, so $v$ is a 3-valent ordinary vertex and therefore, automatically, it is of type I, and this proves the case (1).

   \item If there is no end adjacent to $v$, we claim that the direction of movement of $v$ is $b=(1,0)$. 
Assume not, then either $b=(-1,0)$ in $H^{3}_v$, or $b$ is in $H^{1}_v$, or $H^{2}_v$. In all three cases, due to Lemma \ref{1Ray}, there cannot be any other vertex adjacent to $v$ in that 1-ray hyperplane containing $b$. As we assumed, $v$ is not adjacent to any end, and moreover, $v$ is not adjacent to any other vertex in the 1-ray half-plane containing $b$ ($H^{3}_v$, or $H^{1}_v$, or $H^{2}_v$). This is in contradiction with the balancing condition at $v$.
So $b$ must be $(1,0)$ and accordingly, we can say $v$ is of type I.
This proves case (2).

\item Where $v$ is a vertex of valence higher than 3 and it is adjacent to marked right ends or ends of $(-1,0)$ direction or both, we claim $b = (1,0)$. Assume not, then either $b= (-1,0)$ or $b \in H^{1}_v$ or $b \in H^{2}_v$. Due to Lemma \ref{1Ray}, there cannot be any other vertex adjacent to $v$ in the 1-ray hyperplane containing $b$. Therefore, $v$ must be adjacent to an end in one of the standard directions. As we assumed, $\val(v)>3$, so there is a marked point caused by a multi-line condition on the end in the standard direction adjacent to $v$. This results in a contradiction because in each case ($b=(-1,0) \in H^{3}$ or $b \in H^{1}_v$ or $b \in H^{2}_v$), the unbounded direction of movement of $v$ (as described in Figure \ref{Ends} for a vertex adjacent to an end in standard direction) doesn't live inside of the 1-ray half-plane that we chose first. So $b= (1,0)$ and automatically $v$ is of type I, and this proves the case (3).

\end{itemize}
\end{proof}

\begin{lemma}
    
\label{NoRay}
Let $ C $ be a tropical stable map as described in Remark \ref{B}. If $v$ and $v'$ are a pair of adjacent comparable vertices in $B$, then a special 2-ray half-plane $H$ exists such that $v$ and $v'$ are comparable with respect to $H$.

\end{lemma}
\begin{proof}
Based on Lemma \ref{1Ray}, we cannot have any pair of adjacent vertices in $B$ comparable with a 1-ray special open half-plane. If $v$ and $v'$ are a pair of adjacent vertices that are not comparable with a 2-ray special open half-plane either, then the bounded edge adjacent to these vertices $ e $, and  $ b $, the direction of movement of $v$, do not live in any special open half-plane together. Therefore, $e, b$ live in one of the opposite cones in Figure \ref{Fi}. 
Notice that $b'$ must also live in the opposite cone from the cone that $e$ lives in; otherwise, $ v, v'$ will be comparable with respect to a special open half-plane.
\begin{figure}[h!]
\centering
\includegraphics[scale = 0.8]{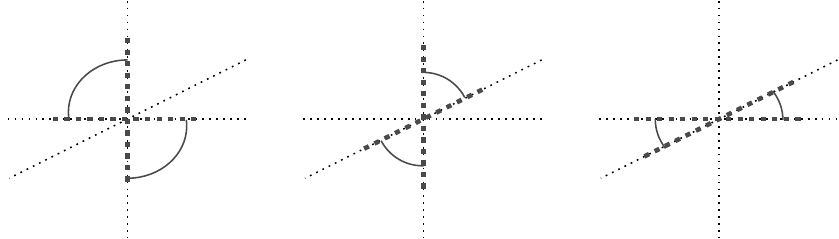}
\caption{In this figure, we can see all three possible opposite cones with respect to the three lines $ \langle (1,0) \rangle $, $ \langle (0,1) \rangle $, and $ \langle (r,1) \rangle $. }
\label{Fi}
\end{figure}
Either $b$ or $b'$ has to live in the intersection of two 1-ray special half-planes because at least one of the opposite cones in each of the three cases, as it is shown in Figure \ref{conecases}, lives in the intersection of two 1-ray half-planes. Assume without restriction that $b$ is in the intersection of two 1-ray half-planes, then there cannot be any other vertex in the movable component adjacent to $v$ in those half-planes; otherwise, it is in contradiction with Lemma \ref{1Ray}. Moreover, $v$ cannot be adjacent to a vertex in the fixed component in the same half-plane as the one containing its direction of movement. Therefore, $v$ is adjacent to two ends, so we are in one of the cases depicted in Figure \ref{conecases}. In all cases, $v$ is a vertex adjacent to two ends and another vertex $v'$ (the one we started with).

Since $B$ cannot have more than one-dimensional movement, the vertex $v$ must have a valence greater than 3. However, this would lead to marked points caused by multi-line conditions on the ends adjacent to $v$. Consequently, the unbounded direction of movement for $v$ must be consistent with the directions described in Observation \ref{9}. In these scenarios, the initial direction of movement $b$ is not an unbounded direction of movement, which is a contradiction.

\begin{figure}[h!]
\centering
\includegraphics[scale = 0.8]{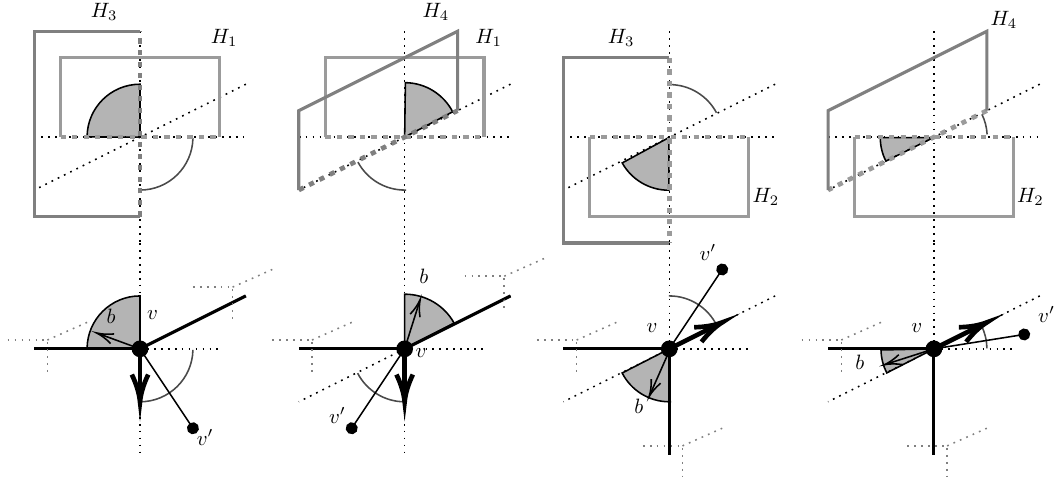}
\caption{In the above picture, we can see all possible cases when the direction of movement of $v$ lives in the cone, which is at the intersection of two 1-ray half-planes. In the picture below, in each case, the vertex $v$ with all of its adjacent edges and its unbounded direction of movement (the thickened black arrow), the initial chosen direction of movement $b$, is depicted. As it is shown in the picture, in all cases, the unbounded direction of movement is outside of the grey cone we started with. }
\label{conecases}
\end{figure}

\end{proof}

\begin{remark}
    \label{v1vn}
    Based on the case (1) of Lemma \ref{4cases}, a vertex in $B$ adjacent to an end in $(r,1)$ or $(0,-1)$ direction must be 3-valent of type I. Due to Remark \ref{B}, the number of vertices in the movable component $\#V>1$, so such 3-valent vertices are adjacent to one end and two bounded edges, such that one of the bounded edges is adjacent to a vertex in the fixed component. Such vertices can only be the first or the last vertex of a maximal chain of vertices in $B$. In Lemmas \ref{2Ray} and \ref{2rayn}, we prove that the movement direction of vertices adjacent to ends in the \( (r,1) \) or \( (0,-1) \) direction must be \( (1,0) \).

\end{remark}

\begin{lemma}
 \label{2Ray}

Let $ v_1 \geq v_2 \geq \cdots \geq v_n $ be a maximal chain of vertices in a movable component $B$ as described in Remark \ref{B}. If $v_n$ is adjacent to an end of direction $(r,1)$ or $(0,-1)$, the direction of movement of $v_n$ is $b_n = (1, 0)$.
 
\end{lemma}
\begin{proof}
    
Since $v_n$ is adjacent to an end of direction $(r,1)$ or $(0,-1)$, it is a 3-valent vertex of type I, and $v_n$ must be adjacent to two bounded edges. We call the two vertices adjacent to $v_n$ by $v_{n-1}$ and $v$ (where $v$ is a vertex in a fixed component) and denote the primitive direction vector of these bounded edges with $(x_{i}, y_i)$ for $i=1,2$ (see Figure \ref{xyv}).  Based on Lemma \ref{NoRay}, $v_n$ and $v_{n-1}$ are comparable with respect to a 2-ray special open half-plane. Now we want to determine the vectors $(x_{i}, y_i)$ in each case:
\begin{figure}[h!]
\centering
\includegraphics[scale = 0.8]{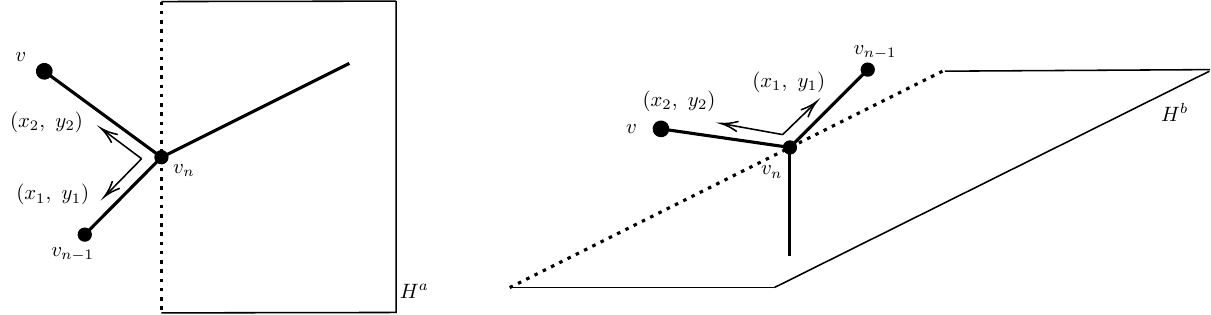}
\caption{The diagram illustrates the primitive direction vectors of the two bounded edges adjacent to $ v_n $. }
\label{xyv}
\end{figure}

\begin{itemize}
    \item  Assume $v_n$ is adjacent to an end in $(r,1)$ direction. We claim that both bounded edges adjacent to $v_n$ must be in $H^{3}_{v_n}$. As discussed in Remark \ref{v1vn}, the vertex \( v_n \) is a 3-valent vertex of type I. Given that the direction of one of the edges adjacent to \( v_n \) is known, determining the direction of any one of the remaining adjacent edges uniquely determines the direction of the third.
 The vertex $v_{n-1}$ cannot live in $ \overline{H^{a}_{v_n} \cap H^{b}_{v_n}}$ otherwise, if $ v_{n-1} \in {H^{a}_{v_n} \cap H^{b}_{v_n}}$ then $v \in H^{4}_{v_n} \cap H^{3}_{v_n} $ (due to balancing at $v_n$) and therefore $b_n \in  H^{a}_{v_n} \cap H^{b}_{v_n}$, so $v_{n-1}$ is greater than $v_n$ with respect to both 2-ray special half-planes $H^a_{v_n}, H^b_{v_n}$ and this is in contradiction with the assumption that $v_n$ is the last vertex of a maximal chain with respect to a 2-ray special half-plane.
If $v_{n-1}$ is adjacent to $v_n $ along an edge in direction $(0,-1)$, then $v_{n-1}$ cannot be greater than $v_n$ with respect to any 2-ray special open half-plane, which is a contradiction. Moreover, $v_{n-1}$ cannot be adjacent to $v_n $ along an edge in direction $(r,1)$, so $v_{n-1}$ is not in $\overline{H^{a}_{v_n} \cap H^{b}_{v_n}}$. If $v_{n-1}$ lives in $H^{a}_{v_n}\setminus \overline{H^{a}_{v_n} \cap H^{b}_{v_n}}$, as shown in Figure \ref{vnvn}, $v_{n-1}, b_{n} \in H^1_{v_n} $, \( v_n \) and \( v_{n-1} \) are comparable with respect to the 1-ray special half-plane $ H^1_{v_n} $, which contradicts Lemma \ref{1Ray}.

\begin{figure}[h!]
\centering
\includegraphics[scale = 0.9]{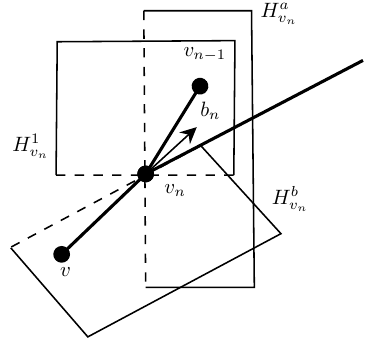}
\caption{This diagram shows that \( v_{n-1} \)  lies in \( H^a_{v_n} \setminus \overline{H^{a}_{v_n} \cap H^{b}_{v_n}} \), and \( v_n \) is adjacent to an end in the \( (r,1) \) direction. Since $v_n$ is 3-valent of type I, one can uniquely find that $b_{n} \in H^1_{v_n}$. So \( v_n \) and \( v_{n-1} \) are comparable with respect to the 1-ray special half-plane $ H^1_{v_n}. $
}  
\label{vnvn}  
\end{figure}
If either of $v$ or $v_{n-1}$ is adjacent to $v_n $ along an edge in direction $(0,1)$, then \( v_n \) and \( v_{n-1} \) are comparable with respect to a 1-ray special half-plane and again this contradicts Lemma \ref{1Ray}. 
If $v$ is adjacent to $v_n $ along an edge in direction $(0,-1)$, then due to balancing $v \in H^4_{v_n}$ and since $b_n = (0,1)$, we have $v_n \geq v_{n-1}$ with respect to $H^4_{v_n}$ which is again a contradiction.
  
So far we know that $v_{n-1}$ is in $ H^{3}_{v_n} $. If $v \in H^{a}_{v_n}$, then $v_{n-1}, b_{n} \in H^3_{v_n} $, so \( v_n \) and \( v_{n-1} \) are comparable with respect to the 1-ray special half-plane $ H^3_{v_n} $, which contradicts Lemma \ref{1Ray}. Therefore, both bounded edges adjacent to $v_n$ are in $H^{3}_{v_n}$. If both of the $y_i$s are non-zero (see Figure \ref{xyv}), $v_n$ is comparable with $v_{n-1}$ with respect to a 1-ray special half-plane. Therefore, one of the $y_i$s must be zero. More precisely, we must have $y_2 = 0$, otherwise, if $y_1 = 0$ then $ v_{n} \geq v_{n-1} $ with respect to $H^{4}_{v_n}$, which is in contradiction with Lemma \ref{1Ray}. Therefore, $y_2 = 0$ so $b_n =(1,0)$, and the exact direction of the adjacent edges to $ v_n $ are shown in Figure \ref{fHa}. 
    
\begin{figure}[h!]
\centering
\includegraphics[scale = 0.9]{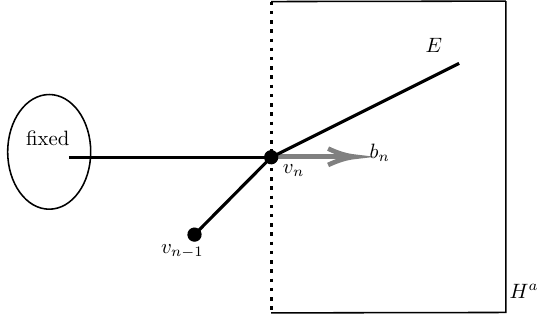}
\caption{In this diagram, $v_n$ is a 3-valent vertex of type I, which represents the last vertex of a maximal chain. Since $v_n$ is of type I, its movement direction is aligned (in the direction of moving away from the fixed component) with the bounded edge connecting $v_n$ to the fixed component. Additionally, since we know that $v_n$ is 3-valent, the direction of the two bounded edges adjacent to $v_n$ can be determined through a simple computation. These directions are $(-w,0)$ and $(w-r,-1)$ where $w$ is the weight of the bounded edge connecting $v_n$ to a vertex in the fixed component and $b_n =(1,0)$. In this figure $w=1$ and $r=2$. }
\label{fHa}
\end{figure}

    \item  Assume $v_n$ is a vertex adjacent to an end in $(0,-1)$ direction. Due to the balancing condition at $ v_n $, each of the two bounded edges must lie in either $H^{3}_{v_n}$ or $H^{a}_{v_n}$. Also, we know that $x_{i}\neq 0 $ for $i=1,2$, moreover $v_n$ is 3-valent and adjacent to an end in $(0,-1)$ direction, so $x_1=-x_2$. If one of the $y_i$ is negative, then $v_n$ is comparable with $v_{n-1}$ with respect to a 1-ray special open half-plane. Therefore $y_i \geq 0$ and we also know that $y_{1}+y_{2}=1$, so one of the $y_i$ is equal to zero. The exact direction of the adjacent edges to $ v_n $ are shown in Figure \ref{fHb}. The direction of movement of $ v_n $ is $ (1, 0) $.

   \begin{figure}[h!]
\centering
\includegraphics[scale = 0.9]{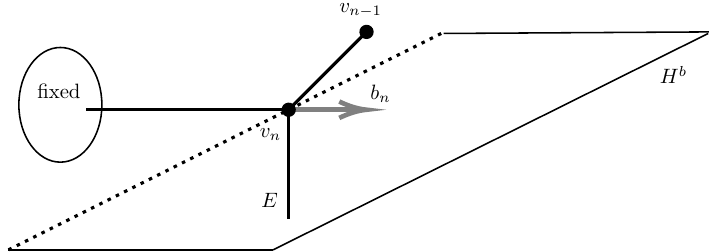}
\caption{In this diagram, $v_n$ is a 3-valent vertex of type I, which represents the last vertex of a maximal chain with respect to $H^b$. Since $v_n$ is of type I, its movement direction is aligned with the bounded edge connecting $v_n$ to the fixed component. The direction of the two bounded edges adjacent to $v_n$ are $(-w,0)$ and $(w,1)$, where $w$ is the weight of the bounded edge connecting $v_n$ to a vertex in the fixed component. }
\label{fHb}
\end{figure}

\end{itemize}

\end{proof}

\begin{lemma}\label{2rayn}
Let $ v_1 \geq v_2 \geq \cdots \geq v_n $ be a maximal chain of vertices in a movable component $B$ as described in Remark \ref{B}. If $v_1$ is adjacent to an end of direction $(r,1)$ or $(0,-1)$, the direction of movement of $v_1$ is $b_1 = (1, 0)$.
\end{lemma}
\begin{proof}

As discussed in Remark \ref{v1vn}, if $v_1$ is adjacent to an end of direction $(r,1)$ or $(0,-1)$, it is a 3-valent ordinary vertex of type I in $B$. 
Due to Remark \ref{v1vn}, $v_1$ is adjacent to another vertex in the movable component $v_2$, and it is also adjacent to a vertex in the fixed component, which we call $v$. 

 \begin{figure}[h!]
\centering
\includegraphics[scale = 0.9]{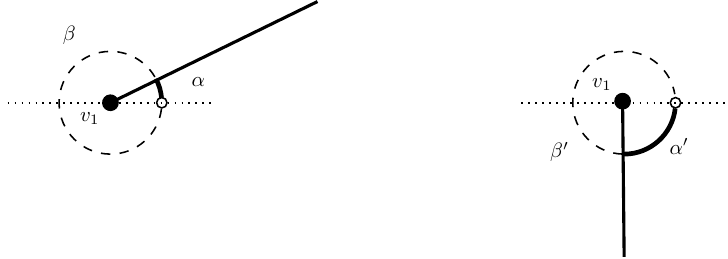}
\caption{In this diagram, $v_1$ is a 3-valent vertex of type I which appears in case (1) of Lemma \ref{4cases}, and its direction of movement can occur either in the open cone $\beta$ (or $\beta'$ if $v_1$ is adjacent to an end in $(0,-1)$), or in the half-open cone $\alpha$ (or $\alpha'$), or $ b_1 = (1,0) $. Here half-open cone means that $(r,1)\in \alpha$ but $(1,0)\notin \alpha$ (similarly for $\alpha'$, $(1,0)\notin \alpha'$). }
\label{abg}
\end{figure}

If $b_1 = (1,0)$, then the proof is complete. So we assume $b_1 \neq (1,0)$ and proceed to show that any other direction of movement for $v_1$ leads to a contradiction. As shown in Figure \ref{abg}, when $v_1$ is adjacent to an end in the direction $(r,1)$, $b_1$ can lie in the open cone $\beta$ (or $\beta'$ if $v_1$ is adjacent to an end in the direction $(0,-1)$), or in the half-open cone $\alpha$ (similarly $\alpha'$).

\begin{itemize}

    \item 
If $b_1$ lies in either of the cones $\beta$ or $\beta'$, then due to the balancing condition, $v_2$ must lie in a 1-ray half-plane that contains $b_1$. Therefore, $v_1$ and $v_2$ are comparable with respect to a 1-ray half-plane, which contradicts Lemma \ref{1Ray}.

    \item If $b_1$ lives in the half-open cones $\alpha$ or $\alpha'$ (depending on the direction of the end adjacent to it), on the one hand, we know $v_2$ must live in a 2-ray half-plane containing $b_1$, and on the other hand $v_2$ cannot be in a 1-ray half-plane containing $b_1$. Therefore, $v_2$ can live in either of the open cones $\theta$ or $\theta'$  or the primitive direction of the bounded edge adjacent to $v_1, v_2$ is $(0,-1)$ when $v_1$ is adjacent to an end in $(r,1)$ direction (or $(r,1)$ when $v_1$ is adjacent to an end in $(0,-1)$ direction), see Figure \ref{ttp}.

 \begin{figure}[h!]
\centering
\includegraphics[scale = 0.9]{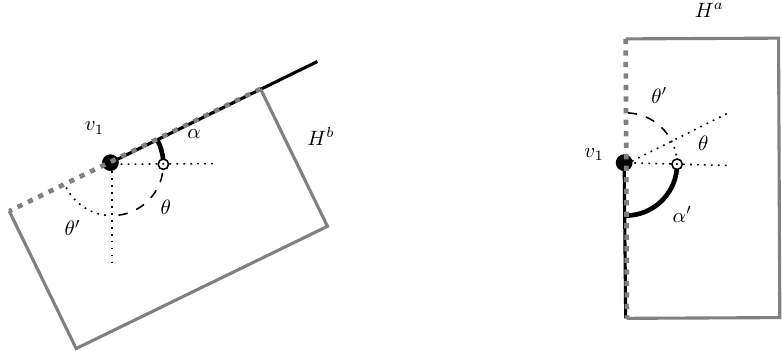}
\caption{In this diagram, $v_1$ is a 3-valent vertex of type I, which appears in case (1) of Lemma \ref{4cases}. Its direction of movement lives in the half-open cone $\alpha$ (similarly $\alpha'$). Since $v_2$ and $b_1$ cannot live in the same 1-ray half-plane, $v_2$ must live in either of the open cones $\theta$ or $\theta'$ or the primitive direction of the bounded edge adjacent to $v_1, v_2$ is $(0,-1)$ in the left picture (or $(r,1)$ in the right picture). }
\label{ttp}
\end{figure}

    \begin{itemize}
        \item If $v_2$ lives in $\theta'$, then since $b_2 = (1,0)$, we have $v_2 \geq v_1$ with respect to a special 2-ray half-plane. In the left picture of Figure \ref{ttp}, this 2-ray half-plane is $H^a_{v_2}$, and in the right picture, it is $H^b_{v_2}$. So $v_1$ is the last vertex of a maximal chain with respect to a special 2-ray half-plane and therefore, due to Lemma \ref{2Ray} $b_1 = (1,0)$, but this is in contradiction to our assumption earlier that $b_1$ is in the half-open cone $\alpha$ (or $\alpha'$).

        \item 
  If $v_2$ lives in $\theta$ or the primitive direction of the bounded edge adjacent to $v_1, v_2$ is $(0,-1)$, when $v_1$ is adjacent to an end in $(r,1)$ direction, then since we know that $b_1$ is in the half-open cone $\alpha$, $v \in H^2_{v_1}$, the direction of both bounded edges adjacent to $v_1$ lives in $H^{2}_{v_1}$ and therefore $y_1, y_2 <0$; see Figure \ref{alphatp}. 

\begin{figure}[h!]
\centering
\includegraphics[scale = 0.9]{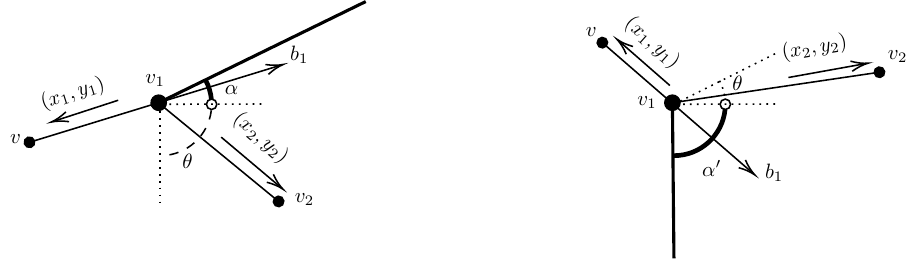}
\caption{In this diagram, $v_1$ is a 3-valent vertex of type I, which appears in case (1) of Lemma \ref{4cases}. Its direction of movement lives in the half-open cone $\alpha$ (similarly $\alpha'$). Moreover, $v_2$ lives in cones $\theta$. }
\label{alphatp}
\end{figure}

We have $y_1, y_2 <0$, but the balancing condition at $v_1$ implies that $y_1 + y_2 +1 =0$ which is a contradiction.
A similar argument holds for the case when $v_2$ lives in $\theta$ and is adjacent to an end in $(0,-1)$ directions. 


    \end{itemize}

\end{itemize}

\end{proof}

\begin{remark} \label{bigstring}
    Let $ C $ be a tropical stable map as described in Remark \ref{B}. Based on Lemma \ref{1Ray}, there cannot be any two vertices in $B$ that are comparable with respect to a 1-ray special open half-plane. Also, in Lemma \ref{NoRay} we proved that any pair of vertices in $B$ is comparable with respect to a 2-ray special open half-plane. 
    If $ v_1 \geq \cdots \geq v_n $ is a maximal chain with respect to a 2-ray special half plane with length greater than 1, then due to Lemmas \ref{4cases}, \ref{2Ray} and, \ref{2rayn} the direction of movement of all of the vertices in this maximal chain are $(1,0)$ and moreover, each vertex adjacent to an end in $(r,1)$ or $(0,-1)$ direction is 3-valent and each vertex adjacent to an end of primitive directions $(1,0)$ or $(-1,0)$ has valence higher than 3. So $ v_1 \geq \cdots \geq v_n $ is a movable branch $\mathcal{S}$ defined in \ref{DefString}.
In Remark \ref{B}, we assumed that the set of vertices in this movable component $ V $ contains more than one vertex. In Lemma \ref{1ver}, we study the case where $\#V=1$. 
\end{remark}
\begin{lemma}\label{1ver}
       Let $ B $ be a movable component of a tropical stable map contributing to ${\mathcal{N}^{0}_{\Delta_{\mathbb{F}_{r}}(a,b, w_{\underline{d}})}(p_{\underline{n}}, L_{\underline{k}}, {\lambda}_{\underline{l-1}}, {\lambda}'_{l})}$, where $|{\lambda}'_{l}|$ can be very large.
    If the set of vertices in this movable component $ V $ contains only one ordinary vertex, then $B$ is a movable branch $\mathcal{S}$ with one vertex as it is described in the last part of Definition \ref{DefString}.
\end{lemma}
\begin{proof}
 Assume $ v $ is the only ordinary vertex in $ B $, so it must be connected to the fixed component of the tropical stable map by a bounded edge.
 Therefore, there can only be some ends adjacent to $ v $ and either all of the edges adjacent to $v$ are parallel or not. Assume that ${\lambda}'_{l} = (t_{\alpha}t_{\beta} \,|\, t_{\eta}t_{\gamma})$. Since the length of $|{\lambda}'_{l}|$ varies depending on the movement of $v$, one pair of marked ends must belong to the movable component, while the other pair remains in the fixed part. Without loss of generality, let us assume that $t_{\alpha}$ and $t_{\beta}$ are in the movable component.

 Assume all of the adjacent edges to $v$ are parallel, if $v$ is adjacent to an end in $(r,1)$ or $(0,-1)$ direction, based on case (1) of Lemma \ref{4cases} $v$ is a 3-valent vertex. On the one hand, all of the adjacent edges to $v$ have to be parallel, so $v$ is a vertex adjacent to two ends in $(r,1)$ direction (similarly in $(0,-1)$ direction) and its direction of movement is $(r,1)$ (correspondingly $(0,-1)$). On the other hand, the length of $ |{\lambda}'_{l}| $ depends on the movement of $ v $, so we claim there must be at least one marked point caused by a multi-line condition on one of the ends adjacent to $v$. Notice that the only {marked end}s we can have in $ B $ are the contracted ends with the label in $ \underline{k} $ or $ \underline{f} $ or marked right ends. At least one of the $\alpha, \beta$ must be in $ \underline{k} $ otherwise we have two marked $f$-points on $v$, and in this case, these $f$-points can freely move on the two adjacent ends to $v$ and the cross-ratio condition ${\lambda}'_{l}$ still holds. But this contradicts the fact that $C$ has a 1-dimensional movement.

 Having a vertex $v$ adjacent to an end in $(r,1)$ direction with a marked point caused by a multi-line condition on it has been studied in Figure \ref{Ends}, and the direction of unbounded movement of $v$ cannot be $(r,1)$. Similarly, for the case where the direction of the end is in $(0,-1)$ direction. So all of the ends adjacent to $v$ are in $(1,0)$ and $(-1,0)$ directions, and this is one type of movable branch described in the last case of Definition \ref{DefString}.

If the adjacent edges to $v$ are not parallel, then $v$ must be adjacent to at least one end in one of the directions $(r,1)$ or $(0,-1)$. Again, due to the case (1) of Lemma \ref{4cases}, we know that $\val(v)=3$. Since the length of $ |{\lambda}'_{l}| $ depends on the movement of $ v $, we claim there must be at least one marked point caused by a multi-line condition on one of the ends adjacent to $v$ (because in some cases we can have marked right ends or marked $f$-points).

	\begin{figure}[h!]
		
		\includegraphics[scale=0.8]{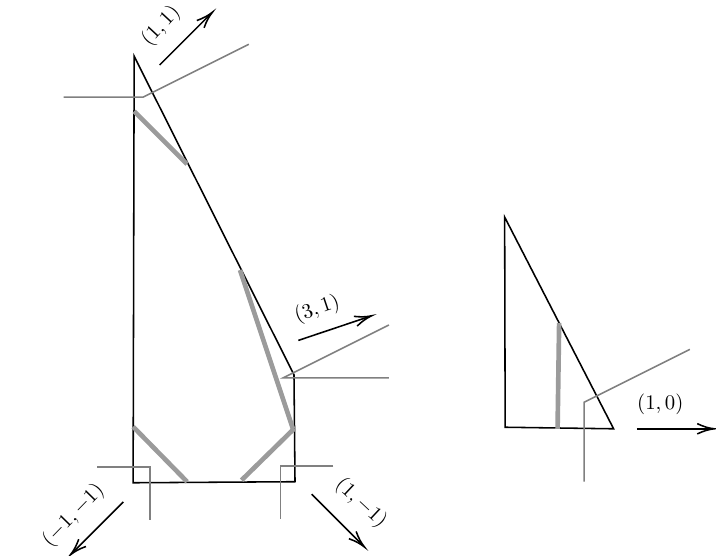}
		\centering
		\caption{If there is only one ordinary 3-valent vertex in $B$ that satisfies no cross-ratio conditions, then all of the possible 3-valent vertices adjacent to two ends are those shown on each corner of the Newton polygon dual to the image of a tropical stable map. The movement direction of each vertex must be away from the fixed component (which is the rest of the tropical stable map) and aligned with the bounded edge connecting the vertex to the fixed component. Therefore, it is easy to find the movement directions, which are written next to each case in the figure. }
		\label{P}
	\end{figure}
    
So $v$ is an ordinary 3-valent vertex, and all the possible cases that can occur are shown in Figure \ref{P}. As we can see in Figure \ref{P}, the direction of movement of $ v $ is not in the directions that are shown in Figure \ref{9} or \ref{Ends} which $ v $ can move unboundedly in those directions and maintain the intersection with the multi-line conditions. Hence, $ v $ cannot move unboundedly, but this is in contradiction with $|\lambda'_l|$ being very long.

\end{proof}

Now, we have all the tools to present the proof of Proposition \ref{*}:

\begin{proof}[\textbf{Proof of Proposition \ref{*}}]

	We show that the set of all points $ \ft_{{\lambda}'_{l}}(C) \in \mathcal{M}_{4} $ is bounded, if $  C = (\Gamma, x_{\underline{m}}, \Tilde{e}_{\underline{d}}, h) $ has no contracted bounded edge or movable branch $ \mathcal{S} $ and runs over all tropical stable maps  contributing to ${ \mathcal{N}^{0}_{\Delta_{\mathbb{F}_{r}}(a,b, w_{\underline{d}})}( p_{\underline{n}}, L_{\underline{k}}, {\lambda}_{\underline{l-1}} ) }$. Based on Remark \ref{-1} the set of such stable maps is a 1-dimensional family in ${ \mathcal{M}_{0,m+d}(\mathbb{R}^{2}, \Delta_{\mathbb{F}_{r}}(a,b, w_{\underline{d}}))} $. Without loss of generality, we restrict ourselves to a fixed combinatorial type in ${ \mathcal{M}_{0,m+d}(\mathbb{R}^{2}, \Delta_{\mathbb{F}_{r}}(a,b, w_{\underline{d}}))} $.
	That 1-dimensional family of stable maps is denoted by $ Y $:

\begin{center}
    $ Y = \prod_{i\in \underline{n}}\ev^{*}_i (p_i ) \cdot \prod_{j\in \underline{k}}\ev^{*}_j (L_j ) \cdot \prod_{\iota\in \underline{l-1}}ft^{*}_{{\lambda}_{\iota}} (0) \cdot \mathcal{M}_{0,m+d}(\mathbb{R}^{2}, \Delta_{\mathbb{F}_{r}}(a,b, w_{\underline{d}})). $
\end{center}
 
	If we choose a tropical stable map $ C' $ in $ Y $, it has a movable component such that we can move it and obtain all other tropical stable maps in $ Y $. As it is mentioned in the Remark \ref{B} this movable component is called $ B $. Since $ Y $ is a 1-dimensional family, $ B $ is connected, and $ C' $ has no other movable component.

 We assume that $n \geq 1$; therefore, $C'$ has at least one fixed component, and $B$ is connected to this fixed component by at least one bounded edge. If $B$ is as in Figure~\ref{fBf}, then we cannot obtain large $|{\lambda}'_{l}|$. Hence, all bounded edges connecting $B$ to the fixed components of $C'$ lie on one side of $B$. In other words, if there are edges adjacent to $B$ on opposite sides, as in Figure~\ref{fBf}, then the deformations of $C'$ with fixed combinatorial type and incidence conditions are bounded on both sides.

	\begin{figure}[h!]
		
		\includegraphics[scale=0.9]{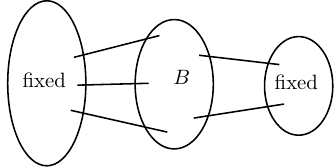}
		\centering
		\caption{In this figure, we depict a rough illustration of a movable component $B$ connected to the fixed components from two opposite sides.}
		\label{fBf}
	\end{figure}

As it is explained in Remark \ref{bigstring}, we have demonstrated that if $ \Tilde { \Gamma } $ denotes the metric tree of the tropical stable map $ C' $ (see Notation \ref{Gati}) which allows an unbounded 1-dimensional movement, then this movable component either contains a single vertex or is a movable branch $ \mathcal{S} $ described in Definition \ref{DefString}. But this is in contradiction with the assumption of $C$ having no movable branch $\mathcal{S}$. Therefore, the movable component $ B $ contains exactly one vertex, and we have no movable branch $\mathcal{S}$.
Based on the contrapositive of Lemma \ref{1ver}, when $C$ has no movable branch $\mathcal{S}$, $ \ft_{{\lambda}'_{l}}(C) \in \mathcal{M}_{4} $ is bounded and this finishes the proof.

\end{proof}


Up to this point, we have assumed that $n \geq 1$. Now, let's consider the case where $n = 0$, meaning that there are no point conditions on the tropical stable maps. We approach this case in a similar manner as we would with tropical stable maps in $ \mathbb{P}^2 $; for more details, see \cite[Lemma 51 and 52, Proposition 53]{G}.


Let $ C $ be a tropical stable map contributing to $ \mathcal{N}^{0}_{\Delta_{\mathbb{F}_{r}}(a,b, w_{\underline{d}})}( L_{\underline{k}}, {\lambda}_{\underline{l}} ) $ where $ |\Delta_{\mathbb{F}_{r}}(a,b, w_{\underline{d}})|+m-1 = k+l$. There is a vertex $ v $ of $ C $ adjacent to at least two contracted ends $ e_{1}, e_2 $ caused by two multi-line conditions. Otherwise, every vertex of $ C $ is adjacent to at most one multi-line condition, so all vertices have a 1-dimensional movement, which is a contradiction.

This vertex $ v $ is of valence higher than 3, otherwise, the only bounded edge adjacent to $ v $ must be contracted too. Therefore, we have a 1-dimensional family again, which is a contradiction, hence $ \val(v) > 3 $.

One of the cross-ratio conditions satisfying in $ v $ includes both $ e_{1}, e_2 $ like $ {\lambda} = \{e_{1}, e_{2}, \beta_{3}, \beta_4  \} $. If this is not the case, then we can resolve the vertex $ v $ as it is shown in figure \ref{resolve}. Obviously, the bounded edge $ e $ arising from the resolution of $ v $ is a contracted bounded edge, so it leads to a 1-dimensional family, which is a contradiction. 

\begin{figure}[h!]
		
		\includegraphics[scale=0.8]{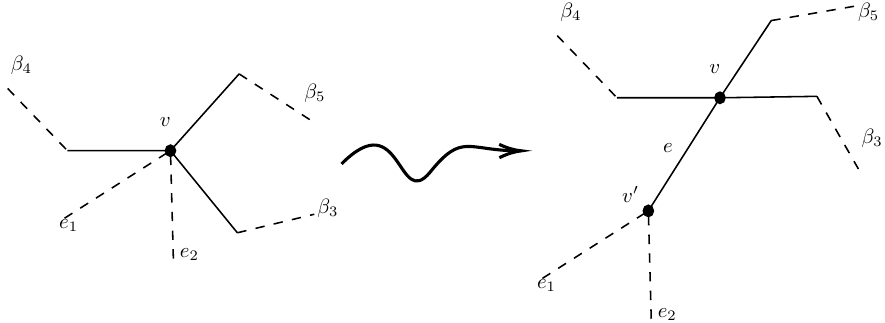}
		\centering
		\caption{Resolving $ v $ with respect to the $ {\lambda}' = (e_{1} e_{2}|\beta_{3} \beta_{4}) $, where $ e_{1}, e_2 $ are contracted ends caused by two multi-line conditions.}
		\label{resolve}
	\end{figure}

With the mentioned arguments, we have established the following corollary.
\begin{corollary}\label{0}
    If $ C $ is a tropical stable map contributing to $ \mathcal{N}^{0}_{\Delta_{\mathbb{F}_{r}}(a,b, w_{\underline{d}})}( L_{\underline{k}}, {\lambda}_{\underline{l}} ) $, then there is a vertex $ v $ of valence greater than 3 which is adjacent to at least two contracted ends $ e_{1}, e_2 $ caused by two multi-line conditions and which satisfies a cross-ratio condition including $ e_{1}, e_2 $. 
\end{corollary}

\begin{lemma}\label{n0}
    We use the notation from \ref{0} and without loss of generality we can assume that $ {\lambda}_{l} = \{e_{1}, e_{2}, \beta_{3}, \beta_4  \} $. Let $ {\lambda}'_l = (e_{1} e_{2} | \beta_{3} \beta_4) $ be a non-degenerate cross-ratio that degenerates to $ {\lambda}_l $. Then every tropical stable map that contributes to $ \mathcal{N}^{0}_{\Delta_{\mathbb{F}_{r}}(a,b, w_{\underline{d}})}( L_{\underline{k}}, {\lambda}_{\underline{l-1}}, {\lambda}'_l ) $ arises from a tropical stable map $ C $ that contributes to $ \mathcal{N}^{0}_{\Delta_{\mathbb{F}_{r}}(a,b, w_{\underline{d}})}( L_{\underline{k}}, {\lambda}_{\underline{l}} ) $ by adding a contracted bounded edge $ e $ to $ C $ as it is shown in figure \ref{resolve}.
\end{lemma}
    
\begin{proof}

Assume that $ C $ is a tropical stable map contributing to $ \mathcal{N}^{0}_{\Delta_{\mathbb{F}_{r}}(a,b, w_{\underline{d}})}( L_{\underline{k}}, {\lambda}_{\underline{l}} ) $. Let $ v $ be the vertex from \ref{0}, if $ \val(v) = 4 $ then the only possible way to resolve $ v $ is to add a bounded edge $ e $ to $ C $ that separates $ e_{1}, e_{2} $ from $ \beta_{3}, \beta_{4} $. In this case, $ e $ is automatically a contracted bounded edge.

If $ \val(v) >5 $ then assume that the edge $e'$ we add by resolving $v$ according to $\lambda'_l$
is not contracted and denote the tropical stable map obtained this way ba $C''$. Denote the
vertex adjacent to $e'$ and $e_1, e_2$ by $v''$. Consider $C''$ as a point in the cycle that arises from dropping the cross-ratio condition $\lambda'_l$. Then $C''$ is in the boundary of a 2-dimensional cell of the same cycle that arises from $C''$ by adding a contracted bounded edge $e$ to $C''$ that
separates $v''$ from $e_1, e_2$. Hence, there is a 2-dimensional cell inside a 1-dimensional cycle, which is a contradiction.

Therefore, by resolving the vertex $ v $, we will always obtain a contracted bounded edge adjacent to $ e_{1}, e_{2} $. Referring to Remark \ref{RationalyEquivalent}, we know that $\mathcal{N}^{0}_{\Delta_{\mathbb{F}_{r}}(a,b, w_{\underline{d}})}( L_{\underline{k}}, {\lambda}_{\underline{l}} ) = \mathcal{N}^{0}_{\Delta_{\mathbb{F}_{r}}(a,b, w_{\underline{d}})}( L_{\underline{k}}, {\lambda}_{\underline{l-1}}, {\lambda}'_l )$, so every tropical stable map that contributes to $ \mathcal{N}^{0}_{\Delta_{\mathbb{F}_{r}}(a,b, w_{\underline{d}})}( L_{\underline{k}}, {\lambda}_{\underline{l-1}}, {\lambda}'_l ) $ arises from a tropical stable map $ C $ that contributes to $ \mathcal{N}^{0}_{\Delta_{\mathbb{F}_{r}}(a,b, w_{\underline{d}})}( L_{\underline{k}}, {\lambda}_{\underline{l}} ) $ by adding a contracted bounded edge $ e $ to $ C $ to resolve the vertex $ v $ as it is shown in figure \ref{resolve}.

\end{proof}

\begin{center}
	
\end{center}

\section{Behavior of split tropical stable maps }\label{Split curves}

Proposition \ref{*} shows that a tropical stable map contributing to $ \mathcal{N}^{0}_{\Delta_{\mathbb{F}_{r}}(a,b, w_{\underline{d}})}( p_{\underline{n}}, L_{\underline{k}}, {\lambda}_{\underline{l-1}}, {\lambda}'_{l} ) $, where $ n \geq 1 $ and $ |{\lambda}'_{l}| $ goes to infinity, either has a contracted bounded edge or a movable branch $ \mathcal{S} $. Also Lemma \ref{n0} shows that a tropical stable map contributing to $ \mathcal{N}^{0}_{\Delta_{\mathbb{F}_{r}}(a,b, w_{\underline{d}})}( L_{\underline{k}}, {\lambda}_{\underline{l-1}}, {\lambda}'_l ) $ when $ |{\lambda}'_{l}| $ goes to infinity, has a contracted bounded edge. 
When splitting tropical stable maps to obtain lower degrees, two general cases arise, which we discuss in the following subsections \ref{4.1} and \ref{4.2}.

\subsection{Splitting over a contracted bounded edge}\label{4.1}
    
After cutting the contracted bounded edge $ e $, two tropical stable maps $ C_i  = (\Gamma_i, x_{\underline{m_i}}, \Tilde{e}_{\underline{d_i}}, h_i) $ of degrees $ \Delta_{ i } = \Delta_{\mathbb{F}_{r}}(a_i,b_i, w_{\underline{d_i}}) $ for $ i = 1, 2 $ are obtained such that $ a = a_1 + a_2 $, and $ w_{\underline{d}} = w_{\underline{d_1}} \cup w_{\underline{d_2}}$. The two pieces of the cut edge $ e $ are denoted as $ e_1 $ and $ e_2 $ on each part of the tropical stable map. The {marked point}s caused by these contracted ends are also referred to as $ e_1 $ and $ e_2 $, while the vertices corresponding to these {marked point}s are denoted by $ v_1 $ and $ v_2 $ (Indeed, $ v_1 $ and $ v_2 $ are the two vertices adjacent to $e$ in $h(\Gamma) $ as it is shown in Figure \ref{cut}). It is important to understand the local effects around $ e_{ i } $ when cutting the contracted bounded edge. Since $ e $ is a contracted bounded edge in $ C $, cutting $ e $ does not affect the balancing around $ v_{ i } $.

\begin{figure}[h!]

\centering
\includegraphics[scale = 0.8]{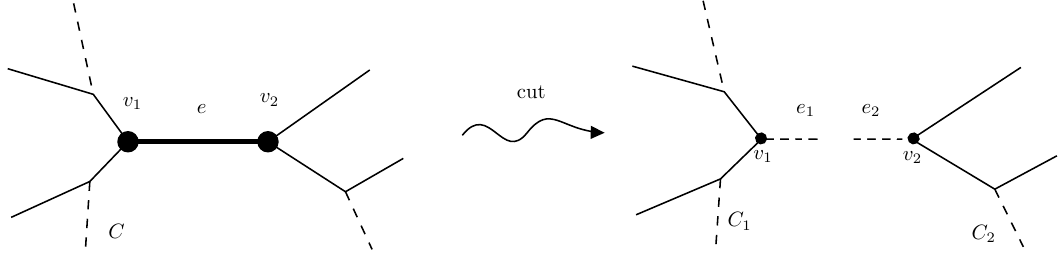}
\caption{In this figure, we can see what happens locally around the cut edge. Notice that both graphs are some parts of the abstract tropical curve. The bounded edge $e$ is the contracted bounded edge so it is cut into $e_1$ and $e_2$.}
\label{cut}
\end{figure}

\begin{remark} \label{v12}
    Let $ C $ be a tropical stable map contributing to $ \mathcal{N}^{0}_{\Delta_{\mathbb{F}_{r}}(a,b, w_{\underline{d}})}( p_{\underline{n}}, L_{\underline{k}}, {\lambda}_{\underline{l-1}}, {\lambda}'_{l} ) $, where $ |{\lambda}'_{l}| $ goes to infinity and it has a contracted bounded edge $e$.
   Before cutting the edge $e$, let us denote the image of $e$ in the plane as $v$.
 After cutting the edge $e$, the vertex $v$ will be distributed into two vertices $v_1$ and $v_2$. Notice that this distribution for the vertex $v$ happens for the image of the abstract tropical curve in $\mathbb{R}^2 $. At the vertex $v$, the cross-ratio ${\lambda}'_{l}$ is satisfied, but there might be other cross-ratios that are also satisfied at this vertex. Hence, after cutting $e$, the valence of each of the vertices $v_i$ might be greater than or equal to 3. Due to Lemma \ref{n0}, when $n=0$, one of the split vertices is of valence 3, and another one is of valence greater than or equal to 4.
\end{remark}

\begin{notation}\label{splitNotation}
Let $ C $ be a tropical stable map contributing to $ \mathcal{N}^{0}_{\Delta_{\mathbb{F}_{r}}(a,b, w_{\underline{d}})}( p_{\underline{n}}, L_{\underline{k}}, {\lambda}_{\underline{l-1}}, {\lambda}'_{l} ) $, where $ n \geq 1 $ and $ |{\lambda}'_{l}| $ goes to infinity and it has a contracted bounded edge $e$. Assume we cut $e$ and obtain two tropical stable maps $ C_1 $ and $ C_2 $.
We refer to \((\Delta_{1}, n_{1}, k_{1}, l_{1}, f_{1} \mid \Delta_{2}, n_{2}, k_{2}, l_{2}, f_{2})\) as a split of degrees and conditions between the two parts of the cut tropical stable map, \( C_1 \) and \( C_2 \), where the following equalities hold:

\begin{equation}\label{conditions on splitting cbd}
    \Delta_{1}\cup \Delta_{2} = \Delta_{\mathbb{F}_{r}}(a,b, w_{\underline{d}}),\;  n_{1} + n_{2} = n,\;  k_{1} + k_{2} = k,\;  l_{1} + l_{2} = l - 1,\; f_{1} + f_{2} = f . 
 \end{equation}
    
Based on Remark \ref{v12}, there might be some cross-ratios that need to be adapted after cutting $ e $. If $ {\lambda}_j $ is a degenerate cross-ratio that is satisfied at a vertex $ v_i \in h(\Gamma_i)  $ for $ i = 1, 2 $  and one and exactly one of the {marked end}s contributing in $ {\lambda}_j $ is a marked end of $ h(\Gamma_t)  $ such that $ t \neq i $, we replace this end with $ e_i $  and we denote the adapted cross-ratio condition with $ {\lambda}_{j}^{\rightarrow e_i} $ (in the case of $n=0$ this will be $ {\lambda}_{j}^{\rightarrow e} $ where $e$ is the contracted end in Lemma \ref{n0}, see \eqref{n=0}). Indeed, we adapted
all cross-ratios to the cut edge $e$. 
The notation introduced here is analogous to the notation used in \cite[Construction 54]{G}.
  
\end{notation}

Also, we can denote the number of {marked point}s on each $ C_i $ with $ m_i = n_i + k_i + f_i + 1 $, since $ e_i $ is a new {marked point} that appears after splitting the tropical stable map, so $ m_1 + m_2 = m + 2 $.
The following cycles address the locus of the two new marked points $ e_i$, on each $ C_i $. 
\begin{equation}\label{Xi}
     X_i = \prod_{k\in \underline{k_i}}\ev^{*}_k (L_k )\cdot \prod_{t\in \underline{n_i}}\ev^{*}_t (p_t )\cdot \prod_{j\in \underline{l_i}}ft^{*}_{{\lambda}_{j}^{\rightarrow e_i}} (0)\cdot \mathcal{M}_{0,m_i}(\mathbb{R}^{2},\Delta_i),
\end{equation}
\begin{center}
    $  Y_i =\ev_{e_i,*}(X_i) \subset \mathbb{R}^2$
\end{center}
For $i=1,2$, these cycles $ Y_i$ tell us how to glue $ C_i$ via $ e_i$. Indeed, $ Y_i $ for $ i = 1, 2 $ act as an additional condition on $ C_t $, where $ t\neq i $ and determines the locus of the {marked point} $ e_t $ on the image of $ \Gamma_t $.  
Based on the valence of each $ v_i $, there will be three different cases that can happen by cutting the contracted bounded edge:

\begin{itemize}
    \item \textbf{$ \val(v_i) = 3 $ for $ i = 1, 2 $:}\\
    In this case, on one hand, 
     \begin{eqnarray*}
 \dim\mathcal{M}_{0,m_i}(\mathbb{R}^{2},\Delta_i) = \# \Delta_i + (m_i+1 ) - 1 &=& 2n_i + (k_i +1) + l_i  \\
         \Rightarrow \# \Delta_i + n_i + (k_i+1) + f_i &=& 2n_i + (k_i+1) + l_i  \\
         \Rightarrow \# \Delta_i &=& n_i + l_i - f_i+1 . 
\end{eqnarray*}
    so $ \dim(X_i) = 1 $ for $ i = 1, 2 $, and therefore the {marked point} $ e_i $ can move unboundedly along the edge adjacent to it in $ C_i $. On the other hand, $ Y_t $ where $ t \neq i $ as a condition on $ C_i $ fixes the {marked point} $ e_i $ and, as a result, fixes the tropical stable map $ C_i $ in $ X _i $. Therefore, if we forget $ Y_t $, we obtain the 1-dimensional family of the tropical stable maps in $ Y_i $, and the only movable component of a tropical stable map in this family is the vertex $ v_i $.
    Based on Lemma \ref{YiisLst}, the 1-dimensional cycle $ Y_i $ has ends of primitive directions $ (r,1), (0,1), (1, 0) $ and $ (-1,0) $.
    Hence, in such a split we can think about $ Y_t $ as a condition on $ C_i $ which reduces the dimension of the intersection product on $ \mathcal{M}_{0,m_i}(\mathbb{R}^{2},\Delta_i) $ by one. 

  A split $ (\Delta_{1}, n_{1}, k_{1}, l_{1}, f_{1} | \Delta_{2}, n_{2}, k_{2}, l_{2}, f_{2}) $ is called a \textbf{1/1 split} if $ \# \Delta_i = n_i + l_i - f_i + 1 $ holds for $ i = 1, 2 $.
In this case, we refer to $ e $ as a 1/1 edge.
\begin{example}\label{Y1Y1}
    In Figure \ref{Ex11}, we see an example of a 1/1 slit. Let \( C = (\Gamma, x_{\underline{5}}, \tilde{e}_{1}, h) \) be a tropical stable map of degree \( \Delta_{\mathbb{F}_{r}}(1,1, w_1) \), where \( \tilde{e}_1 \) has weight \( w_1 = 1 \). The tropical stable map \( C \) satisfies three point conditions, \( p_{\underline{3}} \), and two line conditions, \( L_{\underline{2}} \), such that \( h(x_i) = p_i \) for \( i = 1,2,3 \), and \( h(x_j) = p_j \in L_j \cap h(\Gamma) \) for \( j = 4,5 \). 

Furthermore, \( C \) satisfies two cross-ratio conditions: \( \lambda_1 = \{x_{2}, x_{3}, x_{4}, x_{5} \} \) and \( \lambda'_2 = (x_{1} \tilde{e}_{1} | x_{2} x_{3}) \), where \( \lambda_2 \) has a very long length. Additionally, \( C \) contains a contracted bounded edge \( e \), which we split into two ends on the abstract tropical curve \( \Gamma \) of \( C \). These ends correspond to two marked points, \( e_1 \) and \( e_2 \), on the two components of the tropical stable map, \( C_1 \) and \( C_2 \), respectively. These marked points can move unboundedly on each \( C_i \).

     \begin{figure}[h!]
		
		\includegraphics[scale=0.8]{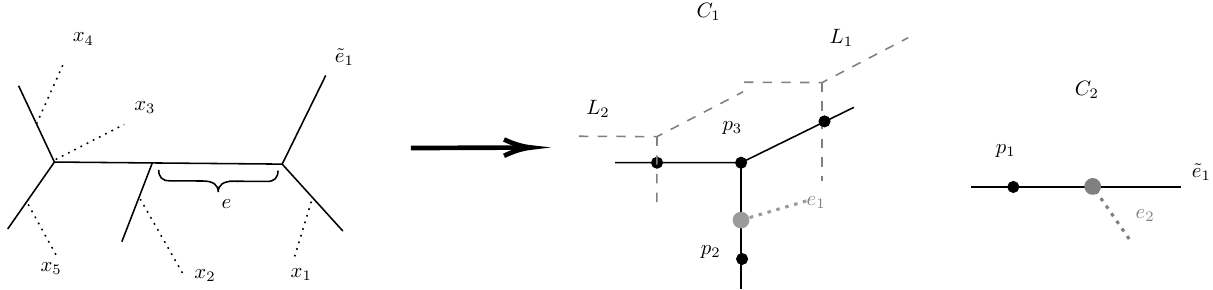}
		\centering
		\caption{In Example \ref{Y1Y1}, the tropical stable map 
$C = (\Gamma, x_{\underline{5}}, \tilde{e}_{1}, h)$ satisfies three point conditions $p_{\underline{3}}$, two line conditions $L_{\underline{2}}$, and two cross-ratio conditions $\lambda_1 = \{x_{2}, x_{3}, x_{4}, x_{5} \}$ and $\lambda'_2 = (x_{1} \tilde{e}_{1} | x_{2} x_{3})$, with $h(x_i) = p_i$ for $i = 1, 2, 3$, and $h(x_i) = L_i$ for $i = 4, 5$.
On the left, the graph $\Gamma$ is shown. We split the edge $e$ into two marked ends $e_1$ and $e_2$. The tropical stable map $C_1$ in the center satisfies $p_{2}, p_{3}, L_{1}, L_{2},$ and $\lambda_1$. On the right, $C_2$ satisfies $p_1$. The marked points $e_1$ and $e_2$ are unbounded on $C_1$ and $C_2$, giving $\dim Y_i = 1$ for $i = 1, 2$. Notice that the labeled point $e_1$ on $C_1$ can move unboundedly in the primitive direction $(0,-1)$ and $e_2$ on $C_2$ can move unboundedly in the primitive direction $(1,0)$.}

		\label{Ex11}
	\end{figure}

\end{example}

\begin{lemma}\label{YiisLst}
  Let $ C $ be a tropical stable map contributing to $ \mathcal{N}^{0}_{\Delta_{\mathbb{F}_{r}}(a,b, w_{\underline{d}})}( p_{\underline{n}}, L_{\underline{k}}, {\lambda}_{\underline{l-1}}, {\lambda}'_{l} ) $ where $ n \geq 0 $ and $\lambda'_l$ has very long length. If $ C $ has a 1/1 contracted bounded edge, then the 1-dimensional cycles $Y_i$ have ends of primitive direction $(1,0), (r,1), (-1,0)$, and $(0,-1)$.
    
\end{lemma}

\begin{proof}
    Assume that \( C = (\Gamma, x_{\underline{m}}, \Tilde{e}_{\underline{d}}, h) \) is a tropical stable map contributing to  \(
    \mathcal{N}^{0}_{\Delta_{\mathbb{F}_{r}}(a,b, w_{\underline{d}})}( p_{\underline{n}}, L_{\underline{k}}, {\lambda}_{\underline{l-1}}, {\lambda}'_{l} )
   \) 
    where \( \lambda'_l \) has a very long length, and \( C \) contains a 1/1 contracted bounded edge \( e \).  
We split \( C \) along \( e \) into, \( C_1 \) and \( C_2 \), and split \( e \) into two ends, \( e_1 \) and \( e_2 \), which are adjacent to the vertices \( v_1 \) and \( v_2 \).  
The abstract tropical curve $\Gamma_i$ associated with a tropical stable map in $X_i$ (see~\ref{Xi}) contains a single vertex $v_i$. Since $\dim X_i = 1$, this vertex moves unboundedly.
Since \( n \geq 0 \), \( v_i \) is adjacent to a vertex in the fixed component, implying that \( v_i \) is of type I (see Definition \ref{IIII}). If \( v_i \) is a 3-valent vertex, $v_i$ has to be adjacent to a non-contracted end so that it can move in one direction unboundedly, and the proof is complete.
    
    If the valence of \( v_i \) is greater than 3, either $v_i$ is connected to multiple fixed vertices of \( \Gamma_i \) along edges in different directions. This would force \( v_i \) to be fixed, which is a contradiction.
    Or $\val(v_i)>3$ and all edges adjacent to $v_i$ are parallel, then the movable component of $C_i$ is not a single vertex, but it is a movable branch containing a single vertex. On the one hand, to have an unbounded movement for $v_i$, it has to be adjacent to non-contracted ends in its direction of movement. On the other hand, $\val(v_i)>3$, so there exists a cross-ratio \( \lambda_j \in \lambda_{[l-1]} \) such that \( \lambda_{j}^{\to e_i} \) is satisfied at \( v_i \). So \( v_i \) can only be adjacent to a marked right end; otherwise, Remark \ref{Gol18} guarantees that there is a marked point caused by a multi-line condition on the end adjacent to $v_i$ and this results in having a bounded movement for $v_i$, which is a contradiction. Therefore, if $\val(v_i)>3$, then $v_i$ can only move in the $(1,0)$ direction.
\end{proof}

    \item \textbf{Without loss of generality $ \val(v_1) = 3, \val(v_2) \geq 4 $: }\\
Since $ \val(v_2) \geq 4 $, then there is at least one cross-ratio that is satisfied at $ v_2 $. On the one hand, we know that other conditions fix both $ C_i$s, and the only possible movable part of them are $ v_i$s raised from the cut edge $ e $. On the other hand, having at least one cross-ratio condition satisfied on $v_2$ such that 3 marked ends contributing to this cross-ratio live on the fixed parts of the tropical stable map, makes $v_2$ a fixed vertex. Therefore, $ v_2 $ is fixed, and this shows that $ \dim Y_2 = 0 $. Hence the following equality holds and $ Y_2 $ acts as a condition on $ C_1 $ that reduces the dimension of the intersection product on $ \mathcal{M}_{0,m_1}(\mathbb{R}^{2},\Delta_1) $ by 2.

\begin{equation}    
\label{hiddenf}
 \begin{aligned}
 \# \Delta_2 + (m_2 + 1) - 1 &= 2n_2 + k_2 + l_2  \\
         \Rightarrow \# \Delta_2 + n_2 + k_2 + (f_2 + 1)-1 &= 2n_2 + k_2 + l_2  \\
         \Rightarrow \# \Delta_2 &= n_2 + l_2 - f_2 .
 \end{aligned}
 \end{equation}
We have the following relation between the number of conditions on $ C_1 $.

\begin{eqnarray*}
 \dim\mathcal{M}_{0,m_1}(\mathbb{R}^{2},\Delta_1) = \# \Delta_1 + (m_1 + 1) - 1 &=& 2n_1 + k_1 + l_1 + 2 \\
         \Rightarrow \# \Delta_1 + n_1 + k_1 + f_1 &=& 2n_1 + k_1 + l_1 + 2 \\
         \Rightarrow \# \Delta_1 &=& n_1 + l_1 - f_1 + 2.
\end{eqnarray*}

Since one of the $ Y_i$s is 0 dimensional and the other one is 2 dimensional, we call such a split a \textbf{2/0 split}.
A split $ (\Delta_{1}, n_{1}, k_{1}, l_{1}, f_{1} | \Delta_{2}, n_{2}, k_{2}, l_{2}, f_{2}) $ is called a 2/0 split if $ \# \Delta_i = n_i + l_i - f_i + 2 $ and $ \# \Delta_t = n_t + l_t - f_t  $ hold for $ i \neq t $. In this case, we refer to $ e $ as a 2/0 edge.

Notice that not only when $ n\geq 1 $ we may have 2/0 split, but also when there is no point condition, the contracted edge we face in Lemma \ref{n0} is a 2/0 split.

\begin{example}\label{Y2Y0}

    In Figure \ref{Ex20}, we see an example of a 2/0 slit. Let \( C = (\Gamma, x_{\underline{5}}, \tilde{e}_{1}, h) \) be a tropical stable map of degree \( \Delta_{\mathbb{F}_{r}}(1,1, w_1) \), where \( \tilde{e}_1 \) has weight \( w_1 = 1 \). The map \( C \) satisfies three point conditions, \( p_{\underline{3}} \), and two line conditions, \( L_{\underline{2}} \), such that \( h(x_i) = p_i \) for \( i = 1,2,3 \), and \( h(x_j) = p_j \in L_j \cap h(\Gamma) \) for \( j = 4,5 \). 

Furthermore, \( C \) satisfies two cross-ratio conditions: \( \lambda_1 = \{x_{2}, x_{3}, x_{4}, x_{5} \} \) and \( \lambda'_2 = (x_{1} \tilde{e}_{1} | x_{2} x_{3}) \), where \( \lambda_2 \) has a very long length. Additionally, \( C \) contains a contracted bounded edge \( e \), which we split into two ends on the abstract tropical curve \( \Gamma \) of \( C \). These ends correspond to two marked points, \( e_1 \) and \( e_2 \), on the two components of the tropical stable map, \( C_1 \) and \( C_2 \), respectively. Notice that the labeled point $e_2$ on $C_2$ can move within a 2-dimensional family, while $C_1$ remains fixed.

     \begin{figure}[h!]
		
		\includegraphics[scale=0.8]{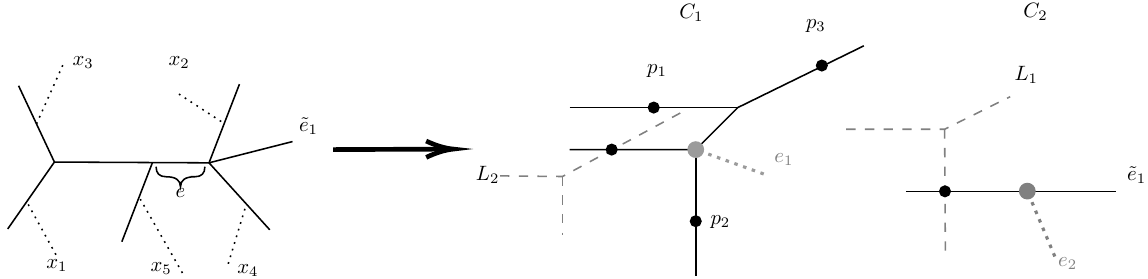}
		\centering
		\caption{In Example \ref{Y2Y0}, the tropical stable map 
$C = (\Gamma, x_{\underline{5}}, \tilde{e}_{1}, h)$ satisfies three point conditions $p_{\underline{3}}$, two line conditions $L_{\underline{2}}$, and two cross-ratio conditions \( \lambda_1 = \{x_{2}, x_{3}, x_{4}, x_{5} \} \) and \( \lambda'_2 = (x_{1} \tilde{e}_{1} | x_{2} x_{3}) \), with $h(x_i) = p_i$ for $i = 1, 2, 3$, and $h(x_i) = L_i$ for $i = 4, 5$.
On the left, the graph $\Gamma$ is shown. We split the edge $e$ into two marked ends $e_1$ and $e_2$. The tropical stable map $C_1$ in the center satisfies $p_{2}, p_{3}, L_{1}, L_{2},$ and $\lambda^{\rightarrow e_1}_1 = \{x_{2}, x_{3}, e_{1}, x_{5} \}$. On the right, $C_2$ satisfies $L_1$.  }\label{Ex20}
\end{figure}
\end{example}

    \item \textbf{$ \val(v_i) \geq 4 $ for $ i = 1, 2 $:}\\
    The argument of the previous case shows that $ \dim Y_i = 0 $ for $ i = 1, 2 $. But in such a case $ Y_1 \cdot Y_2 $ is empty, so this is not the case, and the only possible ways to split the tropical stable maps over a contracted bounded edge are 1/1 or 2/0 splits. 
    
\end{itemize}

\subsection{Splitting over a movable branch $\mathcal{S}$}\label{4.2}

Let $ C=(\Gamma, x_{\underline{m}}, \Tilde{e}_{\underline{d}}, h) $ be a tropical stable map contributing to  $\mathcal{N}^{0}_{\Delta_{\mathbb{F}_{r}}(a,b, w_{\underline{d}})}( p_{\underline{n}}, L_{\underline{k}}, {\lambda}_{\underline{l-1}}, {\lambda}'_{l} )$, 
with a very long cross-ratio $\lambda'_l$. Suppose $C$ has a movable branch consisting of $\sigma$ edges in the $(-1,0)$ direction, each having weights $c_{\underline{\sigma}}$. This movable branch is connected to $\sigma$ fixed components $C_{\underline{\sigma}}$ of degree 
\[ \Delta_{\mathbb{F}_{r}}(a_i,b_i, w_{i\underline{d_i}}), \quad 1 \leq i \leq \sigma, \]
where the first parts of Equations \eqref{eq:1}, \eqref{eq:2}, and \eqref{eq:3} hold. 
After cutting the movable branch, denote the number of ends in the $(1,0)$ direction by $d'_i$ for each $1 \leq i \leq \sigma$, where $d'_i = d_i + 1$. 
Each fixed component $C'_i$ thus gains one additional marked right end in the $(1,0)$ direction with weight $c_i$. Denote the degree of the fixed tropical stable maps $C'_{\underline{\sigma}}$ with
\begin{equation}
  \Delta_{\mathbb{F}_{r}}(a'_i, b'_i, w_{i\underline{d'_i}}) \text{ for } 1 \leq i \leq \sigma.\label{degoffixed}
\end{equation}

 After cutting the movable branch, all bounded edges with weight $c_j$ turn into unbounded ends of $\mathcal{S}_\sigma$ of primitive direction $(-1,0)$. Additionally, there may be a cross-ratio condition imposed at a vertex of $\mathcal{S}_\sigma$ containing a marked end on the $j$th fixed component $C'_j$. So, such an end of weight $c_j$ needs to be labeled.
\begin{definition} \label{ftildas}
A \emph{marked left end} is a non-contracted marked end of $\mathcal{S}_\sigma$ of primitive direction $(-1,0)$, denoted by $\Tilde{f}_j$. This end previously connected $\mathcal{S}_\sigma$ to the fixed component $C_j$ before the movable branch was cut. An end is called a marked left end if there is at least one cross-ratio condition imposed at a vertex of $\mathcal{S}_\sigma$ involving a marked end on the $j$-th fixed component $C'_j$.
Let $\delta$ be the set of all non-marked left ends of $\mathcal{S}_\sigma$, and let $\tilde{\delta}$ be the set of marked left ends. We denote the number of marked left ends by $\#\tilde{\delta} = \Lambda$, and the number of non-marked left ends by $\#\delta = \sigma - \Lambda$.
\end{definition}
If $ {\lambda} $ is a degenerate cross-ratio that is satisfied at a vertex $ v \in \mathcal{S}_\sigma $ and a {marked end} contributing in $ {\lambda} $ is a marked end of a $ C_i $ for $1 \leq i \leq \sigma$, we replace this end with $ \Tilde{f}_i $ (this is a well-defined cross-ratio condition based on Definition \ref{CR}.)

In the definition of the degree of a tropical stable map (see \ref{Degree}), if there are other weighted ends in other directions, their weights must be included in the notation. For $\mathcal{S}_{\sigma}$, which contains:
marked ends in the $(-1,0)$ direction and usual marked right ends. The degree of $\mathcal{S}_{\sigma}$ is then denoted as:
\[
\Delta_{\mathbb{F}_{r}}(a_0, b_0, w_{\underline{d_0}}, c_{\underline{\sigma}}),
\]
where $c_{\underline{\Tilde{\delta}}}$ represents the weights of ends in the $(-1,0)$ direction and $a_0 \leq a$ and $d_0 \leq d$. The number of marked right ends is denoted by $\Tilde{e}_{0\underline{d_0}}$.
Figure \ref{ExS} illustrates this setting before cutting the movable branch, while Figure \ref{ExSa} shows the configuration after the cut.

After gluing back the $C'_i$s and $\mathcal{S}_\sigma$ together, we aim to reconstruct a tropical stable map of degree $\Delta_{\mathbb{F}_{r}}(a,b, w_{\underline{d}})$. This requires the following equalities to hold:

\begin{align}
    a &= \sum_{i=0}^{\sigma} a_{i} = a_0 + \sum_{i=1}^{\sigma} a'_i, \label{eq:1} \\
   d &= \sum_{i=0}^{\sigma} d_i = d_0 + \sum_{i=1}^{\sigma} (d'_i -1), \label{eq:2} \\ 
    b= \sum_{i=0}^{\sigma} b_i, \quad b_i &= \sum_{j=1}^{d_i} w_{ij}, \quad b'_i = \sum_{j=1}^{d'_i} w_{ij} = b_i + c_i. \label{eq:3}
\end{align}
Also, cutting the movable branch will distribute the conditions between connected components $\{ (n_i, k_i, l_i, f_i) \}_{i=0}^{\sigma}$. The index zero is the number of conditions on the movable branch, and of course $n_0 = 0$. The following equalities must also hold:

\begin{equation}
    n = \sum\limits_{i=1}^{\sigma}n_{i} , \quad 
    k = \sum\limits_{i=0}^{\sigma}k_{i} , \quad 
    l-1 = \sum\limits_{i=0}^{\sigma}l_{i} , \quad 
    f = \sum\limits_{i=0}^{\sigma}f_{i} .
    \label{eq:all_sums}
\end{equation}

    \begin{figure}
	\includegraphics[scale=0.8]{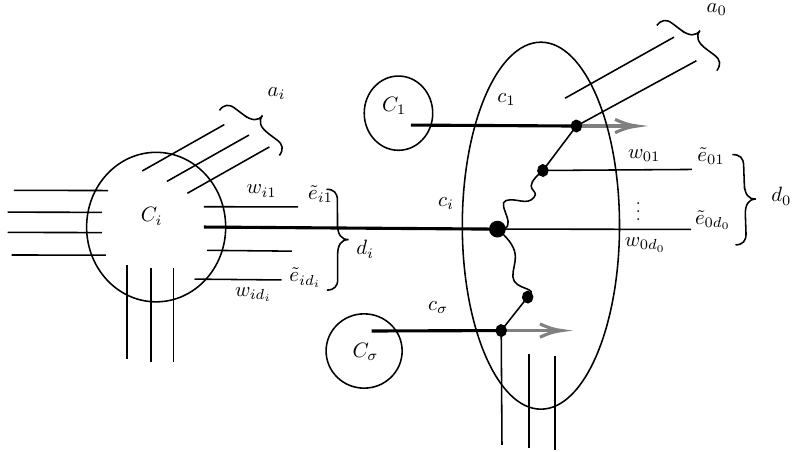}
		\centering
	\caption{In this figure, we can locally see the labeling of the ends and their weights before we cut the movable branch. Here, $C_{i} $ is one of the fixed components, and without loss of generality, we can assume $\Tilde{e}_{ij}$ has weight $w_{ij}$ for $1 \leq j \leq d_{i} $. }
	\label{ExS}
\end{figure}

    \begin{figure}
	\includegraphics[scale=0.8]{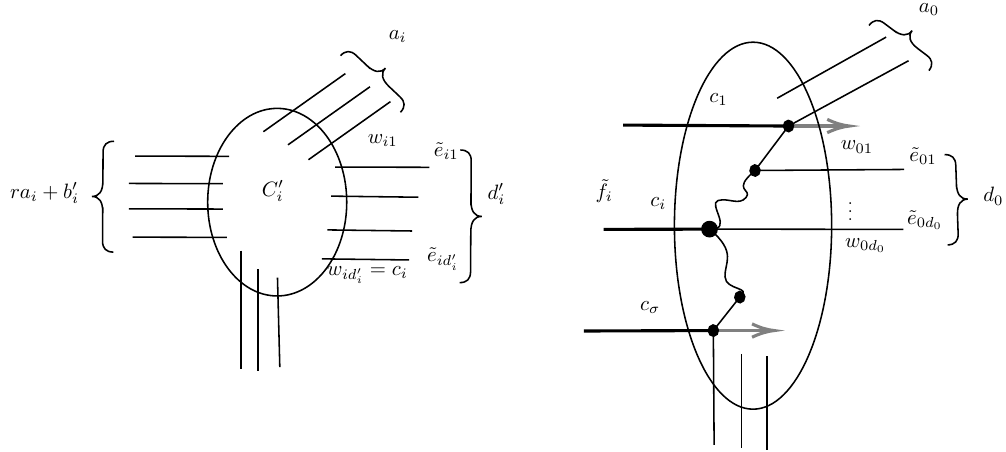}
		\centering
	\caption{In this figure, we can locally see the labeling of the ends and their weights. Here, we can see that the movable branch is of degree $ \Delta_{\mathbb{F}_{r}}(a_0,b_0, w_{\underline{d_0}}, c_{\underline{\sigma}})$. This movable branch has $d_0$ marked right ends $\Tilde{e}_{0\underline{d_0}}$. The tropical stable map $C'_i$ is of degree $\Delta_{\mathbb{F}_{r}}(a'_i,b'_i, w_{i\underline{d'_i}})$ with $d'_i$ marked right ends of weights $w_{i\underline{d'_i}}$.}
	\label{ExSa}
\end{figure}

\section{Multiplicities of split tropical stable maps}\label{MOSC}

In this section, we will see what the relation is between the multiplicity of the main tropical stable map $ C $ and the fixed split parts in different cases where we split $ C $ over a contracted edge or cut a movable branch $ \mathcal{S} $.

First, we want to recall a new type of multi-line conditions similar to what has been defined in \cite[Definition 61 and Notation 62]{G}:
\begin{definition}\label{st}
The following tropical intersections 
$ L_{(1,0)} := \max_{(x,y) \in \mathbb{R}^{2}}(x,0) \cdot \mathbb{R}^{2} $,
$ L_{(0,1)} :=\max_{(x,y) \in \mathbb{R}^{2}}(y,0) \cdot \mathbb{R}^{2} $ and $ L_{(1,-r)} := \max_{(x,y) \in \mathbb{R}^{2}}(x,ry) \cdot \mathbb{R}^{2} $ and any translations of them are called degenerate tropical lines and they will play the role of a new type of multi-line conditions that we need to impose on split tropical stable maps. See Figure \ref{lines}.

\begin{figure}[h!]
\centering
\includegraphics[scale = 0.8]{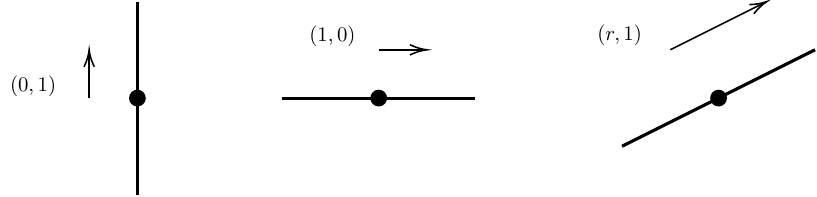}
\caption{Degenerate tropical lines on $ \mathbb{F}_{r} $ from left to right $ L_{(1,0)}, L_{(0,1)} $ and $ L_{(1,-r)} $ with ends of weight one. }
\label{lines}
\end{figure}

\end{definition}

We are going to use Definition \ref{st}, to study the {marked point} $ e_i $ in a 1/1 split. Indeed, in a 1/1 split, we saw that $ Y_i $ for $ i = 1, 2 $ acts as a 1-dimensional condition on $ C_j $ for $ j \neq i $. These degenerate tropical lines play the role of 1-dimensional conditions $ Y_i $ for $ i = 1, 2 $. 

Let $ (s, t) \in \{ (0,1), (1,0), (1,-r) \} $, then $ L_{(s, t)} $ is a degenerate tropical line as in Definition \ref{st}. Due to Lemma \ref{YiisLst}, we can replace the tropical stable map $ C_i $ that satisfies in $ Y_j $ at $ v_i $, with $ C_{i, (s, t)} $ which is equal to $ C_i $ but satisfies in $ L_{(s, t)} $ instead of $ Y_j $.

Now we will see what is the relation between the ev-multiplicity of $ C $ and $ C_i $ for $ i = 1, 2 $, for the three possible cases which is mentioned in the previous section.

     \subsection{Splitting over a 2/0 edge} 
    
    \begin{lemma}

     Let $ C $ be a tropical stable map contributing to $ \mathcal{N}^{0}_{\Delta_{\mathbb{F}_{r}}(a,b, w_{\underline{d}})}( p_{\underline{n}}, L_{\underline{k}}, {\lambda}_{\underline{l-1}}, {\lambda}'_{l} ) $ where $ n \geq 0 $ and assume that $ C $ has a 2/0 contracted bounded edge, then:
   
    \begin{equation}
    	\label{2/0mult}	
{  \mult_{\ev}(C) = \mult_{\ev}(C_1) \cdot \mult_{\ev}(C_2) }
    \end{equation}
 \end{lemma}
\begin{proof}

    This kind of splitting occurs in two different settings, where we have at least one point condition or where there is no point condition on $ C $. So we address such a split for each case separately:
    
    \begin{itemize}
    	\item \textit{There is at least one point condition on $ C $:} \\
    	 Let $ C $ be a tropical stable map contributing to $ \mathcal{N}^{0}_{\Delta_{\mathbb{F}_{r}}(a,b, w_{\underline{d}})}( p_{\underline{n}}, L_{\underline{k}}, {\lambda}_{\underline{l-1}}, {\lambda}'_{l} ) $ where $ n \geq 1 $, such that $ C $ has a 2/0 contracted bounded edge. We call the contracted bounded edge $ e $ and the vertices adjacent to $ e $ by $ v_{1}, v_{2} $ such that $ v_i \in C_i $ for $ i = 1, 2 $.
    	 
    	 Without loss of generality, assume $ \dim(Y_1) = 0 $ and consider $ v_1 $ as a base point to compute the ev-matrix $ M(C) $.

    	 	\begin{equation*} M(C) = 
    	 		\begin{array}{cc}
    	 			
    	 			&\text{ Base }  v_1   \text{ \quad lengths in } C_1   \text{\quad lengths in }  C_2  \\  & \\ 
    	 			
    	 			\begin{array}{c}
    	 				\text{conditions in } C_1 \\ \\ \\ \text{conditions in }  C_2 
    	 			\end{array}
    	 			&  
    	 			
    	 			\left(
    	 			\begin{array}{ccccc|ccccc||ccccc} &&&&&&&&&&&&&&\\
    	 				&&*&& &&&*&& &&& 0 && \\ &&&&&&&&&&&&&& \\ \hline\hline &&&&&&&&&&&&&& \\
    	 				&&*&& &&&0&& &&& * &&  \\ &&&&&&&&&&&&&&
    	 			\end{array}
    	 			\right)

    	 		\end{array}
    	 	\end{equation*}
    	 	
    	 In order to obtain the equality \eqref{2/0mult} we need to prove that $ M(C) $ is a squared block diagonal matrix, that the upper left block corresponds to $ M(C_1) $ and the determinant of the lower right block corresponds to $ \mult_{\ev}(C_2) $. These two blocks are distinctly separated within the matrix $ M(C) $ by a double line.
    	 
    	 Since $ \dim(Y_1) = 0 $, so we have  $ \# \Delta_1 = n_1 + l_1 - f_1  $ an therefore the upper left block of $ M(C) $ has exactly these number of columns:
    	 
    	 \begin{center}
    	 	$ 2 +  \# \Delta_1 + m_1 + 1 - l_1 - 3 = \# \Delta_1 + n_1 + k_1 + f_1 - l_1 = 2n_1 + k_1 $
    	 \end{center}
    	  On one hand, this is the number of bounded edges in $ C_1 $  plus the 2 columns for the base point, which is on $ C_1 $ again. On the other hand, $ 2n_1 + k_1 $ is exactly the number of conditions on $ C_1 $. Thus, the upper left block of $ M(C) $ corresponds to $ M(C_1) $. If we call the lower right block of $ M(C) $ by $ M $, then we proved that $ |\det(M(C))| = \mult_{\ev}(C_1)\cdot |\det(M)| $.
    	  
    	  If we take $ v_2 $ as a base point for $ C_2 $ then $ M(C_2) $ will be the following blocked matrix:

      \begin{equation*} M(C_2) = \begin{array}{c}
      		\text{Base } v_2 \qquad\qquad \\
      	
      \left(	\begin{array}{ccc|ccc}
      	&&&&& \\	
     	1&0 && & 0 & \\
     	0&1 && & & \\ \hline &&&&& \\ 
     	 &* && & M & \\ &&&&&

      	\end{array}	\right) \end{array}
      \end{equation*}
    	  
    Notice that $ |\det(M(C_2))| = |\det(M)| = \mult_{\ev}(C_2) $, so the equality \eqref{2/0mult} holds in this case.

    	\item \textit{There is no point condition on $ C $:}

    To compute the ev-matrix $ M(C) $, selecting the base point on $ C_1 $ yields the same result as in the previous scenario, with the upper-right block determining $ M(C_1) $.

    	However, there is no point condition on $ C $, so $ Y_1 $ serves as a condition on $ C_2 $ and effectively reduces the dimension of $ Y_2 $ by 2.  So we can consider $ Y_1 $ as a point condition on $ C_2 $. Therefore, the ev-matrix of $ C_2 $ will remain as described in the previous case, and we obtain the same formula for computing the multiplicity in this case $ \mult_{\ev}(C) = \mult_{\ev}(C_1) \cdot \mult_{\ev}(C_2) $.

    \end{itemize} 
\end{proof}

    \subsection{Splitting over a 1/1 edge}
    
    \begin{lemma}\label{multof1/1case}
   Let $ C $ be a tropical stable map contributing to $ \mathcal{N}^{0}_{\Delta_{\mathbb{F}_{r}}(a,b, w_{\underline{d}})}( p_{\underline{n}}, L_{\underline{k}}, {\lambda}_{\underline{l-1}}, {\lambda}'_{l} ) $ where $ n \geq 1 $ and assume that $ C $ has a 1/1 contracted bounded edge, then:
      \begin{equation}\label{1/1mult}
           \mult_{\ev}(C) = |\det(M(C_{1,(1, 0)})) \cdot \det(M(C_{2,(0, 1)})) - \det(M(C_{1,(0, 1)})) \cdot \det(M(C_{2,(1, 0)}))| 
    \end{equation}
\end{lemma}

\begin{proof}
   The tropical stable map $ C $ contributing to $ \mathcal{N}^{0}_{\Delta_{\mathbb{F}_{r}}(a,b, w_{\underline{d}})}( p_{\underline{n}}, L_{\underline{k}}, {\lambda}_{\underline{l-1}}, {\lambda}'_{l} ) $ has a 1/1 contracted bounded edge. We call the contracted bounded edge \( e \), and denote the vertices adjacent to \( e \) by \( v_{1}, v_{2} \) such that \( v_i \in C_i \) for \( i = 1, 2 \).
 Notice that in this case $ \dim(Y_i) = 1 $ for $ i = 1, 2 $ and they act similarly to multi-line conditions, so we replace them with the degenerate tropical lines that have been introduced in definition \ref{st}. We omit the column corresponding to the length of the contracted bounded edge, and accordingly, the row corresponding to \( \lambda'_l \).


\begin{equation*} M(C) = 
	\begin{array}{cc}
		
		&\text{ \qquad Base }v_1\text{ \qquad lengths in } C_1   \text{\quad lengths in }  C_2  \\  & \\ 
		
		\begin{array}{c}
			\text{conditions in } C_1 \\ \\ \\ \text{conditions in }  C_2 
		\end{array}
		&  
		
		\left(
		\begin{array}{ccccc|ccccc||c|ccccc} &&&&&&&&&&*&&&&&\\
			&&*&& &&&*&&&\vdots &&& 0 && \\ &&&&&&&&&&*&&&&& \\ \hline\hline &&&&&&&&&&0&&&&& \\
			&&*&& &&&0&&&\vdots &&& * &&  \\ &&&&&&&&&&0&&&&&
		\end{array}
		\right) 
	\end{array}
\end{equation*}

Denote the number of bounded edges in $C_i$ by $x_i$, and let $y_i = 2n_i + k_i$ for $i = 1,2$. Since $\dim Y_i = 1$ for $i=1,2$, we have the relation:
\[
x_i + 2 = y_i + 1.
\]
Now, consider the lower right block of $M(C)$, whose entries are denoted by $*$ and correspond only to the columns associated with the lengths in $C_2$. Let us denote this submatrix by $M$. The matrix $M$ has dimensions $x_2 \times y_2$, meaning it is not a square matrix.
Next, define $A$ as the submatrix of $M(C)$ consisting of all *-entries above the double line. This submatrix $A$ is also not square and has dimensions $(x_1 +2) \times y_1$.

To compute $\det M(C)$, we apply the Laplace expansion starting from its $(x_1+2)$nd column. Recursively, we continue applying the Laplace expansion to each column corresponding to the lengths in $C_1$, beginning with the rightmost column, that is the $(x_1+2)$nd column of $M(C)$.

By the time we reach the third column of $M(C)$ during the Laplace expansion, all nonzero elements above the double line, except for two entries $a_{i1}, a_{i2}$ in one row $1 \leq i \leq y_1$, have been covered. Additionally, all nonzero elements of $M(C)$ below the double line remain.
Using the $i$th row that contains the remaining entries above the double line in $M(C)$, we define $M_{(s,t)}$ for a pair $(a_{i1}, a_{i2}) = (s,t) \in I$, where:
$I = \{ (1,0), (0,1), (1, -r) \}$.


\begin{equation*} M_{(s, t)} = \begin{array}{c}
		\qquad\text{\qquad lengths in }  C_2 \\
		
		\left(	\begin{array}{cc|cccc}
			a_{r1}& a_{r2}&&0&\cdots&0 \; \\	
			\hline	&&&&& \\
			&&&&& \\ 
			&* \quad && & M & \\ &&&&& \\ &&&&&

		\end{array}	\right). \end{array}
\end{equation*}

By grouping the summands, we obtain:

\begin{equation*}
	\det(M(C)) = \sum_{(s, t) \in I}F_{(s, t)}\det(M_{(s, t)}),
\end{equation*}
\begin{equation}\label{Fst}	
	 \text{such that \;} F_{(s, t)} = \sum_{i:(a_{i1}, a_{i2})=(s,t)}\sum_{\sigma}sgn(\sigma)\prod_{j=3}^{x_1+2}a_{\sigma(j)j},
\end{equation}

where the second sum in $ F_{(s, t)} $ goes over all bijections $ \sigma : \{ 3, \cdots, x_1+2
\} \rightarrow \{ 1, \cdots, i-1, i+1, \cdots, y_1 \} $.
Since $ v_1 $ and $ v_2 $ map to the same point in $ \mathbb{R}^2 $, $ M_{(s, t)} $ is actually $ M(C_{2,(s, t)}) $, therefore:

\begin{equation}\label{detM}
	 |\det(M(C))| = |F_{(1, 0)} \det(M(C_{2,(1, 0)})) + F_{(0, 1)} \det(M(C_{2,(0, 1)})) + F_{(1, -r)} \det(M(C_{2,(1, -r)}))| 
\end{equation}

Let $ A_{(s, t)} $ be the following matrix, obtained from the matrix $ A $ by adding a new first row $ (s, t, 0, \cdots, 0) $. Again notice that $ A_{(s, t)} $ is actually $ M(C_{1,(s, t)}) $.


\begin{equation*} A_{(s, t)} = \begin{array}{cc}
		&\qquad\text{\quad lengths in }  C_1 \\  \text{conditions in } C_1 &
		
		\left(	\begin{array}{cc|cccc}
			s&t&&0&\cdots&0 \; \\	
			\hline	&&&&& \\
			&&&&& \\ 
			&* \quad && & * & \\ &&&&& \\ &&&&&

		\end{array}	\right) \end{array}
\end{equation*}
	
Using Leibniz' formula to compute the determinant of $ A_{(s, t)} $ and replacing \eqref{Fst}, leads to the following equality:

\begin{equation}\label{detA}
	\det(A_{(s, t)}) = \det \left(\begin{array}{cc}
		s&t\\1&0
	\end{array}\right) F_{(1, 0)} + \det \left(\begin{array}{cc}
	s&t\\0&1
\end{array}\right) F_{(0, 1)} + \det \left(\begin{array}{cc}
s&t\\1&-r
\end{array}\right) F_{(1, -r)}
\end{equation}	
	 
Therefore, we have the following relations:

\begin{center}
$ 	\det(A_{(1, 0)}) = F_{(0, 1)} - rF_{(1, -r)} $ \\ $
	\det(A_{(0, 1)}) = -F_{(1, 0)} - F_{(1, -r)}$ \\ $
	\det(A_{(1, -r)}) = rF_{(1, 0)} + F_{(0, 1)} $
\end{center}
 By solving this system of linear equations that we obtained from equality \eqref{detA}, we get the following answer:
 
 \begin{equation*}
 	\left(\begin{array}{c}
 		F_{(1, 0)}\\F_{(0, 1)}\\F_{(1, -r)}
 	\end{array}\right) = \left(\begin{array}{c}
 	\det (-A_{(0, 1)})\\ \det( A_{(1, 0)})\\0
 \end{array}\right)
 \end{equation*}


Replacing this solution in equality \eqref{detM}, finishes the proof.

\end{proof}

  \subsection{Splitting over a movable branch $\mathcal{S}$}
  \begin{remark}\label{Sfix}
Let $\mathcal{S}_\sigma$ be a movable branch defined in \ref{DefString} of degree $\Delta_{\mathbb{F}_{r}}(a_0,b_0, w_{\underline{d_0}}, c_{\underline{\sigma}})$ (see Definition \ref{Degree}). By Definition \ref{DefString},  the direction of movement of all vertices in $\mathcal{S}_\sigma$ is $(1,0)$. Hence, all vertices in $\mathcal{S}_\sigma$ are type I vertices adjacent to fixed edges in $(-1,0)$ direction. Therefore, $\mathcal{S}_\sigma$ consists of exactly $\sigma$ vertices connected by $\sigma-1$ bounded edges with directions $(x_i, y_i)$ for $1 \leq i \leq \sigma-1$. 
A fixed point on any bounded edge of $\mathcal{S}_\sigma$ or an end in direction $(r,1)$ or $(0,-1)$ prevents $\mathcal{S}_\sigma$ from moving in $(1,0)$. If $ \sigma \geq 2 $, without loss of generality, we can consider a fixed point \( p \) on an end in direction \( (r,1) \). Turning \( \mathcal{S}_\sigma \) into a fixed tropical stable map then allows us to compute its multiplicity; see Lemma \ref{mul1}. If \( \sigma = 1 \), we let the ev-multiplicity of \( \mathcal{S}_1 \) to be equal to one.

\end{remark}

\begin{lemma}\label{mul1}
Let $\mathcal{S}_\sigma$, for $ \sigma \geq 2 $, be the movable branch described in Remark \ref{Sfix}. Then the ev-multiplicity of $\mathcal{S}_\sigma$ is \( \left| \prod_{i=1}^{\sigma-1} y_i \right| \).
\end{lemma}

\begin{proof}
By Remark~\ref{Sfix}, \( \mathcal{S}_\sigma \) has \( \sigma \) vertices connected by \( \sigma-1 \) bounded edges \( l'_{\underline{\sigma-1}} \) with directions \( (x_i, y_i) \) for \( 1 \leq i \leq \sigma-1 \). For each \( 1 \leq i \leq \sigma-1 \), \( y_i \) cannot be zero; otherwise, we obtain a two-dimensional movement for \( \mathcal{S}_\sigma \). Moreover, in Remark~\ref{Sfix}, we have a fixed point \( p = (p_x, p_y) \) on an end in direction \( (r,1) \). Let \( l \) denote the length of the bounded edge adjacent to \( p \).
The multiplicity of the tropical stable map satisfying:
\begin{itemize}
    \item $\sigma$ fixed ends $L^*_{\underline{\sigma}}$ (defined via real numbers $g_{\underline{\sigma}}$ as $L^*_i = (\ev_e)^*_y(g_i)$ for each $1\leq i \leq \sigma$ defined in \ref{fixedends}),
    \item  1 point condition $p$,
    \item $l_0$ cross-ratio conditions $\lambda_{\underline{l_0}}$,
    \item $k_0$ multi-line conditions $L_{\underline{k_0}}$,
\end{itemize}
is given by the determinant of the matrix in~(\ref{multS}), where the lengths of the bounded edges adjacent to the marked points caused by the multi-line conditions are denoted by \( l^*_{\underline{k_0}} \). Note that the $k_0$ multi-line conditions $L_{\underline{k_0}}$ can only be imposed on the ends of $\mathcal{S}_\sigma$ in primitive direction $(-1,0)$. Otherwise, placing them on other edges would block unbounded movement in the $(1,0)$ direction.

\begin{equation}\label{multS}  \begin{array}{c|ccccccccc}
		
        &p_x&p_y&l&l'_1&\cdots&l'_{\sigma-1}&l^*_1&\cdots&l^*_k\\
		\hline
         p_x&1&&&&&&&&\\
         p_y&0&1&&&&&&&\\
       L^*_1&0&1&-1&&&&0&&\\
      
       L^*_2&0&1&-1&y_1&&&&&\\
       \vdots&\vdots&\vdots&\vdots&\vdots&\ddots&&&&\\
       L^*_\sigma&0&1&-1&y_1&&y_{\sigma-1}&&&\\
       L_{1}&0&1&-1&y_1&&&-1&&\\
       \vdots&\vdots&\vdots&\vdots&\vdots&&*&&\ddots&\\
       L_{k_0}&0&1&-1&y_1&&&&&-1\\
         \end{array}        
\end{equation}
    The absolute value of the determinant of this matrix is equal to the multiplicity of $\mathcal{S}_\sigma$.
\end{proof}

\begin{lemma} \label{multSC}
    Let $ C $ be a tropical stable map contributing to $ \mathcal{N}^{0}_{\Delta_{\mathbb{F}_{r}}(a,b, w_{\underline{d}})}( p_{\underline{n}}, L_{\underline{k}}, {{\lambda}}_{\underline{l-1}}, {{\lambda}}'_{l} ) $ where $ n \geq 1 $ and assume that $ C $ has a movable branch $ \mathcal{S}_\sigma $ with $\sigma \geq 2 $, which is described as in the second part of Proposition \ref{*}. (As mentioned in Definition \ref{DefString}, $c_{\underline{\sigma}}$ are the weights of the ends of $\mathcal{S}_\sigma$ of primitive direction $(-1,0)$.) \\If  $\lambda'_l$  has two marked ends in $ C_i$ and two marked ends in $ C_{j} $ for  $1 \leq i,j \leq \sigma$, then:
\begin{equation}\label{cicj}
   \mult_{\ev}(C) =  (c_{i}+c_{j})\cdot \mult_{\ev}(\mathcal{S}_\sigma)\cdot \mult_{\ev}(C_{i})\cdot \mult_{\ev}(C_{j})\prod^{\sigma}_{\substack{q=1,\\q\neq i, j}} c_q \cdot \mult_{\ev}(C_q).
\end{equation}
  If $\lambda'_l$  has two marked ends in $ C_i$ for  $1 \leq i \leq \sigma$ and two marked ends on two other fixed components, then: 
  \begin{equation}\label{cq}
   \mult_{\ev}(C) = \mult_{\ev}(\mathcal{S}_\sigma)\cdot \mult_{\ev}(C_{i})\prod^{\sigma}_{\substack{q=1,\\q\neq i}} c_q \cdot \mult_{\ev}(C_q).
\end{equation}  
 
\end{lemma}

\begin{proof}
Let $C$ be a tropical stable map satisfying the conditions of Proposition \ref{*} with a movable branch $ \mathcal{S}_\sigma $.
Depending on the distribution of the marked ends contributing to the non-degenerate cross-ratio condition ${\lambda}'_{l}$, we have two general cases. Either there are $i\neq j \in \{1, \cdots, \sigma \}$ such that two marked ends of $\lambda'_l$ are on $C_i$ and the other two marked ends on $C_j$. The other case is when there are two marked ends on $C_i$, and the other marked ends each live on different fixed components. This means that the only bounded edge contributing to $\lambda'_l$ whose length grows as we move the movable branch to the right is the bounded edge connecting $C_i$ to $\mathcal{S}_\sigma$.

Let's first label the bounded edges connecting $\mathcal{S}_\sigma$ to the fixed components with $l_q$ for $1\leq q \leq \sigma$, and each edge can have a weight $c_q\in \mathbb{N}$. Without loss of generality, assume that the base point (we denote it by $p$), for identifying $C$ in its moduli space, lives in $C_1$, and $C_1$ is a fixed component adjacent to a vertex satisfying case (1) in Lemma \ref{4cases}. Then the order of the $C_q$s is such that for smaller $q$, $C_q$ is connected to $C_1$ with fewer bounded edges (when the number of bounded edges connecting $C_1$ to two fixed components is equal, we can choose randomly which one can have smaller $q$). We label the fixed components such that $C_q $ is connected to $\mathcal{S}_\sigma$ via $l_q$.

We label the bounded edges in \(\mathcal{S}_\sigma\) adjacent to \(l_s\) and \(l_{s+1}\) for \(1\leq s \leq \sigma-1\) as \(l'_s\).
If there exists a bounded edge in \(\mathcal{S}_\sigma\) adjacent to \(l_s\) and \(l_t\) for \(1\leq s < s+1 < t \leq \sigma-1\), we label it as \(l'_{\min\{s,t\}}\).
If \(l'_{\min\{s,t\}}\) has already been assigned to another bounded edge, we increment its index to \(l'_{\min\{s,t\}+u}\) for \(u = 1, 2, \dots\) until we find an unused label. If multiple values of \(t\) satisfy this condition, we assign labels starting from the lowest \(t\).

\begin{figure}[h!]
\centering
\includegraphics[scale = 0.8]{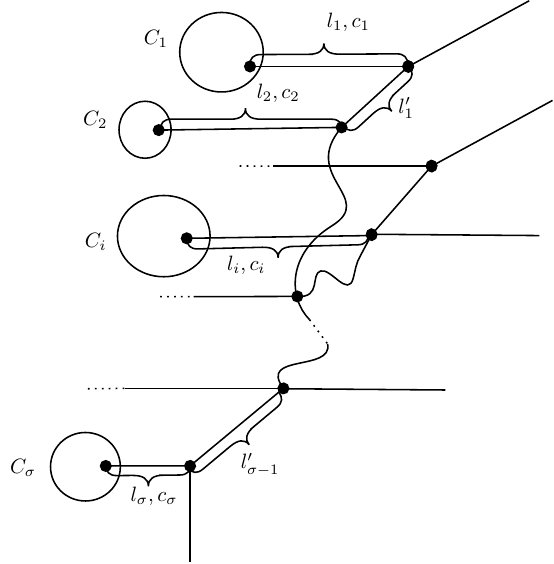}
\caption{In this figure, we present a topological sketch of a movable branch (not an actual tropical curve) connected to the fixed components via edges labeled by their lengths \(l_q\) and weights \(c_q\) for \(1 \leq q \leq \sigma\). The fixed components are represented by circles labeled \(C_q\).
 }
\label{ll'}
\end{figure}
As an example, see Figure \ref{ll'}.
In the evaluation matrix of $C$, we put the lengths of the bounded edges of $C$ in the order given in the following table:

\begin{equation*} \footnotesize{ \begin{array}{c|cc|ccc|ccc|ccc|c}
		&  &\text{ lengths in } &&&   \text{ lengths in } &&&\text{lengths in } &&&\text{lengths in } &\\
        &p&C_1&l_1&l'_1& C_2 &l_2&\cdots&C_{i-1}&l_{i-1}&l'_{i-1}& C_i&l_i\\
		\hline
         p&I&&&&&&&&&&&\\
       C_1&I&*&&&&&&&&&&\\
     C_2&I&*&\begin{array}{c}
            c_1  \\
            0 
       \end{array}&*&*&\begin{array}{c}
            -c_2  \\
            0 
       \end{array}&&&&&\\
       \vdots&\vdots&&\vdots&\vdots&&\vdots&&&&&&\\
       C_{i-1}&I&*&\begin{array}{c}
            c_1  \\
            0 
       \end{array}&*&&0&\cdots&*&\begin{array}{c}
            -c_{i-1}  \\
            0 
       \end{array}&&&\\
       C_i&I&*&\begin{array}{c}
            c_1  \\
            0 
       \end{array}&*&&0&&&0&*&*&\begin{array}{c}
            -c_i  \\
            0 
       \end{array}\\
       C_{i+1}&I&*&\begin{array}{c}
            c_1  \\
            0 
       \end{array}&*&&0&&&0&*&&0\\
       \vdots&\vdots&&\vdots&\vdots&&\vdots&&&\vdots&\vdots&&\vdots\\
       \lambda'_l&&&&&&&&&&&&1\\
         \end{array}
         }       
\end{equation*}

    We want to prove that the evaluation matrix of $C$, is a blocked diagonal matrix, and the determinant of each block, up to a sign, is equal to the evaluation multiplicity of a tropical stable map $C_q$.
If we expand the block of the evaluation matrix $M(C)$, for each $C_i$, for $1\leq i \leq \sigma$, we can see that all blocks of are size $2+ \# \{\text{lengths of all bounded edges in } C_i \}$, and this is equal to $2n_i + k_i$. (Note that the columns corresponding to the lengths of the contracted bounded edges are omitted here, and similarly, the rows corresponding to the associated cross-ratio conditions are also omitted.)

 By applying the following elementary column operations: 
\(
l_{q-1} \cdot \frac{c_q}{c_{q-1}} + l_q \hookrightarrow l_q,
\) starting from \( q = 2 \) and repeating this process for each \( q = 2, \ldots, \sigma \), the evaluation matrix is transformed into a block diagonal matrix. 
The matrix below presents a more detailed form of the evaluation matrix after performing column operations. The variables \(x_i\) and \(y_i\) will be explained later.

\begin{equation*} \footnotesize{ \begin{array}{c|cc|ccc|ccc|ccc|c}
		&  &\text{ lengths in } &&&   \text{ lengths in } &&&\text{lengths in } &&&\text{lengths in } &\\
        &p&C_1&l_1&l'_1& C_2 &l_2&\cdots&C_{i-1}&l_{i-1}&l'_{i-1}& C_i&l_i\\
		\hline
         p&I&&&&&&&&&&&\\
       C_1&I&*&&&&&&&&&&\\
     C_2&I&*&\begin{array}{c}
            c_1  \\
            0 
       \end{array}&\begin{array}{c}
            x_1  \\
            y_1 
       \end{array}&*&&&&&&\\
       \vdots&\vdots&&\vdots&\vdots&&&&&&&&\\
       C_{i-1}&I&*&\begin{array}{c}
            c_1  \\
            0 
       \end{array}&\begin{array}{c}
            x_1  \\
            y_1 
       \end{array}&&\begin{array}{c}
            c_2  \\
            0 
       \end{array}&\cdots&*&&&&\\
       C_i&I&*&\begin{array}{c}
            c_1  \\
            0 
       \end{array}&\begin{array}{c}
            x_1  \\
            y_1 
       \end{array}&&\begin{array}{c}
            c_2  \\
            0 
       \end{array}&&&\begin{array}{c}
            c_{i-1}  \\
            0 
       \end{array}&\begin{array}{c}
            x_{i-1}  \\
            y_{i-1} 
       \end{array}&*&\\
       C_{i+1}&I&*&\begin{array}{c}
            c_1  \\
            0 
       \end{array}&\begin{array}{c}
            x_1  \\
            y_1 
       \end{array}&&\begin{array}{c}
            c_2  \\
            0 
       \end{array}&&&\begin{array}{c}
            c_{i-1}  \\
            0 
       \end{array}&\begin{array}{c}
            x_{i-1}  \\
            y_{i-1} 
       \end{array}&&\begin{array}{c}
            c_i  \\
            0 
       \end{array}\\
       \vdots&\vdots&&\vdots&\vdots&&\vdots&&&\vdots&\vdots&&\vdots\\
       \lambda'_l&0&0&0&0&0&0&\cdots&0&0&0&0&1\\
         \end{array}
         }       
\end{equation*}


     Applying the mentioned elementary operations results in having a value $\frac{c_{i}+c_{j}}{c_{i}c_j}\cdot c_\sigma$ in the column corresponding to $l_\sigma $ and the row corresponding to ${{\lambda}}'_l$ and the rest of the elements in the column of $l_\sigma$ will be zero. Then, we can factor all of the weights of the bounded edges connecting the movable branch to the fixed components so we obtain this factor:
    $ \frac{c_{i}+c_{j}}{c_{i}c_j}\prod^{\sigma}_{\substack{q=1}} c_q $, when each couple of the marked ends contributing to $\lambda'_l$ lives on $C_i $ and $C_j$ and if the only very long length contributing to $\lambda'_l$ is $l_i$ then the factor would be: $\prod^{\sigma}_{\substack{q=1, \\q\neq i}} c_q $.

After factoring the weights $c_q$, the columns corresponding to $l_q$s are up to sign equal to the first column of the ev-matrix of the corresponding $C_q$ for $q=1, \cdots \sigma$. The columns corresponding to $l'_q$ are the ones that are supposed to become equal to the second column of the ev-matrix corresponding to $C_q$.

Having a bounded edge $l'_q$ in the movable branch in $(x_q,0)$ direction results in having a two-dimensional movement. So the direction of a bounded edge $l'_q$ is $(x_q,y_q)$ for $y\neq 0$. Due to Lemma \ref{mul1} if we factor all $y_q$s, then $\prod_{q=1}^{\sigma-1}y_q$ is the ev-multiplicity of the movable branch $\mathcal{S}_\sigma$. So after factoring all $y_q$, it is straightforward that we can use some elementary column operations to turn the column corresponding to $l'_q$ to the second column of the ev-matrix of $C_q$, and this finishes the proof.

\end{proof}

\begin{corollary}\label{whyc1c2}
    Since the cross-ratio multiplicities can be expressed locally at vertices (see Definition \ref{multcr}), the contributions from vertices to cross-ratio multiplicities do not depend on cutting edges. Therefore, from \eqref{2/0mult}, the following equality can be concluded:

   \begin{equation}\label{2c1c2}
    	 \mult(C) = \mult(C_1) \cdot \mult(C_2)
    \end{equation}
Moreover, from \eqref{cicj} and \eqref{cq}, the following equalities can be concluded:
\begin{equation}\label{2cicj}
   \mult(C) =  (c_{i}+c_{j})\cdot \mult(\mathcal{S}_\sigma)\cdot \mult(C_{i})\cdot \mult(C_{j})\prod^{\sigma}_{\substack{q=1,\\q\neq i, j}} c_q \cdot \mult(C_q),
\end{equation}
 
  \begin{equation}\label{2cq}
   \mult(C) = \mult(\mathcal{S}_\sigma)\cdot \mult(C_{i})\prod^{\sigma}_{\substack{q=1,\\q\neq i}} c_q \cdot \mult(C_q),
\end{equation} 
When \( \sigma = 1 \), since \( \lambda'_l \) has two marked ends on \( C_1 \) and two marked ends on \( \mathcal{S}_1 \), and we let the ev-multiplicity of \( \mathcal{S}_1 \) to be equal to one, it follows that \( \mult(C) = \mult(C_1) \).

\end{corollary}

\section{General Kontsevich-style formula for Hirzebruch Surfaces}\label{Formula}

In this section, we prove a Kontsevich-style formula for counting $(m+d)-$marked tropical stable maps in $\mathbb{F}_{r}$ satisfying point, multi-line, and cross-ratio conditions.

\begin{lemma} \label{Cristophlemma}
    Let $(\Delta_{1}, n_{1}, k_{1}, l_{1}, f_{1} | \Delta_{2}, n_{2}, k_{2}, l_{2}, f_{2})$ be a given split of conditions as in Notation \ref{splitNotation}. Then:
\begin{align}
&
\sum_{\substack{(\Delta_{1}, n_{1}, k_{1}, l_{1}, f_{1} | \Delta_{2}, n_{2}, k_{2}, l_{2}, f_{2}) \\ \text{is a split over a contracted bounded } \\ \text{ edge respecting } {\lambda}'_l } } \sum\mult(C) =\label{00}\\ 
&
    \sum_{\substack{(\Delta_{1}, n_{1}, k_{1}, l_{1}, f_{1} | \Delta_{2}, n_{2}, k_{2}, l_{2}, f_{2}) \\ \text{is a 1/1 split respecting } {\lambda}'_l}} \mathcal{N}^{0}_{\Delta_{1}}(p_{\underline{n_1}}, L_{\underline{k_1}},  L_{e_1}, {\lambda}_{\underline{l_1}}^{\rightarrow e_1}) \cdot \mathcal{N}^{0}_{\Delta_{2}}(p_{\underline{n_2}}, L_{\underline{k_2}},  L_{e_2}, {\lambda}_{\underline{l_2}}^{\rightarrow e_2}) \label{01} 
       \\ &+ 
        \sum_{\substack{(\Delta_{1}, n_{1}, k_{1}, l_{1}, f_{1} | \Delta_{2}, n_{2}, k_{2}, l_{2}, f_{2}) \\ \text{is a 2/0 split respecting } {\lambda}'_l \text{, such that} \\ \#\Delta_1 = n_1 + l_1 - f_1 }} \mathcal{N}^{0}_{\Delta_{1}}(p_{\underline{n_1}}, L_{\underline{k_1}}, {\lambda}_{\underline{l_1}}^{\rightarrow e_1}) \cdot \mathcal{N}^{0}_{\Delta_{2}}(p_{\underline{n_2}},  p, L_{\underline{k_2}}, {\lambda}_{\underline{l_2}}^{\rightarrow e_2}),  \label{02}
\end{align}

    holds where the second sum in \eqref{00} goes over all tropical stable maps $C$ with a contracted bounded edge $e$ (that splits to two contracted ends $e_1, e_2$) such that $C$ contributes to $ \mathcal{N}^{0}_{\Delta_{\mathbb{F}_{r}}(a,b, w_{\underline{d}})}(p_{\underline{n}}, L_{\underline{k}}, {\lambda}_{\underline{l-1}}, {\lambda}'_l) $ and the cross-ratio ${\lambda}'_l$ has a very long length, and for $i=1,2$,  \( L_{e_i} \) is a multi-line condition of weight one that is imposed on $e_i$ where $p = \ev_{e}(C) \in Y_1 \cap Y_2$.
\end{lemma}
\begin{proof}

Let $C = (\Gamma, x_{\underline{m}}, \Tilde{e}_{\underline{d}}, h)$ be a tropical stable map contributing to $ \mathcal{N}^{0}_{\Delta_{\mathbb{F}_{r}}(a,b, w_{\underline{d}})}(p_{\underline{n}}, L_{\underline{k}}, {\lambda}_{\underline{l-1}}, {\lambda}'_l) $ where the cross-ratio ${\lambda}'_l$ has a very long length and $C$ has a contracted bounded edge $e$. We can split \( C \) along \( e \) into \( C_1 = (\Gamma_1, x_{\underline{m_1}}, \Tilde{e}_{\underline{d_1}}, h_1) \) and \( C_2 = (\Gamma_2, x_{\underline{m_2}}, \Tilde{e}_{\underline{d_2}}, h_2) \) such that
\( m_1+m_2=m, d_1+d_2=d \), and split \( e \) into two ends, \( e_1 \) and \( e_2 \), such that their image in the plane are \( \ev_{e_1}(C_1)= v_1 \) and \( \ev_{e_2}(C_2)= v_2 \). 



Either \( e \) is a \( 1/1 \) edge or a \( 2/0 \) edge.  
If \( e \) is a \( 1/1 \) edge, then, due to Remark \ref{RationalyEquivalent}, we know that the number of tropical stable maps contributing to \( \mathcal{N}^{0}_{\Delta_{\mathbb{F}_{r}}(a,b, w_{\underline{d}})}(p_{\underline{n}}, L_{\underline{k}}, {\lambda}_{\underline{l-1}}, {\lambda}'_l) \) is independent of the exact position of the point and multi-line conditions.
 Thus, we can reposition all conditions so that they are contained within rectangular boxes far enough away from each other in \( \mathbb{R}^2 \). Specifically, we want \( p_{\underline{n_1}} \) and \( L_{\underline{k_1}} \) be placed in the upper left region, while \( p_{\underline{n_2}} \) and \( L_{\underline{k_2}} \) are in the lower right region.
 The placement of point and multi-line conditions in the boxed regions implies that the vertices of \( h_1(\Gamma_1) \) are contained in the upper left box and the vertices of \( h_2(\Gamma_2) \) in the lower right box of \( \mathbb{R}^2 \). Consequently, \( Y_1 \) and \( Y_2 \) occupy the positions indicated in Figure \ref{Yiinboxes}.

 \begin{figure}[h!]
\centering
\includegraphics[scale = 0.9]{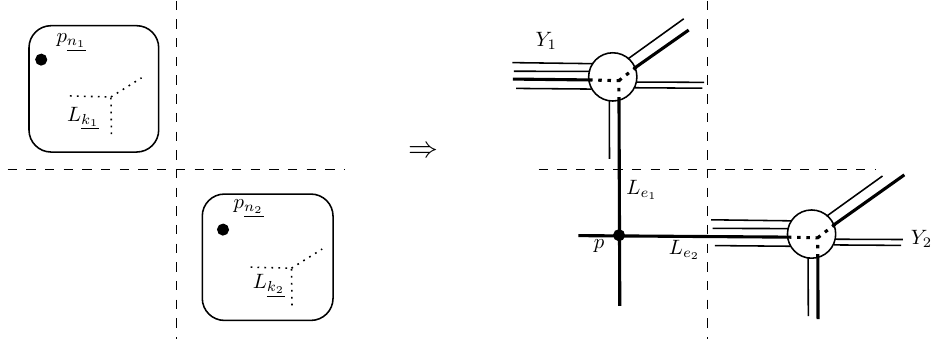}
\caption{In the left figure, the real plane is divided into four quadrants. All of the marked points \( p_{\underline{n_1}} \) and the lines \( L_{\underline{k_1}} \) are placed in the upper-left quadrant, while \( p_{\underline{n_2}} \) and \( L_{\underline{k_2}} \) are located in the lower-right quadrant while they are contained in boxes far away from each other. As a result, \( Y_1 \) and \( Y_2 \) appear in the corresponding quadrants, as shown in the right figure. In the right picture, $L_{e_i}$ are the multi-line conditions that we want to apply on $C_j$ at $ v_j $ for $i\neq j$.  The unbounded ends of \( L_{e_1} \) and \( L_{e_2} \) are drawn with thicker lines to highlight them.}
\label{Yiinboxes}
\end{figure}

We know that for \( i = 1, 2 \), \( v_i \) is a vertex of \( h_i(\Gamma_i) \), and from another perspective, \( v_i = \ev_{e_i}(C_i) \in Y_i \).  For \( i \neq j \), \( Y_i \) acts as a condition on \( C_j \), reducing the dimension of 
\[ \prod_{k\in \underline{k_i}}\ev^{*}_k (L_k )\cdot \prod_{t\in \underline{n_i}}\ev^{*}_t (p_t )\cdot \prod_{j\in \underline{l_i}}ft^{*}_{{\lambda}_{j}^{\rightarrow e_i}} (0)\cdot \mathcal{M}^{\trop}_{0,m_j+d_j}(\mathbb{R}^2, \Delta_j) \]
 by one. Moreover, for \( i \neq j \), \( v_i \in Y_j \cap h_i(\Gamma_i) \), so \( v_i \in Y_j \). Hence, \( p = v_i = v_j \) is a point in \( Y_i \cap Y_j \), which is the image of the contracted edge \( e \). But notice that $Y_i \cap Y_j$ can contain more than one point. 
For a $p \in Y_1 \cap Y_2$, let \( L_{e_i} \) be a multi-line condition of weight one that passes through $p$, as shown in Figure \ref{Yiinboxes}. 
Since the point \( p \) lies on the ends of \( Y_i \) and \( Y_j \) of primitive directions \( (0,-1) \) and \( (-1,0) \), we can replace \( Y_i \) with \( L_{e_i} \) as a condition implying on $C_j$ and consider  \( C_1 \) to be the tropical stable map that contributes to  
\( \mathcal{N}_{d_1} ( p_{\underline{n_1}}, L_{\kappa_1}, L_{e_1}, \lambda^{\to e_1}_{l_1} ) \)  
and  \( C_2 \) to be the tropical stable map that contributes to  
\( \mathcal{N}_{d_2} ( p_{\underline{n_2}}, L_{\kappa_2}, L_{e_2}, \lambda^{\to e_2}_{l_2} ) \). Due to Lemma \ref{multof1/1case} and Corollary \ref{whyc1c2},
\[
\mult(C)= \mult(C_{1,(1,0)}) \cdot \mult(C_{2,(0,1)}),
\]

 because $\mult(C_{1,(0,1)})$ and $\mult(C_{2,(1,0)})$ in \eqref{1/1mult} both vanish by our choice of $L_{e_1}$ and $L_{e_2}$. So when \( e \) is a \( 1/1 \) edge, we can split \( C \) into \( C_{1,(1,0)} \) and \( C_{2,(0,1)} \), which contribute to the sum in \eqref{01}. Since \( \ev_e(C) = p \in L_{e_1} \cap L_{e_2} \), we can glue the tuple \( C_{1,(1,0)}, C_{2,(0,1)} \), satisfying \eqref{01}, and reconstruct a tropical stable map that contributes to \eqref{00}.

If \( e \) is a \( 2/0 \) edge, without loss of generality, consider \( \#\Delta_1 = n_1 + l_1 - f_1 \). Then, \( \dim Y_2 = 2 \), and \( Y_1 \) is exactly the point \( p = \ev_e(C) \), which acts as an additional point condition on \( C_2 \). Reordering the sum in \eqref{00} gives us two sums over \( 1/1 \) and \( 2/0 \) splits, and this completes the proof.

\end{proof}

\begin{theorem}\label{Main0}
We use the same notations as in Section \ref{Split curves} and Lemma \ref{Cristophlemma}. Consider \( n \) points \( p_{\underline{n}} \), \( k \) multi-lines \( L_{\underline{k}} \), and \( l\geq1 \) cross-ratio conditions \( {\lambda}_{\underline{l}} \), all in general position. The number of \((m+d)\)-marked tropical stable maps with \( d \) marked right ends \( \Tilde{e}_{\underline{d}} \) that contribute to  
\(\mathcal{N}^{0}_{\Delta_{\mathbb{F}_{r}}(a,b, w_{\underline{d}})}(p_{\underline{n}}, L_{\underline{k}}, {\lambda}_{\underline{l}})\)  
is determined recursively by the following formulas. Here, \( m = n + k + f \), where the marked points labeled in \( \underline{f} \) satisfy no additional conditions. Let $\lambda'_l$ denote a cross-ratio that degenerates to $\lambda_l$.

\begin{enumerate}
    \item \textbf{For} $n \geq 1$:
    
    \begin{equation}
            \mathcal{N}^{0}_{\Delta_{\mathbb{F}_{r}}(a,b, w_{\underline{d}})}(p_{\underline{n}}, L_{\underline{k}}, {\lambda}_{\underline{l}}) = \label{General}
    \end{equation}

    \begin{align}
    &\sum_{\substack{
        (\Delta_{1}, n_{1}, k_{1}, l_{1}, f_{1} | \Delta_{2}, n_{2}, k_{2}, l_{2}, f_{2}) \\ 
        \text{is a 1/1 split respecting } {\lambda}'_l
    }} 
    \mathcal{N}^{0}_{\Delta_{1}}(p_{\underline{n_1}}, L_{\underline{k_1}}, L_{e_1}, {\lambda}_{\underline{l_1}}^{\rightarrow e_1}) 
    \cdot 
    \mathcal{N}^{0}_{\Delta_{2}}(p_{\underline{n_2}}, L_{\underline{k_2}}, L_{e_2}, {\lambda}_{\underline{l_2}}^{\rightarrow e_2}) 
    \label{1} 
    \\ 
    &+ \sum_{\substack{
        (\Delta_{1}, n_{1}, k_{1}, l_{1}, f_{1} | \Delta_{2}, n_{2}, k_{2}, l_{2}, f_{2}) \\ 
        \text{is a 2/0 split respecting } {\lambda}'_l \\ 
        \text{and } \#\Delta_1 = n_1 + l_1 - f_1 
    }} 
    \mathcal{N}^{0}_{\Delta_{1}}(p_{\underline{n_1}}, L_{\underline{k_1}}, {\lambda}_{\underline{l_1}}^{\rightarrow e_1}) 
    \cdot 
    \mathcal{N}^{0}_{\Delta_{2}}(p_{\underline{n_2}},  p, L_{\underline{k_2}}, {\lambda}_{\underline{l_2}}^{\rightarrow e_2})  
    \label{2}
    \\ 
    &+ \sum_{\sigma=1}^{ra+b} 
    \sum_{\mathcal{S}_\sigma} \mult(\mathcal{S}_\sigma)  
    \sum_{\bigstar} 
    \mathcal{Y} \cdot
    \mathcal{X} \cdot 
    \prod_{i=1}^{\sigma} 
    \mathcal{N}^{0}_{\Delta_{\mathbb{F}_{r}}(a'_i, b'_i, w_{i\underline{d'_i}})}
    (p_{\underline{n_i}}, L_{\underline{k_i}}, {\lambda}_{\underline{l_i}}), 
    \label{3}
    \end{align}\\
such that the first two summands count tropical stable maps that are obtained by gluing two tropical stable maps of smaller degrees along contracted bounded edges, under the conditions specified in \eqref{conditions on splitting cbd}; see part \textbf{I} of Section \ref{Split curves}. The last summand counts tropical stable maps obtained by gluing movable branches $\mathcal{S}_\sigma$ to $\sigma$ fixed components $C'_{\underline{\sigma}}$ of lower degree; see \eqref{degoffixed} and part \textbf{II} of Section \ref{Split curves}. Moreover, the following hold:\\

\begin{itemize}
 \item $\bigstar := \{ (\Delta_{\mathbb{F}_r}(a'_i, b'_i, w_{i\underline{d'_i}}),  n_i, k_i, l_i, f_i) \}_{i=1}^{\sigma}$ where $\sum_{i=0}^{\sigma}(a'_i ,d'_i) = (a ,d+\sigma)$ and equalities \ref{eq:3} and \ref{eq:all_sums} must hold,\\
\item $\mathcal{Y}$ is defined based on the marked ends contributing to $\lambda'_l,$ as follows:
\[\begin{cases} 
         \begin{array}{lc}  \begin{array}{l}
     \text{Case (A): if } \lambda'_l \text{ has two marked ends in } C_s \text{ for } 1 \leq s \leq \sigma, \text{ and two}\\  \text{other marked ends live on two other fixed components:}  
\end{array} & \begin{array}{c}
    \mathcal{Y} = \prod\limits_{\substack{q=1 \\ q \neq s}}^\sigma c_q,
\end{array} \\ &\\

  \begin{array}{l}
    \text{Case (B): if } \lambda'_l \text{ has two marked ends in } C_t \\\text{ and two marked ends in } C_s, \text{ for } 1 \leq s, t \leq \sigma:
 \end{array} &
 \begin{array}{c}
   \mathcal{Y} = (c_s + c_t) \prod\limits_{\substack{q=1 \\ q \neq s,t}}^\sigma c_q. 
\end{array}
 
 \end{array}
    \end{cases}
    \]


    \item $\mathcal{X}$ represents the total number of possible ways to glue $C'_{\underline{\sigma}}$ and $\mathcal{S}_\sigma$ together:
    \[\mathcal{X}= \prod_{\substack{\iota \in \tilde{\delta}}} \#\phi_{\iota} \cdot \prod_{\substack{\iota \in {\delta}}} \#\psi_{\iota }.
    \]
      for $\iota \in {\delta}$, $c_\iota$ is the weight of a non-marked left end in $\mathcal{S}_\sigma$. All possible right ends of $C'_{\underline{\sigma}}$ that can be glued to this end of $\mathcal{S}_\sigma$ are:
      \begin{center}
          $\, \psi_{\iota} := \{ 1 \leq j \leq \sigma,   1\leq i \leq d'_j :
      c_\iota = w_{ji} \},  $
      \end{center}
      for $\iota \in \Tilde{\delta}$, $c_\iota$ is the weight of a marked left end in $\mathcal{S}_\sigma$. All possible right ends of $C'_{\underline{\sigma}}$ that can be glued to this end of $\mathcal{S}_\sigma$ are:
      \begin{center}
      $\, \phi_{\iota}:= \{ 1 \leq i \leq d'_\iota : w_{\iota i}=c_\iota \}$.
\end{center}

\end{itemize}

    \item \textbf{For} $n = 0$ the equation:

\begin{equation}
    \mathcal{N}^{0}_{\Delta_{\mathbb{F}_{r}}(a,b, w_{\underline{d}})}(L_{\underline{k}}, {\lambda}_{\underline{l}}) =    
    \mathcal{N}^{0}_{\Delta_{\mathbb{F}_{r}}(a,b, w_{\underline{d}})}(p, L_{\underline{k-2}}, {\lambda}_{\underline{l-1}}^{\rightarrow e}),  
    \label{n=0}
\end{equation}

holds where  $p = \ev_e (C)$ is in the intersection of the two multi-line conditions in Lemma \ref{0}, and $e$ is the contracted bounded edge in Lemma \ref{n0}.
    
\end{enumerate}

\end{theorem}



\begin{proof}[Proof of Theorem \ref{Main0}]
For a fixed degree $\Delta_{\mathbb{F}_r}(a, b, w_{\underline{d}})$, and given $n$ point conditions $p_{\underline{n}}$, $k$ multi-line conditions $L_{\underline{k}}$, and $l$ cross-ratio conditions $\lambda_{\underline{l}}$, we consider tropical stable maps on the Hirzebruch surface $\mathbb{F}_r$ that satisfy these constraints.

Based on Remark \ref{RationalyEquivalent}, the number of the tropical stable maps contributing to $ \mathcal{N}^{0}_{\Delta_{\mathbb{F}_{r}}(a,b, w_{\underline{d}})} (p_{\underline{n}}, L_{\underline{k}}, {\lambda}_{\underline{l}}) $ is equal to $ \mathcal{N}^{0}_{\Delta_{\mathbb{F}_{r}}(a,b, w_{\underline{d}})}(p_{\underline{n}}, L_{\underline{k}}, {\lambda}_{\underline{l-1}}, {\lambda}'_l) $ for a cross-ratio ${\lambda}'_l$ with a very long length that degenerates to ${\lambda}_l$. 
It follows from Proposition \ref{*} that a tropical stable map that contributes to $\mathcal{N}^{0}_{\Delta_{\mathbb{F}_{r}}(a,b, w_{\underline{d}})}(p_{\underline{n}}, L_{\underline{k}}, {\lambda}_{\underline{l-1}}, {\lambda}'_l)$, where $|{\lambda}'_l|$ goes to infinity, either has a contracted bounded edge or a movable branch $ \mathcal{S} $. 

In Section \ref{Split curves}, two possible ways to split the tropical stable maps over a contracted bounded edge were discussed. The first and second summands of the formula, which are \eqref{1} and \eqref{2}, correspond to the 1/1 and 2/0 splits, respectively. The notation for the degree of the split tropical stable maps is as in Notation \ref{splitNotation}. Reordering  \eqref{General} and applying Lemma \ref{Cristophlemma} proves that \eqref{1} and \eqref{2} count the number of tropical stable maps with a contracted bounded edge.
The same argument applies to the case where $n=0$, as Lemma \ref{n0} shows that $C_1$ is a fixed component with one vertex and three contracted ends, two of which arise from the two multi-line conditions based on \ref{0}.

The second possibility is to split \( C \) over a movable branch into \( C'_{\underline{\sigma}} \), where each \( C'_i \) (for \( 1 \leq i \leq \sigma \)) is a tropical stable map in \( \mathbb{F}_r \) of degree \( \Delta_{\mathbb{F}_r}(a'_i, b'_i, w_{i\underline{d'_i}}) \). Note that for a tropical stable map in \( \mathbb{F}_r \) of degree \( \Delta_{\mathbb{F}_r}(a,b, w_{\underline{d}}) \), the maximum value that \( \sigma \) can attain is \( ra + b \). The movable branch \( \mathcal{S}_{\sigma} \) has degree \( \Delta_{\mathbb{F}_r}(a_0, b_0, w_{0\underline{d_0}}, c_{\underline{\sigma}}) \), and the degree-splitting condition
\[
\sum_{i=0}^{\sigma}(a'_i, d'_i) = (a, d + \sigma)
\]
must be satisfied, along with the equalities in \eqref{eq:3} and \eqref{eq:all_sums}.
We denote the partition of the degrees and conditions with $\{ (\Delta_{\mathbb{F}_r}(a'_i, b'_i, w_{i\underline{d'_i}}), n_i, k_i, l_i, f_i) \}_{i=1}^{\sigma}$.


In \eqref{3}, $\mathcal{X}$ represents the total number of possible ways to glue $C'_{\underline{\sigma}}$ and $\mathcal{S}_\sigma$ together. For each end of primitive direction $(-1,0)$ with weight $c_{\iota}$ in $\mathcal{S}_\sigma$, if it corresponds to a marked left end $f_{\iota}$ (as defined in \ref{ftildas}), it can be glued to one of the ends of $C_\iota$ that has the same weight. The number of such ends is given by the cardinality of the set $\phi_\iota$.  
A non-marked left end of weight $c_\iota$ in $\mathcal{S}_\sigma$  can be glued to any marked right end of any fixed component. This number is precisely given by the cardinality of the set $\psi_\iota$.

Moreover in \eqref{3}, we have $\mathcal{Y}$, which is a factor that depends on the non-degenerate cross-ratio condition $\lambda'_l$, which appears in the multiplicity of the tropical stable map $C$ obtained from gluing $C'_{\underline{\sigma}}$ and $\mathcal{S}_\sigma$ (see Lemma \ref{whyc1c2}).  
This completes the proof.

\end{proof}

\begin{remark}
As seen in the last summand \eqref{3} of the formula in Theorem \ref{Main0}, all possible movable branches \(\mathcal{S}_\sigma\) must be determined. As illustrated in Example \ref{ExF3}, we apply and expand the recursive formula introduced in Theorem \ref{Main0}.

\end{remark}

\section{Counting tropical curves passing through points in Hirzebruch Surfaces}\label{GTKFrn}

In this section, we present a recursive formula for counting the number of tropical curves on the Hirzebruch surface $\mathbb{F}_r$ that satisfy only point conditions. Note that in Theorem \ref{Main0}, it is necessary to have $l \geq 1$.

\begin{theorem} \label{Frnw}
We use the same notations as in Section \ref{Split curves}.
For $ n= \#\Delta_{\mathbb{F}_r}(a, b, w_{\underline{d}})-1$, the number of tropical curves of degree \(\Delta_{\mathbb{F}_r}(a, b, w_{\underline{d}})\), passing through $n$ point conditions can be computed recursively using the following formula. 

 \begin{equation}\label{NFrw}
   \begin{aligned}
\mathcal{N}^{0}_{\Delta_{\mathbb{F}_r}(a, b, w_{\underline{d}})}(p_{\underline{n}}) = 
    & \frac{1}{r}\sum_{\substack{
        (\Delta_{1}| \Delta_{2}) \\ 
        \text{ is a 1/1 split}}} \Phi_1 \cdot
    \mathcal{N}^{0}_{\Delta_{\mathbb{F}_r}(a_1, b_1, w_{\underline{d_1}})}(p_{\underline{n_1}}) 
    \cdot 
    \mathcal{N}^{0}_{\Delta_{\mathbb{F}_r}(a_2, b_2, w_{\underline{d_2}})}(p_{\underline{n_2}}) 
    \\ 
    +&\frac{1}{r} \sum_{\sigma=1}^{ra+b} 
      \sum_{\mathcal{S}_\sigma} \mult(\mathcal{S}_\sigma)
       \sum_{\substack{
        \{ (\Delta_{\mathbb{F}_r}(a'_i, b'_i, w_{i\underline{d'_i}})\}_{i=1}^{\sigma}  \\ 
        \sum_{i=0}^{\sigma}(a'_i ,d'_i) = (a ,d+\sigma) }}
    \varphi \cdot \Phi_2\cdot  
    \mathcal{X} \cdot  
    \prod_{i=1}^{\sigma} 
   \mathcal{N}^{0}_{\Delta_{\mathbb{F}_r}(a'_i, b'_i, w_{i\underline{d'_i}})}(p_{\underline{n_i}}), 
    \end{aligned}
    \end{equation}\\
such that the first summand counts tropical stable maps that are obtained by gluing two tropical stable maps of smaller degrees along a 1/1 contracted bounded edge, under the conditions specified in \eqref{conditions on splitting cbd}. The second summand counts tropical stable maps obtained by gluing movable branches $\mathcal{S}_\sigma$ to $\sigma$ fixed components $C'_{\underline{\sigma}}$ of lower degree, where all vertices in $\mathcal{S}_\sigma$ are 3-valent and $n = \sum_{i=1}^{\sigma}n_{i}$. Moreover, the following hold:\\

\begin{itemize}
\item in the first summand of \eqref{NFrw}, $\Phi_1$ is:
\begin{equation}
   \begin{aligned}
\Phi_1 &= (ra_1 + b_1)(ra_2 + b_2)(a_1  b_2 + a_2 b_1 +ra_1 a_2)\binom{n-3}{n_1 -2}\\ &- (ra_1 + b_1)^2(a_1 b_2 + a_2 b_1 +ra_1 a_2)\binom{n-3}{n_1 -1},
  \end{aligned}
\end{equation}
\item in the second summand of \eqref{NFrw}, $\varphi$ is: 
\begin{equation}
\varphi = 
\begin{cases}
\binom{\sigma}{3} & \text{if } \sigma \geq 3 \\
1 & \text{otherwise},
\end{cases}
\end{equation}

\item for $1\leq s,t,u \leq \sigma$: ($c_{\underline{\sigma}}$ are the weights of ends of $\mathcal{S}_{\sigma}$ in primitive direction $(-1,0)$, see part II of Section \ref{Split curves}.)
\begin{equation}
   \begin{aligned}
       \Phi_2 &= \binom{n-3}{n_s-1}\binom{n-n_s-2}{n_t-1}\prod\limits_{\substack{i=1 \\ i \neq s}}^\sigma\binom{n-n_s-n_t-1-\sum_{\substack{j=1 \\ j \neq s,t}}^{i}n_i}{n_i} \Bigg[ (ra_s + b_s)(ra_t + b_t) (c_t+c_s)
       \\&\prod\limits_{\substack{q=1 \\ q \neq s,t}}^\sigma c_q + (ra_s + b_s)(ra_u + b_u)\prod\limits_{\substack{q=1 \\ q \neq s}}^\sigma c_q - (ra_u + b_u)^2\prod\limits_{\substack{q=1 \\ q \neq u}}^\sigma c_q\Bigg]  
       - \binom{n-3}{n_s-2}
       \\&\prod\limits_{\substack{i=1 \\ i \neq s}}^\sigma\binom{n-n_s-1-\sum_{\substack{j=1 \\ j \neq s,t}}^{i}n_i}{n_i} \Bigg[ (ra_t + b_t)^2(c_t+c_s)\prod\limits_{\substack{q=1 \\ q \neq s,t}}^\sigma c_q + (ra_t + b_t)(ra_u + b_u)\prod\limits_{\substack{q=1 \\ q \neq s}}^\sigma c_q \Bigg],
  \end{aligned}
\end{equation}

    \item $\mathcal{X}$ represents the total number of possible ways to glue $C'_{\underline{\sigma}}$ and $\mathcal{S}_\sigma$ together:
    \[\mathcal{X}= \prod_{\substack{\iota \in \tilde{\delta}}} \#\phi_{\iota} \cdot \prod_{\substack{\iota \in {\delta}}} \#\psi_{\iota }.
    \]
      such that, for $\iota \in {\delta}$, $\, \psi_{\iota} := \{ 1 \leq j \leq \sigma,   1\leq i \leq d'_j :
      c_\iota = w_{ji} \},  $
      and for $\iota \in \Tilde{\delta}$,  $\, \phi_{\iota}:= \{ 1 \leq i \leq d'_\iota : w_{\iota i}=c_\iota \}$.
\end{itemize}

\end{theorem}

\begin{proof}
To count the number of tropical curves of degree \(\Delta_{\mathbb{F}_r}(a, b, w_{\underline{d}})\), passing through $n$ point conditions, we need to find the number of tropical stable maps contributing to $\mathcal{N}^{0}_{\Delta_{\mathbb{F}_r}(a, b, w_{\underline{d}})}(p_{\underline{n}})$.
Let  \( L_a \) and \( L_b \) be two multi-line conditions such that \( p_n \in L_a \cap L_b \). Let $C = (\Gamma, x_{\underline{m}}, \Tilde{e}_{\underline{d}}, h) $ be a tropical stable map contributing to $\mathcal{N}^{0}_{\Delta_{\mathbb{F}_r}(a, b, w_{\underline{d}})}(p_{\underline{n-1}}, L_a, L_b, \lambda'_A)$, where \( h(x_i) = p_i \) for \( i = 1, \cdots, n-1 \), and \( h(x_n) \in L_a \cap h(\Gamma) \), \( h(x_{n+1}) \in L_b \cap h(\Gamma) \) and assume \( \lambda'_A = (x_n x_{n+1} \mid x_c x_d) \) with \( 1 \leq c,d \leq n-1 \). Since we have \( n-1 \) point conditions and \( k = 2 \), it follows that \( m = n + 1 \) and \( l = 1 \), so we can apply Theorem \ref{Main0}.

We claim that the recursive formula in part (1) of Theorem \ref{Main0} reduces in this case to:

 \begin{equation}\label{nabA}
            \mathcal{N}^{0}_{\Delta_{\mathbb{F}_r}(a, b, w_{\underline{d}})}(p_{\underline{n-1}}, L_{a}, L_b, {\lambda'}_{A}) = 
    \end{equation}

    \begin{align}
    &\sum_{\substack{
        (\Delta_{1}, n_{1} | \Delta_{2}, n_{2}) \\ 
        \text{is a 1/1 split respecting } {\lambda}'_l
    }} 
    \mathcal{N}^{0}_{\Delta_{1}}(p_{\underline{n_1}}, L_{a}, L_b, L_{e_1}) 
    \cdot 
    \mathcal{N}^{0}_{\Delta_{2}}(p_{\underline{n_2}},  L_{e_2}) 
    +r\mathcal{N}^{0}_{\Delta_{\mathbb{F}_r}(a, b, w_{\underline{d}})}(p_{\underline{n-1}}, p_n)
    \label{Napears}\\ 
    &+
     \sum_{\sigma=1}^{ra+b} 
      \sum_{\mathcal{S}_\sigma} \mult(\mathcal{S}_\sigma)
       \sum_{\substack{
        \{ (\Delta_{\mathbb{F}_r}(a'_i, b'_i, w_{i\underline{d'_i}}), n_i \}_{i=1}^{\sigma}  \\ 
        \sum_{i=0}^{\sigma}(a'_i ,d'_i) = (a ,d+\sigma) }}
     \varphi \cdot  \mathcal{X} \cdot  
  Y, \label{YA}
    \end{align}

such that $Y$ is a factor with three summands \eqref{Y1}, \eqref{Y2}, and \eqref{Y3}.
According to Proposition~\ref{*}, the curve \( C \) can be decomposed into tropical stable maps of lower degrees over cutting a contracted bounded edge or cutting a movable branch. Since we only have one cross-ratio condition, either \( C \) splits into two components \( C_i = (\Gamma_i, x_{\underline{m_i}}, h_i) \) for \( i=1,2 \) over a 1/1 contracted bounded edge, or it contains a movable branch. On the left hand side of \ref{nabA}, in the first summand, \( L_a \) and \( L_b \) intersect \( C_1 \) at distinct points, while in the second summand, both lines intersect at \( p_n \in L_a \cap L_b \), and we have \( \#(L_a \cap L_b) = r \). So, we obtain the $ r\mathcal{N}^{0}_{\Delta_{\mathbb{F}_r}(a, b, w_{\underline{d}})}(p_{\underline{n-1}}, p_n) = r\mathcal{N}^{0}_{\Delta_{\mathbb{F}_r}(a, b, w_{\underline{d}})}(p_{\underline{n}}) $ as the second summand here which is the number we wanted to count from the beginning.

In the first summand, in order to consider the number of intersection points of $L_{e_i}$ with $C_i$  we should notice that when we split the 1/1 contracted bounded edge $e$ into two contracted ends $e_1, e_2$, for $i=1,2$,  \( L_{e_i} \) is a multi-line condition of weight one that is imposed on $e_i$ where $p = \ev_{e}(C) \in Y_1 \cap Y_2$.
More precisely, the multi-set of ends of $Y_i$ is equal to the degree of $C_i$, $ \Delta_i=\Delta_{\mathbb{F}_r}(a_i, b_i, w_{i\underline{d_i}}) $, for $i=1,2$ and therefore we have $r a_1 a_2 + a_1 b_2 + a_2 b_1$ possible choices for $p$ as $L_{e_i}$ intersect with $h_i(\Gamma_i)$ at $p$ so we only need to consider $\#(Y_1 \cap Y_2)$. Here, $n_i = \#\Delta_i-1 $ and instead of taking the sum over $  (\Delta_1, n_{1} | \Delta_2, n_{2})$ we compute all of possible ways to split $n$ into $n_1, n_2$. 
We have $n-1$ point conditions and $x_c, x_d$, for \( 1 \leq c,d \leq n-1 \), live on $C_2$, so we have $\binom{n-3}{n_1 -1} $ many choices to pick $n_1$ point conditions, and the rest of the point conditions are on $C_2$. 
Hence, we obtain the following from the first summand, on the left side of \eqref{nabA}:

\begin{equation*}
\sum_{\substack{
        (\Delta_{1}| \Delta_{2}) \\ 
        \text{ is a 1/1 split}}}  
  (ra_1 + b_1)^2(a_1 b_2 + a_2 b_1 +ra_1 a_2) \binom{n-3}{n_1 -1}
    \mathcal{N}^{0}_{\Delta_1}(p_{\underline{n_1}}) 
    \cdot 
    \mathcal{N}^{0}_{\Delta_2}(p_{\underline{n_2}}). 
\end{equation*}

Moreover, the summand corresponding to cutting movable branches in \eqref{3} simplifies here to \eqref{YA}, where $Y$ is defined as follows:

\begin{align}
      Y &= (c_t+c_s) \cdot
        \mathcal{N}^{0}_{\Delta_{\mathbb{F}_r}(a'_t, b'_t, w_{t\underline{d'_t}})}(p_{\underline{n_t}}, L_a, L_b)\cdot 
  \mathcal{N}^{0}_{\Delta_{\mathbb{F}_r}(a'_s, b'_s, w_{s\underline{d'_s}})}(p_{\underline{n_s}})   \prod_{\substack{i=1\\i\neq s,t}}^{\sigma} c_i \cdot
   \mathcal{N}^{0}_{\Delta_{\mathbb{F}_r}(a'_i, b'_i, w_{i\underline{d'_i}})}(p_{\underline{n_i}}) 
  \label{Y1} \\&+  \mathcal{N}^{0}_{\Delta_{\mathbb{F}_r}(a'_t, b'_t, w_{t\underline{d'_t}})}(p_{\underline{n_t}}, L_a) \cdot
   \mathcal{N}^{0}_{\Delta_{\mathbb{F}_r}(a'_u, b'_u, w_{u\underline{d'_u}})}(p_{\underline{n_u}}, L_b) \cdot \prod_{\substack{q=1\\q\neq s}}^{\sigma} c_q \cdot \prod_{\substack{i=1\\i\neq u,t}}^{\sigma} 
   \mathcal{N}^{0}_{\Delta_{\mathbb{F}_r}(a'_i, b'_i, w_{i\underline{d'_i}})}(p_{\underline{n_i}})  
   \label{Y2}\\&+ \mathcal{N}^{0}_{\Delta_{\mathbb{F}_r}(a'_u, b'_u, w_{u\underline{d'_u}})}(p_{\underline{n_u}}, L_a, L_b)\cdot 
 \prod_{\substack{i=1\\i\neq u}}^{\sigma} c_i \cdot
   \mathcal{N}^{0}_{\Delta_{\mathbb{F}_r}(a'_i, b'_i, w_{i\underline{d'_i}})}(p_{\underline{n_i}}).\label{Y3}
\end{align}

We need to specify the position of $x_c, x_d$ by determining all possible ways to split $n$ into $n_{\underline{\sigma}}$, where $n = \sum_{i=1}^{\sigma}n_{i}$. When $C$ is a tropical stable map obtained from gluing a movable branch $\mathcal{S}_\sigma$ with $\sigma$ fixed components and $C$ satisfies a cross-ratio condition \( \lambda'_A = (x_n x_{n+1} \mid x_c x_d) \), then either we are in case (A) or (B), defined in Theorem \ref{Main0}. So if we have more than 3 fixed components, we can chose $\binom{\sigma}{3}$ fixed components $1\leq s,u,t\leq \sigma$, and then there are three possible ways to distribute the marked points $x_n, x_{n+1}, x_c, x_d$ among the $C'
_s, C'_t, C'_u$. In \eqref{Y1}, \eqref{Y2}, and \eqref{Y3}, the four marked points are distributed as shown in the left, middle, and right pictures of Figure \ref{stu}, respectively.

\begin{figure}[h!]
\centering
\includegraphics[scale=0.8]{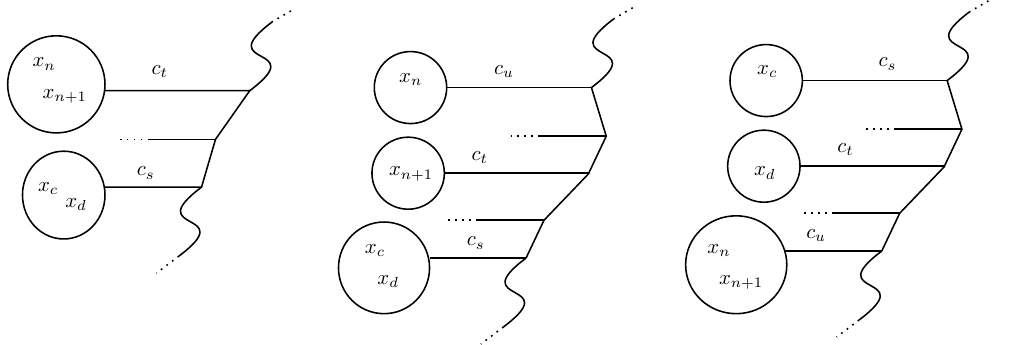}
\caption{The three possible ways to distribute the marked points $x_n, x_{n+1}, x_c, x_d$ among the fixed components $C'_s, C'_t, C'_u$, such that \( \lambda'_A = (x_n x_{n+1} \mid x_c x_d) \) holds.}
\label{stu}
\end{figure}

When $x_c, x_d$ live on the same fixed component, we have
\begin{equation*}
    \binom{n-3}{n_s-2}\prod\limits_{\substack{i=1 \\ i \neq s}}^\sigma\binom{n-n_s-1-\sum_{\substack{j=1 \\ j \neq s,t}}^{i}n_i}{n_i}
\end{equation*}
ways to distribute $n_{\underline{\sigma}}$ between the fixed components. When $x_c, x_d$ live on different fixed components we have
\begin{equation*}
    \binom{n-3}{n_s-1}\binom{n-n_s-2}{n_t-1}\prod\limits_{\substack{i=1 \\ i \neq s}}^\sigma\binom{n-n_s-n_t-1-\sum_{\substack{j=1 \\ j \neq s,t}}^{i}n_i}{n_i}
\end{equation*}
possible ways to distribute the point conditions. Therefore, we can rewrite \eqref{YA} in this way:

\begin{align}
&\sum_{\sigma=1}^{ra+b} 
      \sum_{\mathcal{S}_\sigma} \mult(\mathcal{S}_\sigma)
       \sum_{\substack{
        \{ (\Delta_{\mathbb{F}_r}(a'_i, b'_i, w_{i\underline{d'_i}})\}_{i=1}^{\sigma}  \\ 
           \sum_{i=0}^{\sigma}(a'_i ,d'_i) = (a ,d+\sigma)  }}
    \varphi \cdot   
    \mathcal{X} \cdot \Bigg[ \binom{n-3}{n_s-2}\prod\limits_{\substack{i=1 \\ i \neq s}}^\sigma\binom{n-n_s-1-\sum_{\substack{j=1 \\ j \neq s,t}}^{i}n_i}{n_i} \\& \Bigg[ (ra_t + b_t)^2(c_t+c_s)\prod\limits_{\substack{q=1 \\ q \neq s,t}}^\sigma c_q + (ra_t + b_t)(ra_u + b_u)\prod\limits_{\substack{q=1 \\ q \neq s}}^\sigma c_q \Bigg] + \binom{n-3}{n_s-1}\binom{n-n_s-2}{n_t-1} \\&\prod\limits_{\substack{i=1 \\ i \neq s}}^\sigma\binom{n-n_s-n_t-1-\sum_{\substack{j=1 \\ j \neq s,t}}^{i}n_i}{n_i}(ra_u + b_u)^2\prod\limits_{\substack{q=1 \\ q \neq u}}^\sigma c_q   \Bigg]
    \prod_{i=1}^{\sigma} 
   \mathcal{N}^{0}_{\Delta_{\mathbb{F}_r}(a'_i, b'_i, w_{i\underline{d'_i}})}(p_{\underline{n_i}}).  
\end{align}

Similarly, if we consider another cross-ratio \( \lambda'_B = (x_n x_c \mid x_{n+1} x_d) \), we obtain the following:
    
    \begin{align*}
    &\mathcal{N}^{0}_{\Delta_{\mathbb{F}_r}(a, b, w_{\underline{d}})}(p_{\underline{n-1}}, L_{a}, L_b, {\lambda'}_{B}) = \sum_{\substack{
        (\Delta_{1}, n_{1} | \Delta_{2}, n_{2}) \\ 
        \text{is a 1/1 split respecting } {\lambda}'_l
    }} 
    \mathcal{N}^{0}_{\Delta_{1}}(p_{\underline{n_1}}, L_{a}, L_{e_1}) 
    \cdot 
    \mathcal{N}^{0}_{\Delta_{2}}(p_{\underline{n_2}}, L_b, L_{e_2}) 
    \\&+
    \sum_{\sigma=1}^{ra+b} 
      \sum_{\mathcal{S}_\sigma} \mult(\mathcal{S}_\sigma)
       \sum_{\substack{
        \{ (\Delta_{\mathbb{F}_r}(a'_i, b'_i, w_{i\underline{d'_i}}), n_i \}_{i=1}^{\sigma}  \\ 
          \sum_{i=0}^{\sigma}(a'_i ,d'_i) = (a ,d+\sigma)  }}
    \mathcal{X} \cdot  
   \varphi \cdot \Bigg[ (c_t+c_s) \cdot
        \mathcal{N}^{0}_{\Delta_{\mathbb{F}_r}(a'_t, b'_t, w_{t\underline{d'_t}})}(p_{\underline{n_t}}, L_a)\cdot \\&
  \mathcal{N}^{0}_{\Delta_{\mathbb{F}_r}(a'_s, b'_s, w_{s\underline{d'_s}})}(p_{\underline{n_s}}, L_b)   \prod_{\substack{i=1\\i\neq s,t}}^{\sigma} c_i \cdot
   \mathcal{N}^{0}_{\Delta_{\mathbb{F}_r}(a'_i, b'_i, w_{i\underline{d'_i}})}(p_{\underline{n_i}}) + \mathcal{N}^{0}_{\Delta_{\mathbb{F}_r}(a'_s, b'_s, w_{s\underline{d'_s}})}(p_{\underline{n_s}}, L_a) \cdot \\&
   \mathcal{N}^{0}_{\Delta_{\mathbb{F}_r}(a'_u, b'_u, w_{u\underline{d'_u}})}(p_{\underline{n_u}}, L_b) \cdot  c_u \cdot \prod_{\substack{i=1\\i\neq u,s}}^{\sigma} c_i \cdot
   \mathcal{N}^{0}_{\Delta_{\mathbb{F}_r}(a'_i, b'_i, w_{i\underline{d'_i}})}(p_{\underline{n_i}}) \Bigg].
    \end{align*}
After substituting the factors corresponding to all possible distributions of the point conditions, and taking into account the number of intersection points of $L_a$ and $L_b$ with the fixed components, we use the identity
\[
\mathcal{N}^{0}_{\Delta_{\mathbb{F}_r}(a, b, w_{\underline{d}})}(p_{\underline{n-1}}, L_a, L_b, \lambda'_A) = \mathcal{N}^{0}_{\Delta_{\mathbb{F}_r}(a, b, w_{\underline{d}})}(p_{\underline{n-1}}, L_a, L_b, \lambda'_B)
\]
to derive Equation~\eqref{NFrw}, which completes the proof.

\end{proof}

\begin{corollary}
    \label{Frn}

For $ n= \# \Delta_{\mathbb{F}_{r}}(a,b)-1$, the number of tropical curves of degree \(\Delta_{\mathbb{F}_{r}}(a,b)\), passing through $n$ point conditions can be computed recursively using the following formula. This formula is derived from Theorem~\ref{Frnw} where $b=d$ and $w_i=1$ for $1\leq i \leq b$:

 \begin{equation}\label{NFr}
\mathcal{N}^{0}_{\Delta_{\mathbb{F}_{r}}(a,b)}(p_{\underline{n}}) =     \frac{1}{b!}\mathcal{N}^{0}_{\Delta_{\mathbb{F}_r}(a, b, w_{\underline{b}})}(p_{\underline{n}})
    \end{equation}
\end{corollary}

In the next example, it is shown how we can use Theorem \ref{Main0} to count tropical stable maps when more than one movable branch appears.

\begin{example}\label{ExF3}
We want to count the number of tropical stable maps  
\[
C = (\Gamma, x_{\underline{9}}, \Tilde{e}_1, h)
\]  
of degree \(\Delta_{\mathbb{F}_3}(2,1,1)\), passing through 9 point conditions  
\(p_{\underline{9}}\), where \(h(x_i) = p_i\) for \(1 \leq i \leq 9\),  
and satisfying the cross-ratio conditions \(\lambda_1 = \{x_4, x_5, x_6, \Tilde{e}_1\}\)  
and \(\lambda'_2 = (x_1, x_2 \mid x_3, \Tilde{e}_1)\).

First, we find all possible ways to split \(C\) over a 2/0 contracted bounded edge.  
Consider all \(a_1, a_2\) and \(d_1, d_2\) such that  
\[
a_1 + a_2 = 2, \quad d_1 + d_2 = 1.
\]
Since one of the fixed components has a marked right end \(\Tilde{e}_1\) of weight 1; therefore, \(b_1 + b_2 = 1\).

Moreover, if we consider all possible distributions of the conditions \(p_{\underline{9}}\), \(\lambda_1\) between the two fixed components \(C_1, C_2\), such that  
\[
n_1 + n_2 + 1 = 10 \quad \text{and} \quad \dim Y_1 = 2, \; \dim Y_2 = 0,
\]  
then:
\[
\#\Delta_1 - 2 = n_1, \quad \#\Delta_2 - 1 = n_2.
\]
Therefore, we obtain the following list of possible degrees for \(C_1, C_2\):

\begin{equation}\label{2/0table}
    \begin{array}{cccc|cccc}
       \Delta_1 & \#\Delta_1 & n_1 & l_1 & \Delta_2 & \#\Delta_2 & n_2 & l_2\\ \hline
       \Delta_{\mathbb{F}_3}(2,0) & 10 & 8 & 0 & \Delta_{\mathbb{F}_3}(0,1,1)  & 2 & 1 & 1\\
        \Delta_{\mathbb{F}_3}(0,1,1) & 2 & 0 & 0 & \Delta_{\mathbb{F}_3}(2,0) & 10 & 9 &  1\\
         \Delta_{\mathbb{F}_3}(1,0) & 5 & 3 & 0 &\Delta_{\mathbb{F}_3}(1,1,1)  & 7 & 6 & 1\\
         \Delta_{\mathbb{F}_3}(1,1,1) & 7 & 5 & 0 & \Delta_{\mathbb{F}_3}(1,0) & 5 & 4 &  1\\
    \end{array}
\end{equation}
Notice that there must be 4 marked ends on a fixed component to have a degenerate cross-ratio defined on that component. Three out of 4 marked ends must be among the marked points and \(\Tilde{e}_1\). Therefore, the first row of Table~\ref{2/0table} is not possible. (In the other tables, we have already eliminated the impossible cases.)

If we consider all possible ways to pick \(n_1\) point conditions out of \(p_{\underline{9}}\), while either \(x_1, x_2\) or \(x_3, \Tilde{e}_1\) must lie on \(C_1\), then we realize that the second row of Table~\ref{2/0table} is impossible because \(n_1 = 0\).

Let \(\{p_i\}^{\#N}\) be a set of \(N \leq 9\) point conditions in \(p_{\underline{9}}\).  
In the last row, for \(C_2\) we have  
\[
\mathcal{N}^{0}_{\Delta_{\mathbb{F}_{3}}(1,0)}\left(\{p_i\}^{\#4}, \lambda_{1}^{\rightarrow e_2}\right),
\]  
which is impossible because, in order for the cross-ratio conditions \(\lambda_1\) and \(\lambda'_2\) to hold, we would need \(p_1, p_2, p_4, p_5, p_6\) to lie on \(C_2\), which is not possible.

So what remains is:

\begin{equation}\label{2/0exeq}
    \begin{aligned}
    & \sum_{\substack{
        (\Delta_{1}, n_{1} | \Delta_{2}, n_{2}) \\ 
        \text{is a 2/0 split respecting } {\lambda}'_2 \\ 
        \text{and } \#\Delta_2 = n_2 + 1 
    }} 
    \mathcal{N}^{0}_{\Delta_{1}}(p_{\underline{n_1}}, p) 
    \cdot 
    \mathcal{N}^{0}_{\Delta_{2}}(p_{\underline{n_2}}, {\lambda}_{1}^{\rightarrow e}) = \\
    & \binom{6}{1}\mathcal{N}^{0}_{\Delta_{\mathbb{F}_{3}}(1,0)}(p_{1},p_2, \{p_i\}^{\#1}, p)\cdot \mathcal{N}^{0}_{\Delta_{\mathbb{F}_{3}}(1,1,1)}(p_3, \{p_i\}^{\#5}, {\lambda}_{1}^{\rightarrow e}).
\end{aligned}
\end{equation}
 Here, out of the 9 point conditions, 3 of them—\(p_1, p_2\), and \(p_3\)—are fixed. The number of possible ways to pick the third point condition for \(C_1\) is \(\binom{6}{1}\). As discussed in Section \ref{Split curves}, and shown by the equalities in \ref{hiddenf}, in the case of a 2/0 split, when we have a \({\lambda}^{\rightarrow e}\), the marked end \(e\) is an f-point.
 In the case of a 2/0 split, when we have a \({\lambda}^{\rightarrow e}\), the marked end \(e\) is an f-point.

Now we want to find all possible ways to split \(C\) over a 1/1 contracted bounded edge.  
The following list contains all possible degrees for \(C_1, C_2\), where \(\dim Y_1 = \dim Y_2 = 1\):

\begin{equation}\label{1/1table}
    \begin{array}{ccccc|ccccc}
       \Delta_1 & \#\Delta_1 & k_1 & n_1 & l_1 & \Delta_2 & \#\Delta_2 & k_2 & n_2 & l_2\\ \hline
       \Delta_{\mathbb{F}_3}(2,0) & 10 & 1 & 8 & 1 & \Delta_{\mathbb{F}_3}(0,1,1) & 2 & 1 & 1 & 0 \\
        \Delta_{\mathbb{F}_3}(1,0) & 5 & 1 & 4 & 0 & \Delta_{\mathbb{F}_3}(1,1,1) & 7 & 1 & 5 & 1 
    \end{array}
\end{equation}
In this case, after considering all possible ways to pick \(n_1\) point conditions out of \(p_{\underline{9}}\) such that \(x_1, x_2\) have to lie on \(C_1\), we obtain \eqref{1/1exeq}. Notice that the two fixed components intersect in \(a_1 b_2 + a_2 b_1 + 3 a_1 a_2\) points.

\begin{equation}\label{1/1exeq}
    \begin{aligned}
    & \sum_{\substack{
        (\Delta_{1}, n_{1}, l_1 | \Delta_{2}, n_{2}, l_2) \\ 
        \text{is a 1/1 split respecting } {\lambda}'_2 
    }} 
    \mathcal{N}^{0}_{\Delta_{1}}(p_{\underline{n_1}}, L_{e_1},  {\lambda}_{\underline{l_1}}^{\rightarrow e_1}) 
    \cdot 
    \mathcal{N}^{0}_{\Delta_{2}}(p_{\underline{n_2}}, L_{e_2},  {\lambda}_{\underline{l_2}}^{\rightarrow e_2})  
    \\&=  2\cdot\binom{6}{6}\mathcal{N}^{0}_{\Delta_{\mathbb{F}_{3}}(2,0)}(p_{1},p_2,\{p_i\}^{\#6}, L_{e_1}, {\lambda}_{1}) \cdot \mathcal{N}^{0}_{\Delta_{\mathbb{F}_{3}}(0,1,1)}(p_3, L_{e_2}) 
    \\&+4\cdot\binom{3}{1}\binom{3}{1}\mathcal{N}^{0}_{\Delta_{\mathbb{F}_{3}}(1,0)}(p_{1},p_2, \{p_i\}^{\#2}, L_{e_1}) \cdot \mathcal{N}^{0}_{\Delta_{\mathbb{F}_{3}}(1,1,1)}(p_3, \{p_i\}^{\#4}, L_{e_2}, {\lambda}_{1}^{\rightarrow e_2})
    \\&+4\cdot\binom{3}{1}\binom{2}{1}\mathcal{N}^{0}_{\Delta_{\mathbb{F}_{3}}(1,0)}(p_{1},p_2, \{p_i\}^{\#2}, L_{e_1}) \cdot \mathcal{N}^{0}_{\Delta_{\mathbb{F}_{3}}(1,1,1)}(p_3, \{p_i\}^{\#4}, L_{e_2}, {\lambda}_{1}) 
    \\&= 2\cdot \mathcal{N}^{0}_{\Delta_{\mathbb{F}_{3}}(2,0)}(p_{1},p_2,\{p_i\}^{\#6}, \lambda_1) \cdot \mathcal{N}^{0}_{\Delta_{\mathbb{F}_{3}}(0,1,1)}(p_3) 
    \\&+ 4\cdot9\cdot\mathcal{N}^{0}_{\Delta_{\mathbb{F}_{3}}(1,0)}(p_{1},p_2, \{p_i\}^{\#2}) \cdot \mathcal{N}^{0}_{\Delta_{\mathbb{F}_{3}}(1,1,1)}(p_3, \{p_i\}^{\#4}, L_{e_2}, {\lambda}_{1}^{\rightarrow e_2})
    \\&+ 4\cdot6\cdot\mathcal{N}^{0}_{\Delta_{\mathbb{F}_{3}}(1,0)}(p_{1},p_2, \{p_i\}^{\#2}) \cdot \mathcal{N}^{0}_{\Delta_{\mathbb{F}_{3}}(1,1,1)}(p_3, \{p_i\}^{\#4}, \lambda_1).
\end{aligned}
\end{equation}

As discussed in Section \ref{Split curves}, where we have a \({\lambda}^{\rightarrow e_i}\), we need to keep \(L_{e_i}\), because 
\[
{\lambda}^{\rightarrow e_i} = \{ x_{\alpha}, x_{\beta}, x_{\eta}, e_i \},
\]
where \(\alpha, \beta, \eta \in \underline{m+d}\). This is why we consider two cases separately: when three marked ends contributing to \(\lambda_1\) lie on one component and the fourth one lies on another component (as in the second summand in \eqref{1/1exeq}), or when all marked ends lie on one component (as in the first and third summands in \eqref{1/1exeq}).


All possible movable branches \(\mathcal{S}_\sigma\) for \(1 \leq \sigma \leq ra + b = 7\) are shown in Figure \ref{S7}. Since we have only one marked right end of weight 1 and only one degenerate cross-ratio condition, we cannot have \(\mathcal{S}_{1}, \mathcal{S}_{5}, \mathcal{S}_{6}, \mathcal{S}_{7}\).

\begin{figure}[h!]
    \centering
    \includegraphics[scale=0.7]{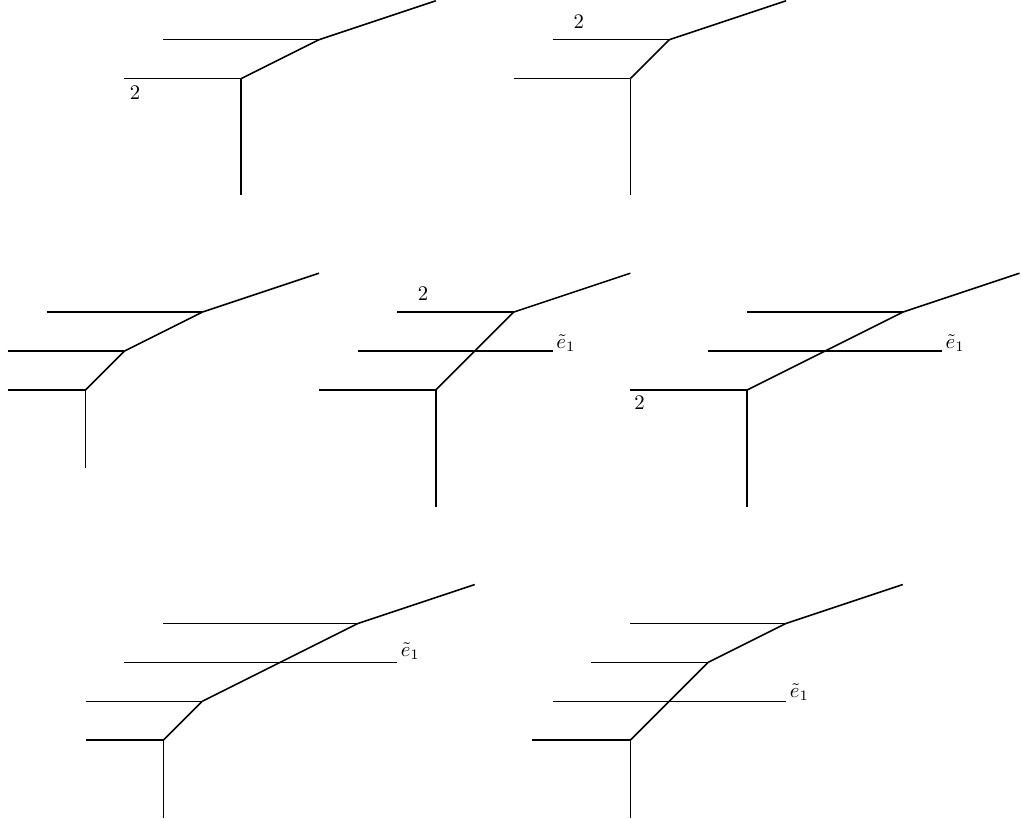}
    \caption{All possible movable branches for a tropical stable map contributing to 
    \(\mathcal{N}^{0}_{\Delta_{\mathbb{F}_{3}}(2,1,1)}(p_{\underline{9}}, \lambda_{1}, \lambda'_{2})\), where \(|\lambda'_2|\) has a very long length.}
    \label{S7}
\end{figure}

Let \(\sigma=2\). As shown in Figure \ref{S7}, the weights of the left ends of \(\mathcal{S}_2\) are \(c_1, c_2 \in \{1, 2\}\) such that \(c_1 + c_2 = 3\). We need to find all \(a'_1, a'_2\) and \(d'_1, d'_2\) satisfying 
\(
a'_1 + a'_2 = 1,  d'_1 + d'_2 = 3
\).


If \(a'_1 = 0\) (or similarly \(a'_2 = 0\)) and we have more than one point condition on a fixed component of degree \(\Delta_{\mathbb{F}_3}(0, b'_1, w_{1\underline{d'_1}})\), then these point conditions must lie on a line and thus cannot be in general position, which leads to a contradiction. Hence, if \(a'_i = 0\), we must have \(n_i = 1\), and therefore \(l_i = 0\).

On the one hand, to satisfy \(\lambda'_2\), the marked right end \(\tilde{e}_1\) and the point condition \(p_3\) must lie on the fixed component of degree \(\Delta_{\mathbb{F}_3}(0, b'_i, w_{i\underline{d'_i}})\). On the other hand, the images of both \(C'_1\) and \(C'_2\) must contain marked right ends in addition to \(\tilde{e}_1\); otherwise, they cannot be glued to \(\mathcal{S}_2\).
So the fixed component of degree \(\Delta_{\mathbb{F}_3}(0, b'_i, w_{i\underline{d'_i}})\) corresponds to a vertex adjacent to four parallel ends, including two ends in the primitive direction \((1,0)\). Since there is only one point condition on this component, no cross-ratio can be satisfied at this 4-valent vertex, which is impossible. Therefore, no pair of fixed components satisfies these conditions.

    


Let $\sigma=3$. As we can see in Figure \ref{S7}, the sum of the weights of the left ends of $\mathcal{S}_3$ can be either 4 or 3.

Using the same argument as for $\mathcal{S}_2$, we can only have one point condition and no cross-ratio on the fixed components of degree $\Delta_{\mathbb{F}_3}(0, b'_i, w_{i\underline{d'_i}})$. However, when $\sum_{i=1}^3 c_i = 4$ and for each $i = 1, \cdots, 3$ we have $c_i \in \{1, 2\}$, since $C$ satisfies both $\lambda_1$ and $\lambda'_2$, there must be at least two point conditions — the point condition $p_3$ and one of $p_4, p_5, p_6$ — on a fixed component of degree $\Delta_{\mathbb{F}_3}(0, b'_i, w_{i\underline{d'_i}})$, which is impossible.

Therefore, the only possible case is when $\sum_{i=1}^3 c_i = 3$ and $c_i = 1$ for each $i = 1, \cdots, 3$, and we obtain:

 \begin{equation}
\begin{aligned}
&\sum_{\substack{
        \{\Delta_{\mathbb{F}_{3}}(a'_i, b'_i, w_{i\underline{d'_i}}), n_{i}, l_i\}_{i=1}^3 \\ 
        \text{where } \sum_{i=0}^3(a'_i, d'_i)=(2,4)
    }} \mathcal{Y}\cdot \mathcal{X}\cdot
    \prod_{i=1}^3
    \mathcal{N}^{0}_{\Delta_{\mathbb{F}_{3}}(a'_i, b'_i, w_{i\underline{d'_i}})}(\{p_i\}^{\#n_i}, {\lambda}_{\underline{l_i}}) = \\&
    2\cdot\mathcal{N}^{0}_{\Delta_{\mathbb{F}_3}(1,2,1,1)}(\{p_i\}^{\#7}, \lambda_1) \cdot \mathcal{N}^{0}_{\Delta_{\mathbb{F}_3}(0,1,1)}(p_1)\cdot \mathcal{N}^{0}_{\Delta_{\mathbb{F}_3}(0,1,1)}(p_2).       
\end{aligned}\label{S3eq}
\end{equation}
Since we are in Case (A) of Theorem \ref{Main0}, we have $\mathcal{Y}=1$. Only one fixed component has two marked right ends of the same weight. We can glue this fixed component to \(\mathcal{S}_3\) in two ways, so here \(\mathcal{X} = 2\).

Let $\sigma = 4$, as we can see in Figure \ref{S7}, the sum of the weights of the left ends of $\mathcal{S}_2$ is $\sum_{i=1}^4 c_i = 4$ where $c_i = 1$ for each $i$. Therefore, we need to find all $a'_{\underline{4}}$, $d'_{\underline{4}}$ such that:
\[
\sum_{i=1}^4 a'_i = 1, \quad \sum_{i=1}^4 d'_i = 4.
\]
If we consider all possible distributions of the point conditions $p_{\underline{9}}$ among the four fixed components such that $\sum_{i=1}^4 n_i = 9$, we obtain:

 \begin{equation}
\begin{aligned}
    &\sum_{\substack{
        \{\Delta_{\mathbb{F}_{3}}(a'_i, b'_i, w_{i\underline{d'_i}}), n_{i}\}_{i=1}^4 \\ 
        \text{where } \sum_{i=0}^4(a'_i, d'_i)=(2,5)
    }} \mathcal{Y}\cdot \mathcal{X}\cdot
    \prod_{i=1}^4
    \mathcal{N}^{0}_{\Delta_{\mathbb{F}_{3}}(a'_i, b'_i, w_{i\underline{d'_i}})}(\{p_i\}^{\#n_i}) = 
    \\& \binom{3}{2}
    \mathcal{N}^{0}_{\Delta_{\mathbb{F}_3}(1,1,1)}(\{p_i\}^{\#6}) \cdot \mathcal{N}^{0}_{\Delta_{\mathbb{F}_3}(0,1,1)}(p_3)\cdot \mathcal{N}^{0}_{\Delta_{\mathbb{F}_3}(0,1,1)}(\{p_i\}^{\#1})\cdot \mathcal{N}^{0}_{\Delta_{\mathbb{F}_3}(0,1,1)}(\{p_i\}^{\#1}).
    \end{aligned}\label{S4eq}
\end{equation}
Here again, $\mathcal{Y}=1$. Since all of the fixed components have one marked right end, we have \(\mathcal{X} = 1\).
But note that there are two possible \(\mathcal{S}_4\) in Figure \ref{S7}. 
To ensure that \(\lambda_1\) holds, the point conditions \(\{p_4, p_5, p_6\}\) must be placed on distinct fixed components. Moreover, to satisfy \(\lambda'_2\), the point conditions \(\{p_1, p_2\}\) must lie on the same fixed component. Therefore, they must be assigned to the fixed component of degree \(\Delta_{\mathbb{F}_3}(1,1,1)\).
So we need to choose 2 point conditions from \(\{p_4, p_5, p_6\}\) and assign them to two of the fixed components of degree \(\Delta_{\mathbb{F}_3}(0,1,1)\), which gives a combinatorial factor of \(\binom{3}{2}\). The point condition \(p_3\) must be assigned to the third fixed component of degree \(\Delta_{\mathbb{F}_3}(0,1,1)\), and the remaining point conditions will be placed on the fixed component of degree \(\Delta_{\mathbb{F}_3}(1,1,1)\).

Therefore, we can count the number of tropical stable maps of degree \(\Delta_{\mathbb{F}_3}(2,1,1)\), which pass through 9 point conditions  
\(p_{\underline{9}}\)  
and satisfy the cross-ratio conditions \(\lambda_1 = \{x_4, x_5, x_6, \Tilde{e}_1\}\)  
and \(\lambda'_2 = (x_1, x_2 \mid x_3, \Tilde{e}_1)\), using the following formula:

 \begin{equation}
\begin{aligned}
    & \sum_{\substack{
        (\Delta_{1}, n_{1} | \Delta_{2}, n_{2}) \\ 
        \text{is a 2/0 split respecting } {\lambda}'_2 \\ 
        \text{and } \#\Delta_2 = n_2 + 1 
    }} 
    \mathcal{N}^{0}_{\Delta_{1}}(p_{\underline{n_1}}, p) 
    \cdot 
    \mathcal{N}^{0}_{\Delta_{2}}(p_{\underline{n_2}}, {\lambda}_{1}^{\rightarrow e}) 
\\&+
\sum_{\substack{
        (\Delta_{1}, n_{1}, l_1 | \Delta_{2}, n_{2}, l_2) \\ 
        \text{is a 1/1 split respecting } {\lambda}'_2 
    }} 
    \mathcal{N}^{0}_{\Delta_{1}}(p_{\underline{n_1}}, L_{e_1},  {\lambda}_{\underline{l_1}}^{\rightarrow e_1}) 
    \cdot 
    \mathcal{N}^{0}_{\Delta_{2}}(p_{\underline{n_2}}, L_{e_2},  {\lambda}_{\underline{l_2}}^{\rightarrow e_2})  \\&+  \sum_{\substack{
        \{\Delta_{\mathbb{F}_{3}}(a'_i, b'_i, w_{i\underline{d'_i}}), n_{i}, l_i\}_{i=1}^3 \\ 
        \text{where } \sum_{i=0}^3(a'_i, d'_i)=(2,4)
    }} 2\cdot
    \prod_{i=1}^3
    \mathcal{N}^{0}_{\Delta_{\mathbb{F}_{3}}(a'_i, b'_i, w_{i\underline{d'_i}})}(\{p_i\}^{\#n_i}, {\lambda}_{\underline{l_i}})\\&+
2\sum_{\substack{
        \{\Delta_{\mathbb{F}_{3}}(a'_i, b'_i, w_{i\underline{d'_i}}), n_{i}\}_{i=1}^4 \\ 
        \text{where } \sum_{i=0}^4(a'_i, d'_i)=(2,5)
    }} 
    \prod_{i=1}^4
    \mathcal{N}^{0}_{\Delta_{\mathbb{F}_{3}}(a'_i, b'_i, w_{i\underline{d'_i}})}(\{p_i\}^{\#n_i})  \\
    &= 6\cdot\mathcal{N}^{0}_{\Delta_{\mathbb{F}_{3}}(1,0)}(p_{1},p_2, \{p_i\}^{\#1}, p)\cdot \mathcal{N}^{0}_{\Delta_{\mathbb{F}_{3}}(1,1,1)}(p_3, \{p_i\}^{\#5}, {\lambda}_{1}^{\rightarrow e})\\&+2\cdot \mathcal{N}^{0}_{\Delta_{\mathbb{F}_{3}}(2,0)}(p_{1},p_2,\{p_i\}^{\#6}, \lambda_1) \cdot \mathcal{N}^{0}_{\Delta_{\mathbb{F}_{3}}(0,1,1)}(p_3) 
    \\&+ 36\cdot\mathcal{N}^{0}_{\Delta_{\mathbb{F}_{3}}(1,0)}(p_{1},p_2, \{p_i\}^{\#2}) \cdot \mathcal{N}^{0}_{\Delta_{\mathbb{F}_{3}}(1,1,1)}(p_3, \{p_i\}^{\#4}, L_{e_2}, {\lambda}_{1}^{\rightarrow e_2})
    \\&+ 36\cdot\mathcal{N}^{0}_{\Delta_{\mathbb{F}_{3}}(1,0)}(p_{1},p_2, \{p_i\}^{\#2}) \cdot \mathcal{N}^{0}_{\Delta_{\mathbb{F}_{3}}(1,1,1)}(p_3, \{p_i\}^{\#4}, \lambda_1) 
    \\& + 2  \cdot\mathcal{N}^{0}_{\Delta_{\mathbb{F}_3}(1,2,1,1)}(\{p_i\}^{\#7}, \lambda_1) \cdot \mathcal{N}^{0}_{\Delta_{\mathbb{F}_3}(0,1,1)}(p_1)\cdot \mathcal{N}^{0}_{\Delta_{\mathbb{F}_3}(0,1,1)}(p_2) \\&+ 6  \cdot 
    \mathcal{N}^{0}_{\Delta_{\mathbb{F}_3}(1,1,1)}(\{p_i\}^{\#6}) \cdot \mathcal{N}^{0}_{\Delta_{\mathbb{F}_3}(0,1,1)}(p_3)\cdot \mathcal{N}^{0}_{\Delta_{\mathbb{F}_3}(0,1,1)}(\{p_i\}^{\#1})\cdot \mathcal{N}^{0}_{\Delta_{\mathbb{F}_3}(0,1,1)}(\{p_i\}^{\#1}).
    \end{aligned}\label{ExS7eq}
\end{equation}
\end{example}

\section{Recovering established results for specific degrees and conditions}\label{Test}

In this section, we explore special cases of the formula presented in Theorem \ref{Main0}, recovering known results under specific conditions.  

\begin{example} 
Consider the case where in \eqref{General}, \( b = d = 0 \) and \( r = 1 \). In this setting, the image of tropical stable maps of degree \( \Delta_{\mathbb{F}_1}(a,0) \) in the plane consists of tropical curves with \( a \) ends in each of primitive directions: \( (1,1) \), \( (-1,0) \), and \( (0,-1) \), in another word, we are counting the tropical stable maps on $\mathbb{P}^2$ of degree $a$. Notice that we have no marked right ends and thus no non-contracted marked ends. As discussed in Section \ref{Prel}, this implies that all marked ends contributing to the cross-ratio conditions are actually marked points. Therefore, the definition of the cross-ratio given in Definition \ref{CR} coincides with the one established in \cite{GM, G, FM}.  
Since the multi-line condition defined in \ref{Line}, also coincides with \cite[Definition 1.8]{G}, so we have the following equality,

\begin{equation}
\mathcal{N}_{a}(p_{\underline{n}}, L_{\underline{k}}, {\lambda}_{\underline{l}})=  \mathcal{N}^{0}_{\Delta_{\mathbb{F}_{1}}(a,0)} (p_{\underline{n}}, L_{\underline{k}}, {\lambda}_{\underline{l}}),
\end{equation}

between the number of rational plane degree $a$ curves in $\mathbb{P}^2$ satisfying general positioned point, curve, and cross-ratio conditions and the number of tropical stable maps contributing to $ \mathcal{N}^{0}_{\Delta_{\mathbb{F}_{1}}(a,0)} (p_{\underline{n}}, L_{\underline{k}}, {\lambda}_{\underline{l}}) $.

\end{example}  
\begin{lemma}[Generalized Kontsevich's formula for $\mathbb{P}^2$ \cite{G}]\label{recoveingG}
 Consider \( n \) points \( p_{\underline{n}} \), \( k \) multi-lines \( L_{\underline{k}} \), and \( l\geq 1 \) cross-ratio conditions \( {\lambda}_{\underline{l}} \), all in general position. The number of \(m\)-marked tropical stable maps of degree $\Delta_{\mathbb{F}_{1}}(a,0)$,  contribute to  
\(\mathcal{N}^{0}_{\Delta_{\mathbb{F}_{1}}(a,0)}(p_{\underline{n}}, L_{\underline{k}}, {\lambda}_{\underline{l}})\)  
is determined recursively by the following formulas. Here, \( m = n + k + f \), where the marked points labeled in \( \underline{f} \) satisfy no additional conditions. Let $\lambda'_l$ denote a cross-ratio that degenerates to $\lambda_l$.

\begin{enumerate}
    \item {For} $n \geq 1$:
    
    \begin{equation}
            \mathcal{N}^{0}_{\Delta_{\mathbb{F}_{1}}(a,0)}(p_{\underline{n}}, L_{\underline{k}}, {\lambda}_{\underline{l}}) = 
    \end{equation}

    \begin{align}
    &\sum_{\substack{
        (\Delta_{1}, n_{1}, k_{1}, l_{1}, f_{1} | \Delta_{2}, n_{2}, k_{2}, l_{2}, f_{2}) \\ 
        \text{is a 1/1 split respecting } {\lambda}'_l
    }} 
    \mathcal{N}^{0}_{\Delta_{1}}(p_{\underline{n_1}}, L_{\underline{k_1}}, L_{e_1}, {\lambda}_{\underline{l_1}}^{\rightarrow e_1}) 
    \cdot 
    \mathcal{N}^{0}_{\Delta_{2}}(p_{\underline{n_2}}, L_{\underline{k_2}}, L_{e_2}, {\lambda}_{\underline{l_2}}^{\rightarrow e_2}) 
    \\ 
    &+ \sum_{\substack{
        (\Delta_{1}, n_{1}, k_{1}, l_{1}, f_{1} | \Delta_{2}, n_{2}, k_{2}, l_{2}, f_{2}) \\ 
        \text{is a 2/0 split respecting } {\lambda}'_l \\ 
        \text{and } \#\Delta_1 = n_1 + l_1 - f_1 
    }} 
    \mathcal{N}^{0}_{\Delta_{1}}(p_{\underline{n_1}}, L_{\underline{k_1}}, {\lambda}_{\underline{l_1}}^{\rightarrow e_1}) 
    \cdot 
    \mathcal{N}^{0}_{\Delta_{2}}(p_{\underline{n_2}},  p, L_{\underline{k_2}}, {\lambda}_{\underline{l_2}}^{\rightarrow e_2}).  
    \end{align}\\

    \item {For} $n = 0$ the equation:

\begin{equation}
    \mathcal{N}^{0}_{\Delta_{\mathbb{F}_{1}}(a,0)}(L_{\underline{k}}, {\lambda}_{\underline{l}}) =    
    \mathcal{N}^{0}_{\Delta_{\mathbb{F}_{1}}(a,0)}(p, L_{\underline{k-2}}, {\lambda}_{\underline{l-1}}^{\rightarrow e}),   
\end{equation}\\
holds where  $p = \ev_e (C)$ is in the intersection of the two multi-line conditions in Lemma \ref{0}, and $e$ is the contracted bounded edge in Lemma \ref{n0}.

\end{enumerate}

   The recursive formula in \cite{G} is the same formula stated above.
\end{lemma}

\begin{proof}

To show that the recursive formula in \cite{G} agrees with the one stated in Theorem \ref{Main0} for tropical stable maps of degree \( \Delta_{\mathbb{F}_{1}}(a,0) \), it is enough to prove that such maps do not have any movable branches \( \mathcal{S}_\sigma \).

Assume otherwise. Let \( C \) be a tropical stable map contributing to \( \mathcal{N}^{0}_{\Delta_{\mathbb{F}_{1}}(a,0)} (p_{\underline{n}}, L_{\underline{k}}, {\lambda}_{\underline{l-1}}, \lambda'_l) \), where \( \lambda'_l \) has a very large length. If \( C \) has a movable branch \( \mathcal{S}_\sigma \), and if \( \sigma = 1 \), then \( \mathcal{S}_\sigma \) contains only one vertex \( v \). According to Definition \ref{DefString}, this vertex is adjacent only to parallel edges in the direction of movement \( (1,0) \). Since the image of a tropical stable map in the plane has no ends in direction \( (1,0) \), \( v \) cannot be balanced, which leads to a contradiction.

If \( \sigma > 1 \), then \( \mathcal{S}_\sigma \) contains at least two vertices, and the only possible directions for the adjacent ends are \( (1,1) \), \( (-1,0) \), and \( (0,-1) \). By part \ref{fifth} of Definition \ref{DefString}, at least one vertex \( v \) in the movable branch must be adjacent to an end in direction \( (1,1) \). By Lemma \ref{4cases}, this vertex \( v \) is a 3-valent vertex of type I, and hence is adjacent to one end and two bounded edges, with one of those edges connecting to a vertex in the fixed component. Since \( \sigma > 1 \), not all edges adjacent to \( v \) can be parallel; otherwise, \( \mathcal{S}_\sigma \) would have a two-dimensional movement. Therefore, the balancing condition at \( v \) implies that the other edges adjacent to \( v \) must lie in the half-planes \( H_v \) and \( H'_v \). This is the situation shown in Figure \ref{HH'}.

\begin{figure}[h!]
\centering
\includegraphics[scale = 0.7]{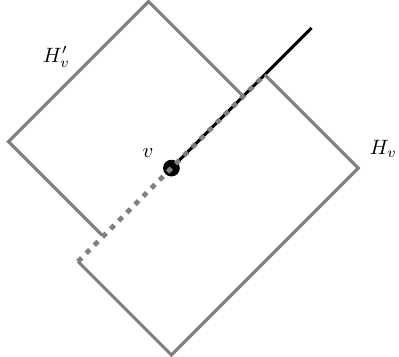}
\caption{In this figure, we see a vertex \( v \) of a movable branch \( \mathcal{S}_\sigma \), with \( \sigma > 1 \), which is connected to an end of primitive direction \( (1,1) \). The half-planes \( H_v \) and \( H'_v \) are 1-ray half-planes.
 }
\label{HH'}
\end{figure}

 Since $v$ is a 3-valent vertex of type I and its direction of movement is $(1,0)$, $v$ is adjacent to a vertex in the fixed component inside \( H'_v \). Inside $H_v$, either $v$ can be adjacent to the only end in $H_v$ or to another vertex in the movable component. But $H_v$ is a 1-ray special half-plane in this case, so $v$ cannot be adjacent to another vertex in the movable component in $H_v$; otherwise this contradicts Lemma \ref{1Ray}. Therefore, \( v \) is a 3-valent vertex adjacent to two ends, which contradicts the assumption that \( \sigma > 1 \).

\end{proof}

As a result of Lemma \ref{recoveingG} and \cite[Corollary 4.7]{G}, in the case \( l = k = 0 \), Kontsevich's formula also follows from the formula presented in Theorem \ref{Main0}.

\begin{example}
Consider the case where, in \eqref{General}, we take \( r = 2 \), \( b = d \), and \( w_i = 1 \) for \( 1 \leq i \leq d \), and assume further that \( l = k = f = 0 \). 
In this setting, we do not mark any non-contracted ends. Thus, we are counting the number of tropical curves of degree \( \Delta_{\mathbb{F}_{2}}(a,b) \) passing through \( n \) points in general position. We obtain the following equality:
\begin{equation}\label{FMKonts}
\mathcal{N}^{0}_{\mathbb{F}_{2}}(a,b)=  \mathcal{N}^{0}_{\Delta_{\mathbb{F}_{2}}(a,b)}(p_{\underline{n}}),
\end{equation}
where the left-hand side of \eqref{FMKonts} denotes the number introduced in \cite{FM}.
\end{example}

\begin{lemma}\label{F2onestring}
    A tropical stable map contributing to \(\mathcal{N}^{0}_{\Delta_{\mathbb{F}_{2}}(a,b)}(p_{\underline{n}})\) can contain only one movable branch with exactly two vertices; see Figure~\ref{F2S}. In other words, in the formula \eqref{NFrw}, we have $\sigma = 2$.
\end{lemma}

\begin{figure}[h!]
\centering
\includegraphics[scale=0.8]{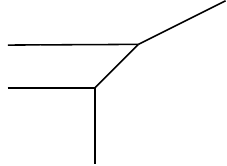}
\caption{The movable branch $\mathcal{S}_2$ for tropical stable maps of degree $\Delta_{\mathbb{F}_{2}}(a,b)$.}
\label{F2S}
\end{figure}

\begin{proof}
Since there is no degenerate cross-ratio condition, all vertices in the movable branch $\mathcal{S}_\sigma$ must be 3-valent. Furthermore, there are no marked right ends, so we cannot have a movable branch $\mathcal{S}_1$ consisting of a single vertex.

If $\mathcal{S}_\sigma$ has more than two vertices, then it must have more than two ends of primitive direction $(-1,0)$, each with a weight $c_i \geq 1$ for $1 \leq i \leq \sigma$. If $\sigma > 2$, then (as illustrated in Figure~\ref{infnite}), the following inequality holds:
\[
\left(\sum_{i=1}^{\sigma} c_i\right) - 2 > 0,
\]
Which implies that $\mathcal{S}_\sigma$ contains an infinite chain of vertices, which is a contradiction.
\end{proof}

\begin{figure}[h!]
\centering
\includegraphics[scale=0.8]{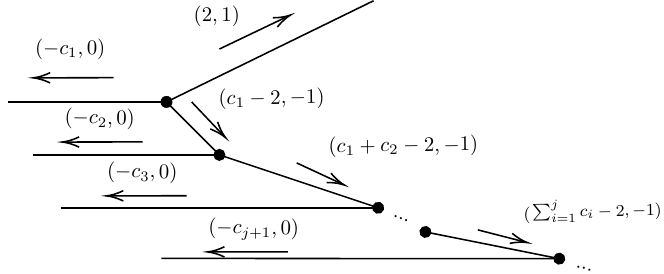}
\caption{A movable branch with an infinite chain of vertices, where $c_i \geq 1$ for each $1 \leq i \leq \sigma$.}
\label{infnite}
\end{figure}

\begin{lemma}[Tropical Kontsevich formula for $\mathbb{F}_2$ \cite{FM}] \label{recoveingFM}

For $ n= \# \Delta_{\mathbb{F}_{2}}(a,b)-1$, the number of \(n\)-marked tropical stable maps of degree \(\Delta_{\mathbb{F}_{2}}(a,b)\), which contribute to  
\(\mathcal{N}^{0}_{\Delta_{\mathbb{F}_{2}}(a,b)}(p_{\underline{n}})\),  
can be computed recursively using the following formula. This formula is derived from Theorem~\ref{Main0} by choosing suitable parameters.

 \begin{equation}\label{NF2}
   \begin{aligned}
   \mathcal{N}^{0}_{\Delta_{\mathbb{F}_{2}}(a,b)}(p_{\underline{n}}) =
    &\sum_{\substack{
        (\Delta_{\mathbb{F}_2}(a_1, b_1) | \Delta_{\mathbb{F}_2}(a_2, b_2)) \\ 
      (a_1, b_1)+(a_2, b_2)=(a,b) }} \Phi_1 \cdot
    \mathcal{N}^{0}_{\Delta_{\mathbb{F}_2}(a_1, b_1)}(p_{\underline{n_1}}) 
    \cdot 
    \mathcal{N}^{0}_{\Delta_{\mathbb{F}_2}(a_2, b_2)}(p_{\underline{n_2}}) 
    \\ 
    &+  \sum_{\substack{
        (\Delta_{\mathbb{F}_2}(a_1, b_1) | \Delta_{\mathbb{F}_2}(a_2, b_2))  \\ 
        (a_1, b_1)+(a_2, b_2)=(a-1,b+2) }}
    \Phi_2 \cdot
   \mathcal{N}^{0}_{\Delta_{\mathbb{F}_2}(a_1, b_1)}(p_{\underline{n_1}}) 
    \cdot 
    \mathcal{N}^{0}_{\Delta_{\mathbb{F}_2}(a_2, b_2)}(p_{\underline{n_2}}), 
    \end{aligned}
    \end{equation}\\
such that the following hold:\\

\begin{equation}
   \begin{aligned}
\Phi_1 &= (2a_1 + b_1)(2a_2 + b_2)(a_1 b_2 + a_2 b_1 +2a_1 a_2)\binom{4a+2b-4}{4a_1 +2b_1 -2}\\ &- (2a_1 + b_1)^2(a_1 b_2 + a_2 b_1 +2a_1 a_2)\binom{4a+2b-4}{4a_1 +2b_1 -1}
  \end{aligned}
\end{equation}

\begin{equation}
\Phi_2 = 2(2a_1 + b_1)(2a_2 + b_2)(b_1 b_2)\binom{4a+2b-4}{4a_1 +2b_1 -2}- 2(2a_1 + b_1)^2(b_1 b_2)\binom{4a+2b-4}{4a_1 +2b_1 -1}
\end{equation}\\

The recursive formula in \cite{FM} is the same as the one given above.   
 
\end{lemma}

\begin{proof}

By Lemma~\ref{F2onestring}, there is only one movable branch $\mathcal{S}_2$ in which all ends have weight 1. Therefore, we have $\sum_{\mathcal{S}_\sigma} \mult(\mathcal{S}_\sigma) = 1$, and $\mathcal{X}$ denotes the total number of ways to glue $C_1$ and $C_2$ along $\mathcal{S}_2$, which in this case is $b_1 b_2$.

To apply Formula~\eqref{NFrw} to this enumerative problem, we need to label all ends in the primitive direction $(1,0)$. Since $b = d$ and $w_i = 1$ for all $1 \leq i \leq d$, there are $b!$ ways to label the right ends of a tropical stable map contributing to $\mathcal{N}^{0}_{\Delta_{\mathbb{F}_{2}}(a,b)}(p_{\underline{n}})$.
By setting $r = \sigma = 2$ in Formula~\eqref{NFrw}, and then applying Remark \ref{Frn}, we obtain Equation~\eqref{NF2}.
\end{proof}

\nocite{*}
\bibliographystyle{IEEEannot}
\bibliography{annot}

\end{document}